\documentclass[11pt,reqno]{amsart}
\usepackage[letterpaper,margin=1in,footskip=0.25in]{geometry}
\usepackage{microtype}
\usepackage[only,llbracket,rrbracket]{stmaryrd}
\usepackage{mathrsfs}
\usepackage{mathtools}
\usepackage[inline]{enumitem}
\usepackage[noadjust]{cite}
\renewcommand{\citepunct}{;\ }

\usepackage{trimclip}
\usepackage{tikz-cd}

\PassOptionsToPackage{pdfusetitle,colorlinks,pagebackref}{hyperref}
\usepackage{bookmark}
\hypersetup{citecolor=[HTML]{2E7E2A},linkcolor=[HTML]{800006},urlcolor=[HTML]{8A0087}}
\usepackage[hyphenbreaks]{breakurl}

\numberwithin{equation}{subsection}
\counterwithout{subsubsection}{subsection}
\makeatletter
\let\c@subsubsection\c@equation

\makeatother

\newtheorem{theorem}[equation]{Theorem}
\newtheorem{conjecture}[equation]{Conjecture}
\newtheorem{corollary}[equation]{Corollary}
\newtheorem{lemma}[equation]{Lemma}
\newtheorem{proposition}[equation]{Proposition}

\newtheorem{innercustomthm}{Theorem}
\newenvironment{customthm}[1]
  {\renewcommand\theinnercustomthm{#1}\innercustomthm}
  {\endinnercustomthm}

\newtheoremstyle{cited}{.5\baselineskip\@plus.2\baselineskip\@minus.2\baselineskip}{.5\baselineskip\@plus.2\baselineskip\@minus.2\baselineskip}{\itshape}{}{\bfseries}{\bfseries .}{5pt plus 1pt minus 1pt}{\thmname{#1}\thmnumber{ #2}\thmnote{ \normalfont#3}}
\theoremstyle{cited}
\newtheorem{citedthm}[equation]{Theorem}
\newtheorem{citedconj}[equation]{Conjecture}
\newtheorem{citedprop}[equation]{Proposition}

\theoremstyle{definition}
\newtheorem{definition}[equation]{Definition}
\newtheorem{convention}[equation]{Convention}
\newtheorem{notation}[equation]{Notation}

\newtheoremstyle{citeddef}{.5\baselineskip\@plus.2\baselineskip\@minus.2\baselineskip}{.5\baselineskip\@plus.2\baselineskip\@minus.2\baselineskip}{}{}{\bfseries}{\bfseries .}{5pt plus 1pt minus 1pt}{\thmname{#1}\thmnumber{ #2}\thmnote{ \normalfont#3}}
\theoremstyle{citeddef}
\newtheorem{citeddef}[equation]{Definition}

\theoremstyle{remark}
\newtheorem{remark}[equation]{Remark}

\newtheoremstyle{step}{.25\baselineskip\@plus.1\baselineskip\@minus.1\baselineskip}{.25\baselineskip\@plus.1\baselineskip\@minus.1\baselineskip}{\itshape}{}{\bfseries}{\bfseries .}{5pt plus 1pt minus 1pt}{\thmname{#1}\thmnumber{ #2}\thmnote{ \normalfont(#3)}}
\theoremstyle{step}
\newtheorem{step}{Step}[equation]

\DeclareMathOperator{\Ann}{Ann}
\DeclareMathOperator{\Ass}{Ass}
\DeclareMathOperator{\SB}{\mathbf{B}}
\DeclareMathOperator{\Bplus}{\mathbf{B}_+}
\DeclareMathOperator{\Bs}{Bs}
\DeclareMathOperator{\Cl}{Cl}
\DeclareMathOperator{\Div}{Div}
\DeclareMathOperator{\Exc}{Exc}
\DeclareMathOperator{\Hom}{Hom}
\DeclareMathOperator{\Pic}{Pic}
\DeclareMathOperator{\Princ}{Princ}
\DeclareMathOperator{\Spec}{Spec}
\DeclareMathOperator{\Supp}{Supp}

\DeclareMathOperator{\WDiv}{WDiv}
\DeclareMathOperator{\WSh}{WSh}
\DeclareMathOperator{\Char}{char}
\DeclareMathOperator{\im}{im}
\DeclareMathOperator{\len}{length}
\DeclareMathOperator{\mult}{mult}
\DeclareMathOperator{\vol}{vol}

\newcommand{\bA}{\mathbf{A}}
\newcommand{\QQ}{\mathbf{Q}}
\newcommand{\RR}{\mathbf{R}}
\newcommand{\ZZ}{\mathbf{Z}}
\newcommand{\bk}{\mathbf{k}}
\newcommand{\cO}{\mathcal{O}}

\newcommand{\fa}{\mathfrak{a}}
\newcommand{\fb}{\mathfrak{b}}
\newcommand{\fm}{\mathfrak{m}}
\newcommand{\fq}{\mathfrak{q}}
\newcommand{\sF}{\mathscr{F}}
\newcommand{\sK}{\mathscr{K}}
\newcommand{\sL}{\mathscr{L}}

\newcommand{\eval}{\textup{eval}}

\newcommand{\jet}{\textup{jet}}
\newcommand{\red}{\textup{red}}

\providecommand\given{}
\newcommand\SetSymbol[1][]{\nonscript\:#1\vert\allowbreak\nonscript\:\mathopen{}}
\DeclarePairedDelimiterX\Set[1]\{\}{\renewcommand\given{\SetSymbol[\delimsize]}#1}

\newcommand{\hooklongrightarrow}{\lhook\joinrel\longrightarrow}
\makeatletter
\newcommand{\longtwoheadrightarrow}{\mathrel{\text{\longtwo@rightarrow}}}
\newcommand{\longtwo@rightarrow}{%
  \sbox0{$\m@th\longrightarrow$}%
  \smash{\rlap{\kern0.175\wd0 \clipbox{{.3\width} {-\height} 0pt {-\height}}{$\m@th\longrightarrow$}}}%
  $\m@th\longrightarrow$%
}
\makeatother

\begin{document}
\title[Moving Seshadri constants and effective Fujita-type conjectures]{Moving
Seshadri constants and\\effective Fujita-type conjectures}
\author{Takumi Murayama}
\address{Department of Mathematics\\Purdue University\\150 N. University
Street\\West Lafayette, IN 47907-2067\\USA}
\email{\href{mailto:murayama@purdue.edu}{murayama@purdue.edu}}
\urladdr{\url{https://www.math.purdue.edu/~murayama/}}

\thanks{This material is based upon work supported by the National Science
Foundation under Grant No.\ DMS-2201251}

\makeatletter
  \hypersetup{
    pdfsubject=\@subjclass,
    pdfkeywords=\@keywords
  }
\makeatother

\begin{abstract}
  Fujita's conjecture is known to be false in positive characteristic.
  We conjecture and give an approach to
  a new variant of Fujita's conjecture for the basepoint-freeness, very
  ampleness, and jet ampleness of linear systems of the form \(\lvert
  K_X+aK_X+bL \rvert\).
  We extend the theory of moving Seshadri constants, previously established
  for smooth
  complex projective varieties by Ein, Lazarsfeld, Musta\c{t}\u{a},
  Nakamaye, and Popa, to the more general setting of complete varieties
  over arbitrary fields.
  This theory is both an important component of our approach to this new 
  conjecture and of independent interest.
  Using our approach,
  we prove our variant of Fujita's conjecture
  for smooth surfaces in arbitrary characteristic and for
  smooth complex projective varieties of arbitrary dimension.
\end{abstract}

\maketitle
\setcounter{tocdepth}{1}
{\hypersetup{hidelinks}\tableofcontents}

\section{Introduction}\label{sect:intro}
Throughout this introduction, let \(X\) be a projective variety over an
algebraically closed field \(k\).
\subsection{Background}
In the 1980s, Fujita \cite{Fuj87,Fuj88} made the following conjecture.
\begin{citedconj}[{\cite[Conjecture 1]{Fuj88} (see also \cite[Conjecture on p.\
  167]{Fuj87})}]\label{conj:fujita}
  Let \(X\) be a smooth projective variety of dimension \(n\) over an
  algebraically closed field \(k\) and let \(L\) be an ample divisor on
  \(X\).
  We then have the following:
  \begin{enumerate}[label=\((\roman*)\)]
    \item \textnormal{(Fujita's freeness conjecture)} \(\bigl\lvert K_X+(n+1)L
      \bigr\rvert\) is basepoint-free.
    \item \textnormal{(Fujita's very ampleness conjecture)} \(\bigl\lvert K_X+(n+2)L
      \bigr\rvert\) is very ample.
  \end{enumerate}
\end{citedconj}
When \(\Char(k) = 0\), Fujita's freeness conjecture is known in dimensions \(\le 5\)
\citeleft\citen{Rei88}\citepunct \citen{EL93a}\citepunct \citen{Fuj93}\citepunct
\citen{Kaw97}\citepunct \citen{YZ20}\citeright\
and Fujita's very ampleness conjecture
is known in dimensions \(\le2\) \cite{Rei88}.
Moreover, it is known that \(\lvert K_X+f(n)\mkern1muL \rvert\) is basepoint-free
for a function \(f(n)\) of order \(n^2\) \cite{AS95,Hel97,Hel99},
of order \(n^{4/3}\) \cite{Hei02}, or of order \(n\cdot\log(\log(n))\) \cite{GL24} in
\(n\).\medskip
\par Without characteristic assumptions on \(k\), Fujita's conjectures are known
when \(n = 1\) and when \(L\) is both ample and globally generated
\cite{EL93syz,Smi97,Kee08}, for example, when \(X\) is toric \cite{Mus02}.
For surfaces (\(n=2\)), Fujita's freeness conjecture is known when \(X\) is
not of general type \cite{Eke88,SB91,Che21,Zha23}.
For surfaces of general type, it is
known that \(\lvert 2K_X+4L \rvert\) is basepoint-free when
\(\Char(k) \ge 3\) and that \(\lvert 2K_X+19L \rvert\) is basepoint-free
when \(\Char(k) = 2\) \cite{DCF15}.
\subsection{A Fujita-type conjecture in arbitrary
characteristic}\label{sect:introconj}
Despite these positive results towards Fujita's conjectures in positive
characteristic, Gu, Zhang, and Zhang \cite{GZZ22} constructed
counterexamples to Fujita's conjectures for surfaces in every characteristic
\(p > 0\).
Motivated by these counterexamples and the existing results for surfaces, we
conjecture new variants of Fujita's Conjecture \ref{conj:fujita}.\medskip
\par Before stating our conjecture, 
we define a higher-order version of basepoint-freeness and very ampleness as
follows.
\begin{citeddef}[{\cite[pp.\ 357--358]{BS93}}]
  Let \(X\) be a complete variety of dimension \(n\) over an arbitrary
  field \(k\).
  Let \(L\) be a \(\ZZ\)-invertible sheaf on \(X\).
  Let \(Z \coloneqq \{x_1,x_2,\ldots,x_r\}\) be a finite set of distinct closed
  points in \(X\).
  We say that \(L\) is \textsl{\(\ell\)-jet ample at \(Z\)} if for all
  non-negative integers \(\ell_1,\ell_2,\ldots,\ell_r\) such that
  \(\ell+1=\sum_{i=1}^r\ell_i\), the
  restriction map
  \[
    H^0\bigl(X,\cO_X(L)\bigr) \longrightarrow H^0\biggl(X,\cO_X(L)
    \otimes_{\cO_X} \cO_X\bigg/\prod_{i=1}^r \fm_{x_i}^{\ell_i}\biggr)
  \]
  is surjective.
  We say that \(L\) is \textsl{\(\ell\)-jet ample} if it is \(\ell\)-jet ample
  at every finite set \(Z\) of distinct closed points in \(X\).
  We use the same terminology for linear systems.
\end{citeddef}
Note that \(\lvert L \rvert\) is basepoint-free if and only if it is \(0\)-jet
ample and, when \(k\) is algebraically closed, \(\lvert L \rvert\)
is very ample if and only if it is \(1\)-jet ample \cite[Chapter II,
Proposition 7.3]{Har77}.\medskip
\par With this terminology, we conjecture the following variants of Fujita's
Conjecture \ref{conj:fujita} in arbitrary characteristic.
\begin{conjecture}\label{conj:fujitavariant}
  Let \(X\) be a smooth projective variety of dimension \(n\)
  over an algebraically closed field \(k\).
  Let \(L\) be an ample divisor on \(X\).
  Let \(\ell \ge 0\) be an integer.
  \begin{enumerate}[label=\((\textup{\Roman*})\),ref=\textup{\Roman*}]
    \item\label{ques:anbn} There exist integer-valued polynomial
      functions \(a_\ell(n)\) and \(b_\ell(n)\)
      only depending on \(n\) such that
      \[
        \bigl\lvert K_X+a_\ell(n)\mkern1muK_X+b_\ell(n)\mkern1muL \bigr\rvert
      \]
      is \(\ell\)-jet ample.
    \item\label{ques:anbns0}
      The functions \(a_\ell(n)\) and \(b_\ell(n)\) in \((\ref{ques:anbn})\) can
      be chosen such that if \(\Char(k) > 0\), the canonical linear system
      defined by Schwede \textnormal{\cite{Sch14}}
      \begin{equation}\label{eq:stablesectionsintro}
        \biggl\lvert S^0\Bigl(X,\omega_X \otimes_{\cO_X}
        \cO_X\bigl(a_\ell(n)\mkern1muK_X+b_\ell(n)\mkern1muL\bigr)\Bigr) \biggr\rvert
      \end{equation}
      is \(\ell\)-jet ample.
  \end{enumerate}
\end{conjecture}
\par Schwede's canonical linear system in \((\ref{ques:anbns0})\) is important
because global sections in \(S^0(X,\omega_X \otimes_{\cO_X} \cO_X(\,\cdot\,))\)
lift from subschemes as though the Kodaira vanishing
theorem were true, and hence are amenable to inductive proofs
\cite[Proposition 5.3]{Sch14}.
Producing sections in \(S^0(X,\omega_X \otimes_{\cO_X} \cO_X(\,\cdot\,))\)
is crucial for the
minimal model program in positive characteristic, for example in
\cite{HX15}.\medskip
\par When \(\Char(k) = 0\), Conjecture
\ref{conj:fujitavariant}\((\ref{ques:anbn})\) is known in some cases.
For surfaces, this is proved in \citeleft\citen{Rei88}\citepunct
\citen{BS93}\citepunct \citen{Laz97}\citemid Theorem 7.4\citeright.
For arbitrary \(n\), 
explicit functions \(a_\ell(n)\) and \(b_\ell(n)\)
were shown to exist in \cite{Dem93,Kol93,Siu96,Siu95,Dem96}, which also prove that
one can even set \(a_1(n) = 1\).
However, the known functions \(a_\ell(n)\) and \(b_\ell(n)\) are not polynomial
in \(n\) unless \(\ell \le 1\).
When \(\ell \ge 2\), the direct analogue of Fujita's conjecture
asking whether \(\ell\)-jet ampleness holds for \(a_\ell(n) = 0\) and
\(b_\ell(n) = n+\ell+1\) is known to be false \cite[Remark 1.7]{EKL95}.
As far as we are aware, there is no conjecture in the literature with explicit
conjectures for what \(a_\ell(n)\) and \(b_\ell(n)\) might look like in general,
even when \(\Char(k) = 0\).\medskip
\par When \(\Char(k) = p> 0\), much less is known.
Existing results on Fujita's conjecture for curves answer \((\ref{ques:anbn})\)
and Schwede answered \((\ref{ques:anbns0})\) for curves in \cite[Theorem
3.3]{Sch14}.
Schwede \cite[p.\ 71]{Sch14} then asked whether \((\ref{ques:anbns0})\)
holds for \(a_0(n) = a_1(n) = 0\), \(b_0(n) = n+1\), and \(b_1(n) = n+2\).
While Schwede's question was answered in the negative by \cite{GZZ22},
question \((\ref{ques:anbn})\) for arbitrary \(\ell\) still holds
for surfaces that are not of general type
\citeleft\citen{Ter99}\citepunct \citen{Zha23}\citeright.
As far as we are aware, when \(\Char(k) = p> 0\), 
there are no results in the literature for
\begin{enumerate*}[label=\((\arabic*)\)]
  \item surfaces of general type when \(\ell > 0\), or
  \item projective varieties of arbitrary dimension for any \(\ell\).
\end{enumerate*}

\subsection{Main results}
In this paper, we give an approach to Conjecture \ref{conj:fujitavariant} using
moving Seshadri constants.
To do so, we extend the theory of moving Seshadri constants to complete
varieties over arbitrary fields.
Previously, the theory has only been developed for smooth complex projective
varieties, in which case the theory is due to Ein, Lazarsfeld, Musta\c{t}\u{a},
Nakamaye, and Popa \cite{ELMNP09}.
We then use our approach to prove Conjecture \ref{conj:fujitavariant} for
smooth surfaces in arbitrary characteristic and for smooth complex
projective varieties of arbitrary dimension.
\subsubsection{Moving Seshadri constants}
The important ingredient in our approach to Conjecture \ref{conj:fujitavariant}
is the following result, which says that lower bounds on moving Seshadri
constants imply \(\ell\)-jet ampleness of adjoint-type divisors.
Proving this result requires extending theory of moving Seshadri constants to
complete varieties over arbitrary fields.
See Definition \ref{def:movingsesh} for the definition of moving Seshadri
constants, which are a generalization of Seshadri constants that make sense for
divisors that are not necessarily nef.
Theorem \ref{thm:introseshjetample}
generalizes results for big and nef divisors on smooth projective varieties
\cite{Dem92,ELMNP09,MS14,Mur18}.
\par The singular case of Theorem \ref{thm:introseshjetample} is
new even over the complex numbers.
When \(\ell > 0\), Theorem \ref{thm:introseshjetample} is new even for
\emph{smooth} complex projective varieties.
Existing results for smooth projective varieties such as
\citeleft\citen{BS93}\citemid Corollary 3.3\citepunct \citen{Lan99}\citemid
Proposition 4.1(1) and Corollary 4.3\citepunct \citen{Laz04a}\citemid
Proposition 5.1.19\citepunct \citen{MS14}\citemid Theorem 3.1\((iv)\)\citepunct
\citen{SZ20}\citemid Theorem 3.13\citeright\ impose
stronger conditions on \(L\): Either \(L\) is ample and globally generated, in
which case \(K_X+(n+\ell)L+A\) is \(\ell\)-jet ample for every ample divisor
\(A\), or \(L\) is a big and nef divisor such that \(\varepsilon(L;x) > 2n\), in
which case \(K_X+L\) is very ample.
Even when \(X\) is smooth,
we do not know of a characteristic zero proof of Theorem
\ref{thm:introseshjetample} that does not use reduction modulo
\(p\) techniques.
\begin{customthm}{\ref{thm:introseshjetample}}
  Let \(X\) be a normal projective variety of dimension \(n\) over an arbitrary
  field \(k\) and let \(L\) be a Cartier divisor on \(X\).
  Let \(\ell\) be a non-negative integer.
      Let \(Z = \{x_1,x_2,\ldots,x_r\}\) be closed points in \(X\) such that \(X\)
      is of dense \(F\)-injective type (if \(\Char(k) = 0\)) or \(F\)-injective (if
      \(\Char(k) = p > 0\)) at every \(x_i\).
      Suppose that
      \[
        \varepsilon\bigl(\lVert L \rVert;x_i\bigr) > n+\ell
      \]
      for every \(i\).
      Then, \(\omega_X \otimes_{\cO_X}
      \cO_X(L)\) is \(\ell\)-jet ample at \(Z\).
  Moreover, if \(k\) is an \(F\)-finite field of characteristic \(p > 0\),
  then the linear system
  \[
    \Bigl\lvert S^0\bigl(X,\omega_X \otimes_{\cO_X} \cO_X(L)\bigr) \Bigr\rvert
  \]
  is \(\ell\)-jet ample at \(Z\).
\end{customthm}
In characteristic zero, rational
singularities are of dense \(F\)-injective type
\citeleft\citen{Har98}\citemid Theorem 1.1\citepunct \citen{MS97}\citemid
Theorem 1.1\citeright.
In positive characteristic,
\(F\)-injectivity is the analogue of Du Bois singularities
\cite[Theorem 6.1]{Sch09DB}.\medskip
\par To prove Theorem \ref{thm:introseshjetample}, we extend the theory of
moving Seshadri constants to complete varieties over arbitrary fields.
We combine this theory with the
theory of Frobenius--Seshadri constants introduced by Musta\c{t}\u{a} and
Schwede \cite{MS14} and further developed by the author of this paper in
previous work \cite{Mur18}.
\subsubsection{Our approach to Conjecture
\ref{conj:fujitavariant}}\label{subsect:ourapproach}
Theorem \ref{thm:introseshjetample} says that to show Conjecture
\ref{conj:fujitavariant}, it suffices to show the following:
\begin{conjecture}\label{conj:seshadjlower}
  Let \(X\) be a smooth projective variety of dimension \(n\)
  over an algebraically closed field
  \(k\) and let \(L\) be an ample divisor.
  There exist integer-valued
  functions \(c(n)\) and \(d(n)\) and a real-valued function \(e(n)\), all
  depending only on \(n\), such that
  \[
    \varepsilon\bigl(\bigl\lVert K_X+c(n)\mkern1muK_X+d(n)\mkern1muL
    \bigr\rVert;x\bigr)
    \ge \frac{1}{e(n)}
  \]
  for every closed point \(x \in X\).
\end{conjecture}
We prove Conjecture \ref{conj:seshadjlower} for surfaces (\(n=2\)).
Over the complex numbers, this statement is due to Bauer and Szemberg
\cite[Corollary 4.2]{BS11}.
For the statement below, note that for ample divisors, the moving Seshadri
constant \(\varepsilon(\lVert L \rVert;x)\) is equal to the Seshadri constant
\(\varepsilon(L;x)\) (Proposition
\ref{prop:elmnp0963}\((\ref{prop:elmnp0963iv})\)).
\begin{customthm}{\ref{cor:seshlowerboundadj}}
  Let \(X\) be a smooth projective surface over an algebraically closed field and
  let \(L\) be a nef divisor on \(X\) such that \(K_X+L\) is ample.
  Then,
  \[
    \varepsilon(K_X+L;x) \ge \frac{1}{2}
  \]
  for every closed point \(x \in X\).
  In particular, for any ample divisor \(L\) on \(X\), we have
  \[
    \varepsilon(K_X+4L;x) \ge \frac{1}{2}.
  \]
\end{customthm}
When \(\Char(k) = 0\),
Bauer and Szemberg \cite[\S3]{BS11} answered Conjecture \ref{conj:seshadjlower}
in the affirmative using existing results on Fujita's freeness conjecture.
We improve their bounds using the progress on 
Fujita's freeness conjecture in \cite{Rei88,Kaw97,YZ20,GL24}.
\begin{customthm}{\ref{thm:lowerboundsoverc}}
  Let \(X\) be a smooth complex projective variety of dimension \(n\).
  Let \(L\) be a nef divisor on \(X\) such that \(K_X+L\) is ample.
  For every closed point \(x \in X\), we have
  \[
    \varepsilon(K_X+L;x) \ge \begin{cases}
      \hfil \dfrac{1}{n+2} & \text{when}\ n \le 5,\\[1em]
      \hfil \dfrac{1}{9} & \text{when}\ n = 6,\\[1em]
      \dfrac{1}{\Bigl\lfloor
      n\Bigl(\log\bigl(\log(n)\bigr)+2.34\Bigr)
      \Bigr\rfloor + 2}
      & \text{for arbitrary}\ n.
    \end{cases}
  \]
\end{customthm}

\subsubsection{Conjecture \ref{conj:fujitavariant} for
smooth surfaces in arbitrary characteristic and for smooth complex projective
varieties}
\par Using Theorems \ref{thm:introseshjetample},
we prove our new variant of Fujita's conjecture \ref{conj:fujitavariant} for
smooth surfaces in arbitrary characteristic and for smooth complex projective
varieties.
For surfaces, we prove:
\begin{customthm}{\ref{thm:smooths0ljetample}}
  Let \(X\) be a smooth projective surface over an algebraically closed field
  \(k\) and let \(L\) be an ample divisor on \(X\).
  For every integer \(\ell \ge 0\), the linear system
  \[
    \bigl\lvert K_X+(4+2\ell)K_X+(17+8\ell)L \bigr\rvert
  \]
  is \(\ell\)-jet ample.
  If \(\Char(k) = p > 0\), then in fact, the linear system
  \[
    \biggl\lvert S^0\Bigl(X,\omega_X \otimes_{\cO_X}
    \cO_X\bigl((4+2\ell)K_X+(17+8\ell)L\bigr)\Bigr) \biggr\rvert
  \]
  is \(\ell\)-jet ample.
\end{customthm}
\par Theorem \ref{thm:smooths0ljetample} suggests that the
functions \(a_\ell(n)\) and \(b_\ell(n)\) in Conjecture \ref{conj:fujitavariant}
may be characteristic independent for arbitrary \(n\).\medskip
\par Over the complex numbers, we prove the following result.
Previously, with notation as below,
Demailly proved \(\ell\)-jet ampleness of divisors of the form
\(2K_X+mL\) where \(m\) is exponential in \(n\) \cite[Theorem 0.2]{Dem96}.
\begin{customthm}{\ref{thm:fujitavariantcomplex}}
  Let \(X\) be a smooth complex projective variety of dimension \(n\).
  Fix an integer \(\ell \ge 0\).
  For all ample divisors \(L\) and \(A\) and every integer
  \[
    m \ge \begin{cases}
      \hfil n+2 & \text{when}\ n \le 5,\\
      \hfil 8 & \text{when}\ n = 6,\\
      n\Bigl(\log\bigl(\log(n)\bigr)+2.34\Bigr)
      & \text{for arbitrary}\ n,
    \end{cases}
  \]
  the linear system
  \[
    \bigl\lvert K_X+(n+\ell)(K_X+mL)+A\bigr\rvert
  \]
  is \(\ell\)-jet ample.
  In particular, the linear system
  \[
    \biggl\lvert
    K_X+(n+\ell)\Bigl(K_X+n\Bigl(\log\bigl(\log(n)\bigr)+2.34\Bigr)L\Bigr)+L
    \biggr\rvert
  \]
  is \(\ell\)-jet ample.
\end{customthm}
\begin{remark}
  An alternative approach to Conjecture \ref{conj:fujitavariant} for surfaces
  is to combine the
  results in \cite{Eke88,SB91,DCF15,Che21}
  with \cite[Theorem 3.13]{SZ20} (and its
  proof) to show that
  \[
    \biggl\lvert S^0\Bigl(X,\omega_X \otimes_{\cO_X}
    \cO_X\bigl((4+2\ell)K_X+(39+19\ell)L\bigr)\Bigr) \biggr\rvert
  \]
  is \(\ell\)-jet ample.
  However, such an approach uses the classification of surfaces.
  Our approach does not rely on the classification of surfaces.
\end{remark}
\subsection*{Outline}
We review some preliminaries on \(\bk\)-divisors, \(\bk\)-invertible sheaves,
and base loci in \S\ref{sect:prelim}.
An important new result is Theorem \ref{prop:kur13prop27}, which describes
how \(\Bplus(D)\) is the locus where \(D\) is ample.
\par In \S\ref{sect:seshadri}, we extend the theory of moving Seshadri constants to
complete varieties over arbitrary fields.
A major new result is that moving Seshadri constants can be described in terms
of jet separation on normal complete varieties over arbitrary fields (Theorem
\ref{thm:seshjet}).
We then prove our new results for Frobenius--Seshadri constants in
\S\ref{sect:frobsesh}.
In particular, we prove Theorem
\ref{thm:introseshjetample} by first comparing Seshadri constants and
Frobenius--Seshadri constants (Proposition \ref{prop:frobseshcomp}) and then
proving the analogue of Theorem \ref{thm:introseshjetample} for
Frobenius--Seshadri constants (Theorem \ref{thm:seshgenjets}).
\par In \S\ref{sect:lowerbounds}, we prove our lower
bound for Seshadri constants of adjoint-type divisors on surfaces (Corollary
\ref{cor:seshlowerboundadj}) and on higher-dimensional varieties over the complex
numbers (Theorem \ref{thm:lowerboundsoverc}).
Finally, we prove Theorems \ref{thm:smooths0ljetample} and
\ref{thm:fujitavariantcomplex} in \S\ref{sect:effective}.
\subsection*{Notation}
All rings are commutative with identity, and all ring maps are unital.
The notation \(\ZZ_{(p)}\) denotes the localization of \(\ZZ\) at the prime
ideal \((p)\) generated by a prime number \(p > 0\).
\par A \textsl{variety} is a separated scheme of finite type over a
field \(k\).
A variety or scheme is \textsl{complete} (resp.\ projective) if it is proper
(resp.\ projective) over a field \(k\).
A \textsl{curve} or a \textsl{surface} is a complete variety
of dimension \(1\) or \(2\), respectively.
\par For a scheme \(X\) of prime characteristic \(p > 0\), we denote by
\(F\colon X \to X\) the \textsl{(absolute) Frobenius morphism}, which is given by
the identity map on points and the \(p\)-th power map
\[
  \begin{tikzcd}[row sep=0,column sep=1.475em]
    \cO_X(U) \rar & \cO_{X}(U)\\
    f \rar[mapsto] & f^{p}
  \end{tikzcd}
\]
on local sections of the structure sheaf for every open subset \(U \subseteq X\).
If \(R\) is a ring of prime characteristic \(p > 0\), we denote the corresponding ring
map by \(F\colon R \to F_{*}R\).
For every integer \(e \ge 0\), we denote
the \(e\)-th iterate of the Frobenius morphisms for
schemes or rings by \(F^e\).
\par We say that a locally Noetherian scheme \(X\) or a Noetherian ring
\(R\) of prime characteristic \(p > 0\) is
\textsl{\(F\)-finite} if the Frobenius morphism \(F\) is finite.
In this case, \(X\) and \(R\) are excellent \cite[Theorem 2.5]{Kun76}
and \(R\) has a dualizing complex \(\omega_R^\bullet\)
\cite[Remark 13.6]{Gab04}.
\subsection*{Acknowledgments}
We thank Junyao Peng for pointing out an improvement to Theorem
\ref{thm:seshgenjets}.
We are grateful to Farrah Yhee for helpful conversations.

\section{\texorpdfstring{\(\bk\)}{k}-divisors,
\texorpdfstring{\(\bk\)}{k}-invertible sheaves, and base loci}\label{sect:prelim}
We collect some preliminaries on \(\bk\)-divisors, \(\bk\)-invertible sheaves,
and base loci.
See \cite[Part I]{LM} for more discussion.
An important new result is Theorem \ref{prop:kur13prop27}, which describes
how \(\Bplus(D)\) is the locus where \(D\) is ample.
\subsection{\texorpdfstring{\(\bk\)}{k}-divisors and
\texorpdfstring{\(\bk\)}{k}-invertible sheaves}
Let \(X\) be a ringed space.
We define the Picard group as in \cite[Chapitre 0, (5.6.2)]{EGAInew} and the
group \(\Div(X)\) of Cartier divisors as in \cite[D\'efinition 21.1.2]{EGAIV4}
(with the corrected definition for the sheaf \(\sK_X\) of meromorphic functions
from \cite[p.\ 204]{Kle79}).
We define the subgroup of principal Cartier divisors \(\Princ(X) \subseteq
\Div(X)\) as in \cite[(21.3.1)]{EGAIV4} (where the notation
\(\operatorname{Div.princ}(X)\) is used).
\par Now let \(X\) be a locally Noetherian scheme.
We define the group \(\WDiv(X)\) of Weil divisors as in \cite[(21.6.2)]{EGAIV4}
(where the notation \(\mathfrak{Z}^1(X)\) is used).
If \(X\) is \(G_1\) (i.e., Gorenstein in codimension 1 \cite[p.\ 142]{RF70}) and
\(S_2\), we define the group of AC divisors \(\WSh(X)\) as in \cite[Definition
on p.\ 301]{Har94} (where the notation \(\operatorname{ACart}(X)\) is
used).\medskip
\par Below, we have stated everything in terms of a coefficient ring
\(\bk\) contained in \(\RR\).
The examples to keep in mind are \(\ZZ\), \(\ZZ_{(p)}\), \(\QQ\), and \(\RR\).
\begin{definition}[see {\citeleft\citen{EGAIV4}\citemid (21.10.9)\citepunct
  \citen{Laz04a}\citemid Definition 1.3.1 and
  Remark 1.3.8\citepunct \citen{MS12}\citemid pp.\ 153--154\citepunct
  \citen{FM}\citemid Definition 1.1\citeright}]
  Let \(\bk \subseteq \RR\) be a subring.
  Let \(X\) be a ringed space.
  A \textsl{\(\bk\)-invertible sheaf} is an element of 
  \begin{align*}
    \Pic_\bk(X) &\coloneqq \Pic(X) \otimes_\ZZ \bk
    \intertext{and a \textsl{\(\bk\)-Cartier divisor} is an element of}
    \Div_\bk(X) &\coloneqq \Div(X) \otimes_\ZZ \bk.
    \intertext{If \(X\) is a locally Noetherian scheme, a \textsl{\(\bk\)-Weil
    divisor} is an element of}
    \WDiv_\bk(X) &\coloneqq \WDiv(X) \otimes_\ZZ \bk.
    \intertext{If \(X\) is a Noetherian scheme satisfying \(G_1\) and \(S_2\), a
    \textsl{\(\bk\)-AC divisor} is an element of}
    \WSh_\bk(X) &\coloneqq \WSh(X) \otimes_\ZZ \bk.
  \end{align*}
  We write these groups additively.
  A \(\bk\)-invertible sheaf, \(\bk\)-Cartier divisor, \(\bk\)-Weil divisor, or
  \(\bk\)-AC divisor 
  is \textsl{integral} if
  it lies in the image of \(\Pic(X)\), \(\Div(X)\), \(\WDiv(X)\), or
  \(\WSh(X)\), respectively.
\end{definition}
\begin{convention}
  When we say ``\(\ZZ\)-invertible sheaf,'' we write the elements of \(\Pic(X)\)
  additively, and when we say ``invertible sheaf,'' we write the elements of
  \(\Pic(X)\) multiplicatively.
  If \(D\) is a \(\ZZ\)-invertible sheaf or a \(\ZZ\)-Cartier divisor, we denote
  by \(\cO_X(D)\) the associated invertible sheaf.
\end{convention}
Setting \(\Princ_{\bk}(X) \coloneqq \Princ(X) \otimes_\ZZ \bk\), we have the
exact sequence
\[
  0 \longrightarrow \Princ_{\bk}(X) \longrightarrow \Div_\bk(X)
  \xrightarrow{l \otimes_\ZZ \bk} \Pic_\bk(X)
\]
by \cite[Proposition 21.3.3\((i)\)]{EGAIV4},
where \(l \otimes_\ZZ \bk\) is surjective if \(\Ass(\cO_X)\) is contained
in an open affine subscheme of \(X\) (e.g., \(X\) is quasi-projective over a
Noetherian ring) or if \(X\) is a reduced scheme whose set of irreducible
components is locally finite
\cite[Proposition 21.3.4 and Corollaire 21.3.5]{EGAIV4}.
If \(X\) is a locally Noetherian scheme,
this exact sequence fits into the commutative
diagram
\[
  \begin{tikzcd}
    0 \rar & \Princ_{\bk}(X) \rar\dar[equal] &
    \Div_\bk(X) \rar{l \otimes_\ZZ \bk}\dar{\textup{cyc} \otimes_\ZZ \bk}
    & \Pic_\bk(X)\dar\\
    & \Princ_{\bk}(X) \rar &
    \WDiv_\bk(X) \rar & \Cl_\bk(X) \rar & 0
  \end{tikzcd}
\]
where \(\textup{cyc}\) is the cycle map from \cite[(21.6.5.1)]{EGAIV4} and
\(\Cl_\bk(X) \coloneqq \Cl(X) \otimes_\ZZ \bk\) for the divisor class group
\(\Cl(X)\) as defined in \cite[(21.6.7)]{EGAIV4}.
Note that \(\textup{cyc} \otimes_\ZZ \bk\) is injective if \(X\) is normal
\cite[Th\'eor\`eme 21.6.9\((i)\)]{EGAIV4}.
\begin{definition}[{see \cite[Remarks 1.1.4 and 1.3.8]{Laz04a}}]
  Let \(\bk \subseteq \RR\) be a subring.
  Let \(X\) be a normal locally Noetherian scheme.
  We say that a \(\bk\)-Weil divisor on \(X\) is \textsl{\(\bk\)-Cartier}
  if it lies in the image of the injective map
  \[
    \textup{cyc}
    \otimes_\ZZ \bk\colon \Div_\bk(X) \hooklongrightarrow \WDiv_\bk(X).
  \]
\end{definition}
\par 
If \(X\) is a Noetherian scheme satisfying \(G_1\) and \(S_2\), the cycle map
factors as
\[
  \Div_\bk(X) \hooklongrightarrow \WSh_\bk(X) \longrightarrow \WDiv_\bk(X)
\]
by the construction of \(\textup{cyc}\) (see also the proof of
\cite[Proposition 2.4]{Har94}),
where the first map is injective by \cite[Corollary 2.6]{Har94}.
\begin{definition}[see {\cite[Definition 1.3.3\((iii)\)]{Laz04a}}]
  Let \(\bk \subseteq \RR\) be a subring.
  Let \(X\) be a ringed space (resp.\ a locally Noetherian scheme, a locally
  Noetherian scheme satisfying \(G_1\) and \(S_2\)).
  We say that two elements \(D,E\) in \(\Div_\bk(X)\) (resp.\ \(\WDiv_\bk(X)\),
  \(\WSh_\bk(X)\)) are
  \textsl{\(\bk\)-linearly equivalent} and write \(D \sim_\bk E\)
  if their difference lies in the image of
  \(\Princ_\bk(X)\).
\end{definition}
To define \(\bk\)-numerical equivalence, we recall that for complete schemes
\(X\), we can define intersection numbers
\[
  \bigl(\sL_1 \cdot \sL_2 \cdots \sL_m \cdot Y \bigr) \in \ZZ
\]
for a closed subscheme \(Y \subseteq X\) and invertible sheaves
\(\sL_1,\sL_2,\ldots,\sL_m\) on \(X\) where \(m \ge \dim(Y)\).
See \cite[Definition B.8]{Kle05}.
By linearity \cite[Theorem B.9(2)]{Kle05}, we can extend this definition to
\(\bk\)-invertible sheaves.
We can also extend this definition to \(\bk\)-Cartier divisors using the map
\(l \otimes_\ZZ \bk\).
In this case, we denote the intersection product by
\[
  \bigl(D_1 \cdot D_2 \cdots D_m \cdot Y \bigr) \in \bk.
\]
\begin{citeddef}[{\citeleft\citen{Kle66}\citemid p.\ 303\citepunct
  \citen{Kee03}\citemid Definition 2.9\citeright}]
  Let \(\bk \subseteq \RR\) be a subring.
  Let \(X\) be a complete scheme.
  A \(\bk\)-invertible sheaf \(D\) is \textsl{nef} if \((D \cdot C)\ge 0\) for
  every integral closed curve \(C \subseteq X\).
  A \(\bk\)-invertible sheaf \(D\) is \textsl{numerically trivial} if both \(D\)
  and \(-D\) are nef.
  We say that two \(\bk\)-invertible sheaves \(D,E\) are
  \textsl{\(\bk\)-numerically trivial} if \(D \equiv_\bk E\).
  We use the same notation and terminology for \(\bk\)-Cartier divisors using
  the map \(l \otimes_\ZZ \bk\).
  We set
  \[
    N^1_\bk(X) \coloneqq \Pic_\bk(X)/{\equiv_\bk}.
  \]
\end{citeddef}
By the theorem of the base \citeleft\citen{Kle66}\citemid Chapter IV, \S1,
Proposition 4\citepunct \citen{Kee03}\citemid Theorem 3.6\citepunct
\citen{Kee18}\citemid Theorem E2.2\citeright, the \(\bk\)-module \(N^1_\bk(X)\)
is finitely generated as a \(\bk\)-module.\medskip
\par We use the following conventions for rounding up and down.
\begin{citeddef}[{\cite[Definition 3.4.1]{BGGJKM} (cf.\
  \citeleft\citen{Bir17}\citemid Definition 1.2\citepunct
  \citen{Laz04b}\citemid Definition 9.1.2\citeright)}]
  \label{def:roundup}
  Let \(\bk \subseteq \RR\) be a subring.
  Let \(X\) be a ringed space (resp.\ a Noetherian scheme satisfying \(G_1\) and
  \(S_2\)), and let \(D \in \Div_\bk(X)\)
  (resp.\ \(\WSh_\bk(X)\)).
  A \textsl{decomposition} \(\mathcal{D}\) of \(D\) is an expression
  \[
    D = \sum_{i=1}^r a_iD_i
  \]
  for some \(a_i \in \bk\) and Cartier divisors (resp.\ AC divisors)
  \(D_i\).
  We note that such a decomposition is not unique because \(\Div(X)\) (resp.\
  \(\WSh(X)\)) may have
  torsion.
  The \textsl{round-up} and \textsl{round-down} of \(D\)
  with respect to
  \(\mathcal{D}\) are the Cartier divisors (resp.\ AC divisors)
  \[
    \lceil D \rceil_{\mathcal{D}} \coloneqq \sum_{i=1}^r \lceil a_i \rceil
    D_i \qquad \text{and} \qquad
    \lfloor D \rfloor_{\mathcal{D}} \coloneqq \sum_{i=1}^r \lfloor a_i \rfloor
    D_i,
  \]
  respectively.
  The round-up and round-down depend on the decomposition
  \(\mathcal{D}\).
  \par Now let \(X\) be a locally Noetherian scheme and let \(D \in
  \WDiv_\bk(X)\).
  Since \(\WDiv(X)\) is torsion-free \cite[(21.10.9)]{EGAIV4},
  there is a unique expression
  \[
    D = \sum_{i=1}^r a_iD_i.
  \]
  Then, the \textsl{round-up} and \textsl{round-down} of \(D\) are
  the Weil divisors
  \[
    \lceil D \rceil \coloneqq \sum_{i=1}^r \lceil a_i \rceil
    D_i \qquad \text{and} \qquad
    \lfloor D \rfloor \coloneqq \sum_{i=1}^r \lfloor a_i \rfloor
    D_i,
  \]
  respectively.
\end{citeddef}

\subsection{Base ideals, base loci, and stable base loci}
We recall the definitions of base ideals, base loci, and stable base loci.
We start with the base ideal.
\begin{citeddef}[{\cite[Definition 2.1]{Kur13}}]
  Let \(X\) be a ringed space and let \(\sL\) be an invertible sheaf on \(X\).
  Let \(\sF_{\sL}\) be the sub-\(\cO_X\)-module of \(\sL\) generated by
  \(H^0(X,\sL)\), i.e.,
  \[
    \sF_{\sL} \coloneqq \im\bigl(\eval\colon \cO_X^{(I)} \longrightarrow
    \sL\bigr) \subseteq \sL
  \]
  where \(\eval\) is the morphism corresponding to the set of
  global sections
  \(H^0(X,\sL)\) under the bijection in \cite[Chapitre 0, (5.1.1)]{EGAInew}.
  The \textsl{base ideal of \(\sL\)} is
  \[
    \fb(\sL) \coloneqq \Ann_{\cO_X}\bigl(\sL/\sF_{\sL}\bigr)
    \subseteq \cO_X.
  \]
  For a Cartier divisor \(L\) or a \(\ZZ\)-invertible sheaf \(L\) on \(X\), we set
  \[
    \fb\bigl(\lvert L \rvert \bigr) \coloneqq \fb\bigl(\cO_X(L)\bigr) \subseteq
    \cO_X.
  \]
\end{citeddef}
This definition matches the usual definition in \((\ref{lem:baselocuslrs})\)
below for invertible sheaves on schemes.
In particular, \((\ref{lem:baselocuslrs})\) says that we can compute base ideals
by evaluating \(H^0(X,\sL)\), and that this does not depend on the ring \(A\)
over which we consider \(X\) as a scheme.
\begin{lemma}\label{lem:baselocus}
  Let \(X\) be a ringed space and let \(\sL\) be an invertible sheaf on \(X\).
  \begin{enumerate}[label=\((\roman*)\),ref=\roman*]
    \item\label{lem:baselocustwist} We have
      \[
        \fb(\sL) = \im\Bigl(\eval \otimes_{\cO_X}
        \sL^{-1}\colon \cO_X^{(I)} \otimes_{\cO_X} \sL^{-1}
        \longrightarrow \cO_X \Bigr) \subseteq
        \cO_X.
      \]
    \item\label{lem:baselocusqc} Suppose that \(X\) is a scheme.
      Then, \(\fb(\sL)\) is quasi-coherent.
    \item\label{lem:baselocuslrs} Suppose that \(X\) is a quasi-compact
      quasi-separated scheme.
      Consider a ring map \(A \to H^0(X,\cO_X)\) and denote by \(\pi_A\colon
      X \to \Spec(A)\) the corresponding morphism under the natural bijection
      \[
        \Hom_{\mathsf{LRS}}\bigl(X,\Spec(A)\bigr)
        \overset{\sim}{\longrightarrow}
        \Hom_{\mathsf{Ring}}\bigl(A,H^0(X,\cO_X)\bigr)
      \]
      from \emph{\citeleft\citen{Gre61}\citemid Proposition 1\citepunct
      \citen{EGAII}\citemid Errata et addenda,
      Proposition 1.8.1\citeright}.
      Consider the twist
      \begin{equation}\label{eq:counittwist}
        \pi_A^*\pi_{A*}\sL \otimes_{\cO_X} \sL^{-1}
        \longrightarrow \cO_X
      \end{equation}
      of the counit morphism for the adjunction \(\pi_A^*
      \dashv \pi_{A*}\) by \(\sL^{-1}\).
      Then,
      \[
        \fb(\sL) = \im\Bigl(\pi_A^*\pi_{A*}\sL \otimes_{\cO_X} \sL^{-1}
        \longrightarrow \cO_X\Bigr).
      \]
  \end{enumerate}
\end{lemma}
\begin{notation}
  We will often denote the twist of the counit morphism in
  \eqref{eq:counittwist} by
  \[
    H^0(X,\sL) \otimes_A \sL^{-1} \longrightarrow \cO_X.
  \]
\end{notation}
\begin{proof}[Proof of Lemma \ref{lem:baselocus}]
  \((\ref{lem:baselocustwist})\).
  We have the equalities
  \begin{align*}
    \fb(\sL) &= \Ann_{\cO_X}\bigl(\sL/\sF_\sL\bigr)\\
    &= \Ann_{\cO_X}\bigl(\sL/\sF_\sL \otimes_{\cO_X} \sL^{-1}\bigr)\\
    &= \Ann_{\cO_X}\Bigl(\cO_X\big/\bigl(\sF_\sL \otimes_{\cO_X} \sL^{-1}\bigr)
    \Bigr)\\
    &= \sF_\sL \otimes_{\cO_X} \sL^{-1}\\
    &= \im\bigl(\eval \otimes_{\cO_X} \sL^{-1}\bigr)
  \end{align*}
  where the middle two equalities and the last equality
  hold by the assumption that \(\sL\) is
  invertible.
  \par \((\ref{lem:baselocusqc})\) follows from \((\ref{lem:baselocustwist})\)
  because \(\cO_X^{(I)} \otimes_{\cO_X \sL^{-1}}\) is quasi-coherent and,
  on a scheme \(X\), the image
  of a morphism of quasi-coherent \(\cO_X\)-modules is quasi-coherent
  \cite[Corollaire
  2.2.2\((ii)\)]{EGAInew}.
  \par \((\ref{lem:baselocuslrs})\).
  Consider the surjection
  \[
    A^{(I)} \longtwoheadrightarrow H^0(X,\sL)
  \]
  where \(A^{(I)}\) is the direct sum of copies of \(A\) indexed by
  \(H^0(X,\sL)\), and the map sends the generator \(1_s\) of the copy of \(A\)
  indexed by \(s \in H^0(X,\sL)\) to \(s\).
  Note that \(\pi_A\colon X \to \Spec(A)\) is quasi-compact and quasi-separated
  by \cite[Corollaire 6.1.10]{EGAInew}.
  Thus, \(\pi_{A*}\sL\) is quasi-coherent \cite[Proposition 6.7.1]{EGAInew}.
  Taking sheaves associated to the modules above, we obtain the surjection
  \begin{equation}\label{eq:evalonspeca}
    \cO_{\Spec(A)}^{(I)} \longtwoheadrightarrow \pi_{A*}\sL.
  \end{equation}
  We claim that the evaluation morphism \(\eval\colon \cO_X^{(I)} \to \sL\)
  factors as
  \begin{equation}\label{eq:evalfactor}
    \cO_X^{(I)} \longtwoheadrightarrow
    \pi_A^*\pi_{A*}\sL
    \longrightarrow \sL
  \end{equation}
  where the first morphism is the surjection obtained by pulling back
  \eqref{eq:evalonspeca}.
  The factorization \eqref{eq:evalfactor} holds because for every affine open
  subset \(U = \Spec(B) \subseteq X\), these morphisms restrict to
  \[
    \begin{tikzcd}[row sep=0,column sep=1.475em]
      B^{(I)} \rar[twoheadrightarrow]
      & H^0(X,\sL) \otimes_A B \rar
      & H^0\bigl(U,\sL\rvert_{U}\bigr)\\
      1_s \rar[mapsto] & s \otimes 1 \rar[mapsto] & s\rvert_U\mathrlap{.}
    \end{tikzcd}
  \]
  Since the first morphism in \eqref{eq:evalfactor} is surjective, the twist
  \(\eval \otimes_{\cO_X}\sL^{-1}\) of the composition and the twist
  \(\pi_A^*\pi_{A*}\sL \otimes_{\cO_X} \sL^{-1} \to\cO_X\) of the
  second morphism have the same image.
\end{proof}
We can now define base schemes and base loci.
\begin{definition}[see {\cite[\S1.1.B]{Laz04a}}]
  Let \(X\) be a scheme and let \(\sL\) be an invertible sheaf on \(X\).
  Since \(\mathfrak{b}(\sL)\) is quasi-coherent by Lemma
  \ref{lem:baselocus}\((\ref{lem:baselocusqc})\), we can define
  the \textsl{base scheme of \(\sL\)} to be the closed subscheme \(\Bs(\sL)
  \subseteq X\)
  corresponding to \(\mathfrak{b}(\sL)\).
  The \textsl{base locus of \(\sL\)} is the underlying closed subset
  \(\Bs(\sL)_\red \subseteq X\) of \(\Bs(\sL)\).
  If \(L\) is a \(\ZZ\)-invertible sheaf, we set
  \begin{align*}
    \Bs\bigl(\lvert L \rvert\bigr) &\coloneqq \Bs\bigl(\cO_X(L)\bigr) \subseteq
    X,\\
    \Bs\bigl(\lvert L \rvert\bigr)_\red &\coloneqq
    \Bs\bigl(\cO_X(L)\bigr)_\red \subseteq X.
  \end{align*}
\end{definition}
Next, we define stable base loci for invertible sheaves, which were introduced by
Fujita \cite{Fuj83} and defined at this level of generality by K\"uronya
\cite{Kur13}.
\begin{citeddef}[{\citeleft\citen{Fuj83}\citemid Definition 1.17\citepunct
  \citen{Kur13}\citemid Definition 2.1 and Remark 2.4\citeright}]
  Let \(X\) be a scheme and let \(\sL\) be an invertible sheaf on \(X\).
  The \textsl{stable base locus of \(\sL\)} is the closed subset
  \[
    \SB(\sL) \coloneqq \bigcap_{m=1}^\infty
    \Bs\bigl(\sL^{\otimes m}\bigr)_\red \subseteq
    X.
  \]
  If the underlying topological space of \(X\) is Noetherian,
  this intersection stabilizes, and hence there exists
  \(m_0 \ge 1\) such that
  \begin{equation}
    \SB(\sL) = \Bs\bigl(\sL^{\otimes pm_0}\bigr)_\red\label{eq:kurrem24}
  \end{equation}
  for all \(p \ge 1\).
\end{citeddef}
Because \(\Pic(X) \to \Pic(X) \otimes_\ZZ \QQ\) is not necessarily injective, we
must take care in defining stable base loci for \(\QQ\)-invertible sheaves.
\begin{definition}\label{def:sbqinvertible}
  Let \(\bk \subseteq \QQ\) be a subring.
  Let \(X\) be a scheme whose underlying topological space is Noetherian
  and let \(L\) be a \(\bk\)-invertible sheaf.
  We define the \textsl{stable base locus of \(L\)} as
  \[
    \SB(L) \coloneqq \SB(\sL) \subseteq X
  \] 
  where \(\sL\) is a invertible sheaf whose image is \(mL\) for some \(m \in
  \ZZ_{>0}\).
\end{definition}
Statement \((\ref{lem:sbwelldefitem})\) below says that this definition does not
depend on the choice of \(\sL\).
\begin{lemma}\label{lem:sbwelldef}
  Let \(\bk \subseteq \QQ\) be a subring.
  Let \(X\) be a ringed space
  and let \(L\) be a \(\bk\)-invertible sheaf.
  Let \(\sL_1,\sL_2\) be invertible sheaves whose images in \(\Pic_{\bk}(X)\)
  are \(n_1L,n_2L\) for some \(n_1,n_2 \in \ZZ_{>0}\).
  \begin{enumerate}[label=\((\roman*)\),ref=\roman*]
    \item\label{lem:baseidealinvsheafiso}
      There exists \(m_0 \ge 1\) such that \(\sL_1^{\otimes m_0n_2} \cong
      \sL_2^{\otimes m_0n_1}\).
    \item\label{lem:baseidealwelldef}
      With \(m_0\) as in \((\ref{lem:baseidealinvsheafiso})\), we have
      \(\fb(\sL_1^{\otimes pm_0n_2})
      = \fb(\sL_2^{\otimes pm_0n_1})\)
      for all \(p \ge 1\).
    \item\label{lem:sbwelldefitem}
      Suppose that \(X\) is a scheme whose underlying topological space is
      Noetherian.
      Then, we have \(\SB(\sL_1) = \SB(\sL_2)\).
      Thus, the definition of \(\SB(L)\) in Definition \ref{def:sbqinvertible}
      does not depend on the choice of \(\sL\).
  \end{enumerate}
\end{lemma}
\begin{proof}
  For \((\ref{lem:baseidealinvsheafiso})\) and
  \((\ref{lem:baseidealwelldef})\),
  since \(\sL_1^{\otimes n_2}\) and \(\sL_2^{\otimes n_1}\) have the same image
  under the map \(\Pic(X) \to \Pic(X) \otimes_\ZZ \bk\), there exists \(m_0 \ge
  1\) such that
  \begin{align*}
    \sL_1^{\otimes m_0n_2} &\cong \sL_2^{\otimes m_0n_1}.
    \intertext{This shows \((\ref{lem:baseidealinvsheafiso})\).
    Taking \(p\)-th tensor powers, we obtain
    \((\ref{lem:baseidealwelldef})\).
    To show \((\ref{lem:sbwelldefitem})\),
    let \(m_0\) be as in \((\ref{lem:baseidealwelldef})\).
    Because the underlying topological space of \(X\) is Noetherian,
    there exists \(p \ge 1\) such that}
    \SB(\sL_1) = \Bs\bigl(\sL_1^{\otimes pm_0n_2}\bigr)_\red
    &= \Bs\bigl(\sL_2^{\otimes pm_0n_1}\bigr)_\red = \SB(\sL_2)
  \end{align*}
  by \eqref{eq:kurrem24}, where the middle equality holds by
  \((\ref{lem:baseidealwelldef})\).
\end{proof}
We have the following birational transformation rule for stable base loci.
\begin{lemma}\label{lem:basebirat}
  Let \(X\) be a normal Noetherian scheme.
  Let \(\sL\) be an invertible sheaf on \(X\) and let \(L\) be a
  \(\bk\)-invertible sheaf on \(X\) for a subring \(\bk \subseteq \QQ\).
  For every proper birational morphism \(f\colon X' \to X\), we have
  \begin{align*}
    \fb(f^*\sL) &= f^{-1}\bigl(\fb(\sL)\bigr) \cdot \cO_{X'},\\
    \Bs(f^*\sL) &= f^{-1}\bigl(\Bs(\sL)\bigr),\\
    \SB(f^*L) &= f^{-1}\bigl(\SB(L)\bigr).
  \end{align*}
\end{lemma}
\begin{proof}
  We claim it suffices to prove the statement for invertible sheaves \(\sL\).
  Let \(m\) be sufficiently large such that \(mL\) is the image of \(\sL\) under
  \(\Pic(X) \to \Pic_\bk(X)\), and such that \(\SB(L) = \Bs(\sL)\)
  and \(\SB(f^*L) = \Bs(f^*\sL)\).
  Such an integer \(m\) exists exists by \eqref{eq:kurrem24} and Lemma
  \ref{lem:sbwelldef}\((\ref{lem:sbwelldefitem})\).
  It therefore suffices to show that \(\Bs(f^*\sL) = f^{-1}(\Bs(\sL))\).
  \par Set \(A \coloneqq \Gamma(X,\cO_X)\).
  We have the commutative diagram
  \[
    \begin{tikzcd}
      H^0(X,\sL) \otimes_A f^*\sL^{-1}
      \rar\arrow{d}[swap,sloped]{\sim} & \cO_{X'}
      \dar[equals]\\
      H^0(X',f^*\sL) \otimes_A f^*\sL^{-1} \rar &
      \cO_{X'}
    \end{tikzcd}
  \]
  where the top horizontal map is the pullback of the evaluation morphism for
  \(\sL\) on \(X\) and the left vertical map is an isomorphism by the
  assumption that \(X\) is normal.
  This shows that \(\fb(f^*\sL) = f^{-1}(\fb(\sL)) \cdot \cO_{X'}\) and
  \(\Bs(f^*\sL) = f^{-1}(\Bs(\sL))\).
\end{proof}
\subsection{Augmented base loci}
Next, we define augmented base loci, which were introduced by Ein,
Lazarsfeld, Musta\c{t}\u{a}, Nakamaye, and Popa \cite{ELMNP06} for normal
projective varieties and by Birkar \cite{Bir17} in general.
The augmented base locus for big and nef divisors
appears in earlier work of Nakamaye
\cite[Theorem 0.3]{Nak00}.
\begin{citeddef}[{\citeleft\citen{ELMNP06}\citemid Definition 1.2\citepunct
  \citen{Bir17}\citemid Definition 1.2 and Lemma 3.1\citeright}]
  Let \(X\) be a projective scheme over a field \(k\)
  and let \(L\) be an \(\RR\)-Cartier divisor
  on \(X\).
  Write
  \[
    L \sim_\RR \sum_{i=1}^r t_iA_i
  \]
  where each \(A_i\) is a very ample Cartier divisor and \(t_i \in \RR\), which
  is possible by \citeleft\citen{Kle66}\citemid Chapter IV, \S1, Theorem
  1\citepunct \citen{Kee03}\citemid Theorem 3.9\citeright.
  Set
  \[
    \langle mL \rangle \coloneqq \sum_{i=1}^r \lfloor mt_i \rfloor A_i
  \]
  for every integer \(m \ge 1\), which depends on the \(\RR\)-linear equivalence
  above and the decomposition \(\sum_{i=1}^r t_iA_i\).
  The \textsl{augmented base locus of \(L\)} is the closed subset
  \[
    \Bplus(L) \coloneqq \bigcap_{m \ge 1} \SB\bigl(\langle mL \rangle - A \bigr)
    \subseteq X
  \]
  for a cnoice of ample Cartier divisor \(A\).
  By \cite[Lemma 3.1]{Bir17}, the set \(\Bplus(L)\) does not depend on the
  choice of \(A\) or the choice of the expression \(L \sim_\RR \sum_{i=1}^r
  t_iA_i\), and we have
  \begin{equation}\label{eq:bpluselmnpdef}
    \Bplus(L) = \bigcap_H \SB(L-H)
  \end{equation}
  where \(H\) runs over all ample \(\RR\)-Cartier divisors such that \(L-H\) is
  \(\QQ\)-Cartier.
  The description in \eqref{eq:bpluselmnpdef} and \cite[Proof of Proposition
  1.4]{ELMNP06} (whose proof works for arbitrary projective schemes)
  imply that \(\Bplus(L)\) only
  depends on the numerical equivalence class of \(L\).
\end{citeddef}
\begin{remark}
  For any subring \(\bk \subseteq \RR\), we can combine \cite[Lemma
  3.1]{Bir17}, \eqref{eq:bpluselmnpdef}, and \cite[Proof of Proposition
  1.4]{ELMNP06} to deduce that the augmented base locus
  only depends on the image of a \(\bk\)-Cartier divisor
  \(L\) in \(\Div_{\RR}(X)\).
  Thus, we can define the augmented base locus of a \(\bk\)-Cartier divisor
  \(L\) as the augmented base locus of the image of \(L\) in \(\Div_\RR(X)\).
  When proving statements about augmented base loci, it therefore suffices to
  consider the case when \(\bk = \RR\).
\end{remark}
Augmented base loci satisfy the following birational transformation rule.
\begin{proposition}[cf.\ {\cite[Proposition 2.3]{BBP13}}]\label{prop:bbp23}
  Let \(f\colon X' \to X\) be a birational morphism between normal projective
  varieties.
  Let \(D\) be an \(\RR\)-Cartier divisor on \(X\).
  Then,
  \[
    \Bplus(f^*D+F) = f^{-1}\bigl(\Bplus(D)\bigr) \cup \Exc(f)
  \]
  for every \(f\)-exceptional \(\RR\)-Cartier divisor \(F\) on \(X'\).
\end{proposition}
\begin{proof}
  The proof of \cite[Proposition 2.3]{BBP13} applies because the negativity
  lemma holds over arbitrary fields \cite[Lemma 2.15]{DW22}.
\end{proof}
We also need the following result, which describes how \(\Bplus(D)\) is the
locus where \(D\) is ample.
Regularity in the proof below is in the sense of Castelnuovo and Mumford
\cite[p.\ 99]{Mum66} (see \cite[Definition 1.8.4]{Laz04a}).
\begin{theorem}[cf.\ {\citeleft\citen{Kur13}\citemid Propisition
  2.7\citepunct \citen{FM21}\citemid Lemma 6.11\citeright}]
  \label{prop:kur13prop27}
  Let \(\bk \subseteq \QQ\) be a subring.
  Let \(X\) be a projective scheme over a field \(k\)
  and let \(D\) be a \(\bk\)-Cartier divisor
  with a decomposition \(\mathcal{D}\).
  Then, \(\Bplus(D)\) is the smallest closed subset of \(X\) such that the
  following property holds:
  \begin{quote}
    For every
    coherent sheaf \(\sF\) on \(X\) and for every \(\RR\)-Cartier divisor \(E\)
    with decomposition \(\mathcal{E}\), there exists an integer \(n_0\) such
    that the sheaves
    \begin{align*}
      \sF &\otimes_{\cO_X} \cO_X\bigl(\lfloor E+nD
      \rfloor_{\mathcal{E}+n\mathcal{D}}
      + P\bigr)\\
      \sF &\otimes_{\cO_X} \cO_X\bigl(\lceil E+nD
      \rceil_{\mathcal{E}+n\mathcal{D}} + P\bigr)
    \end{align*}
    are globally
    generated on \(X - \Bplus(D)\) for every integer \(n \ge n_0\)
    and every nef \(\ZZ\)-invertible sheaf
    \(P\), where the rounding is done with respect
    to the decomposition \(\mathcal{E}+n\mathcal{D}\) obtained by adding the
    decompositions for \(E\) and \(D\).
  \end{quote}
  If \(X\) satisfies \(G_1\) and \(S_2\) (resp.\ is normal),
  then the same conclusion holds for \(\bk\)-Cartier
  divisors \(D\) (resp.\ \(\bk\)-Cartier
  \(\bk\)-Weil divisors \(D\))
  and \(\RR\)-AC divisors \(E\) (resp.\ \(\RR\)-Weil divisors
  \(E\)).
  In the normal case,
  rounding is done in the sense of \(\RR\)-Weil divisors.
\end{theorem}
\begin{proof}
  We first show that \(\Bplus(D)\) satisfies the condition in the proposition.
  If \(\Bplus(D) = X\), then the condition trivially holds.
  We therefore assume that \(\Bplus(D) \ne X\).
  \par Let \(A\) be an ample and free Cartier divisor on \(X\).
  Let
  \begin{align*}
    D &= \sum_{i=1}^r a_iD_i
    \intertext{be the decomposition \(\mathcal{D}\) where \(a_i \in \bk\).
    For each integer \(n\), we denote by \(n\mathcal{D}\) the decomposition}
    nD &= \sum_{i=1}^r na_iD_i.
  \end{align*}
  By \cite[Proposition 1.5]{ELMNP06} (whose proof works for arbitrary projective
  schemes) and \eqref{eq:kurrem24},
  there exist
  positive integers \(q\) and \(r\) such that \(qra_i \in \ZZ\) for all \(i\)
  and
  \begin{equation}\label{eq:bplusstabilize}
    \Bplus(D) = \SB(rD-A) = \Bs\bigl(\bigl\lvert q(rD-A)\bigr\rvert\bigr)_\red
  \end{equation}
  where \(qrD\) has the decomposition \(qr\mathcal{D}\).
  After possibly replacing \(A\) and \(r\) by \(qr\) and \(qA\), respectively,
  we can
  assume that \(r\) is an integer such that \(ra_i \in \ZZ\) for all \(i\)
  and
  \(\Bplus(D) = \Bs(\lvert q(rD-A) \rvert)_\red\) for all \(q \ge 1\).
  \par We claim that there exists an integer \(m_0\) such that
  \begin{equation}\label{eq:sheavestobegg}
    \begin{aligned}
      \sF \otimes_{\cO_X} \cO_X\bigl(mA+\lfloor E+jD \rfloor+P\bigr)\\
      \sF \otimes_{\cO_X} \cO_X\bigl(mA+\lceil E+jD \rceil+P\bigr)
    \end{aligned}
  \end{equation}
  are globally generated for every
  \(m \ge m_0\), every \(1 \le j < r\), and every nef \(\ZZ\)-invertible sheaf
  \(P\), where
  \(\lfloor E+jD \rfloor\) and \(\lceil E+jD \rceil\)
  should either be interpreted in the sense of \(\bk\)-Cartier
  divisors or \(\bk\)-AC divisors
  with respect to the decomposition \(\mathcal{E}+j\mathcal{D}\),
  or interpreted in
  the sense of \(\RR\)-Weil divisors in the situation when \(X\) is
  normal.
  By Fujita's vanishing theorem \citeleft\citen{Fuj83}\citemid Theorem
  5.1\citeright\ applied to \(X_{\bar{k}}\),
  there exists an integer \(m_1\) such that for all
  integers \(m \ge m_1\) and all \(i > 0\), we have
  \begin{align*}
    H^i\Bigl(X,\sF \otimes_{\cO_X} \cO_X\bigl(mA+\lfloor E+jD \rfloor+P\bigr)\Bigr) &= 0\\
    H^i\Bigl(X,\sF \otimes_{\cO_X} \cO_X\bigl(mA+\lceil E+jD \rceil+P\bigr)\Bigr) &= 0
  \end{align*}
  for all \(0 \le j < r\) and every nef \(\ZZ\)-invertible sheaf \(P\).
  Thus, if \(m \ge m_1 + \dim X\), then the coherent sheaves in
  \eqref{eq:sheavestobegg}
  are
  \(0\)-regular with respect to \(A\), and are therefore
  globally generated by \cite[Chapter II, \S1, Proposition 1\((iii)\)]{Kle66}
  applied to \(X_{\bar{k}}\).
  It therefore suffices to set \(m_0 = m_1 + \dim(X)\).
  \par To prove that \(\Bplus(D)\) satisfies the condition in the proposition,
  we note that by the above, the sheaves
  \begin{align*}
    \sF \otimes_{\cO_X} \cO_X\bigl(mA+\lfloor E+jD \rfloor+P\bigr) &\otimes_{\cO_X}
    \cO_X\bigl(q(rD-A)\bigr)\\
    \sF \otimes_{\cO_X} \cO_X\bigl(mA+\lceil E+jD \rceil+P\bigr) &\otimes_{\cO_X}
    \cO_X\bigl(q(rD-A)\bigr)
  \end{align*}
  are globally generated away from
  \(\Bplus(D)\)
  for all \(m \ge m_0\), all \(q \ge 1\), all \(0 \le j < r\), and every
  nef \(\ZZ\)-invertible sheaf \(P\).
  Here, we use the decomposition \(qr\mathcal{D}\) for \(qrD\) as before.
  Setting \(q=m\), we see that the sheaves 
  \begin{align*}
    \sF \otimes_{\cO_X} \cO_X\bigl(mrD+\lfloor E+jD \rfloor+P\bigr) &\cong
    \sF \otimes_{\cO_X} \cO_X\bigl(\lfloor E+(mr+j)D \rfloor+P\bigr)\\
    \sF \otimes_{\cO_X} \cO_X\bigl(mrD+\lceil E+jD \rceil+P\bigr) &\cong
    \sF \otimes_{\cO_X} \cO_X\bigl(\lceil E+(mr+j)D \rceil+P\bigr)
  \end{align*}
  are globally generated away from \(\Bplus(D)\)
  for all \(m \ge m_0\), all \(0 \le j < r\), and every nef
  \(\ZZ\)-invertible sheaf \(P\).
  It therefore suffices to set \(n_0 = m_0r\).
  \par Finally, we show \(\Bplus(D)\) is the smallest closed subset satisfying
  the condition in the proposition.
  Let \(x \in \Bplus(D)\); it suffices to show that for \(\sF = \cO_X(-A)\) where
  \(A\) is ample, the sheaf
  \[
    \sF \otimes_{\cO_X} \cO_X(nD) = \cO_X(nD-A) 
  \]
  is not globally generated at \(x\)
  for all \(n \ge 0\) such that \(nD\) is a Cartier divisor.
  This follows from \cite[Proposition 1.5]{ELMNP06} (whose proof works for
  arbitrary projective schemes) since \(x \in \Bplus(D)\).
\end{proof}

\subsection{The Nakai--Moishezon criterion for
\texorpdfstring{\(\RR\)}{R}-invertible sheaves}
We will need the following result, which is a version of the
numerical criterion for ampleness due to Nakai \citeleft\citen{Nak60}\citemid
Theorem 9\citepunct \citen{Nak63}\citemid Theorem 2\citeright,
Moishezon \citeleft\citen{Moi61}\citemid Theorem 1\citepunct
\citen{Moi62}\citemid Theorem 1\citepunct
\citen{Moi64}\citemid Chapter II, Theorem 2\citeright, and Kleiman
\citeleft\citen{Kle65}\citemid Theorem 2\citepunct \citen{Kle66}\citemid Chapter
III, \S1, Theorem 1\citeright\
for \(\RR\)-invertible sheaves over arbitrary fields.
The projective case over algebraically closed fields is due to Campana and
Peternell \cite{CP90}.
The complete case over algebraically closed fields is due to Fujino and Miyamoto
\cite{FM}.
\begin{theorem}[{cf.\ \citeleft\citen{CP90}\citemid Theorem 1.3\citepunct
  \citen{FM}\citemid Theorem 1.3\citeright}]\label{thm:fujinomiyamoto}
  Let \(X\) be a complete scheme over a field \(k\) and let \(L\) be an
  \(\RR\)-invertible sheaf on \(X\).
  Then, \(L\) is ample if and only if \((L^{\dim(X)} \cdot Z) > 0\) for every
  positive-dimensional closed integral subscheme \(Z \subseteq X\).
\end{theorem}
\begin{proof}
  The proof in \cite{FM} only uses the assumption that \(k\) is algebraically
  closed in two places.
  \begin{enumerate}
    \item In \cite[Proof of Lemma 2.3, Step 2]{FM}, the authors use Siu's
      numerical criterion for a difference of nef divisors to be big
      \cite[Corollary 1.2]{Siu93}.
      The same criterion holds over arbitrary fields by \cite[Proposition
      5.4]{Cut15}.
    \item In \cite[Proof of Theorem 1.4, Line \(-3\)]{FM}, the authors use
      \cite[Lemma 2.1.11]{Fuj17book} to say that an \(\RR\)-invertible sheaf
      \(L\) on a normal complete variety \(X\) is semiample if and only if there
      exists a morphism \(f\colon X \to Y\) onto a normal projective variety
      such that the following properties hold:
      \begin{itemize}
        \item \(f^\#\colon \cO_Y \to f_*\cO_X\) is an isomorphism.
        \item There exists an ample \(\RR\)-invertible sheaf \(A\) on \(Y\) such
          that \(L \sim_\RR f^*A\).
      \end{itemize}
      The proof of \cite[Lemma 2.1.11]{Fuj17book} works over arbitrary fields
      without changes.
  \end{enumerate}
  With these modifications, the proof in \cite{FM} works over arbitrary fields.
\end{proof}

\section{Seshadri constants and moving Seshadri constants over arbitrary fields}
\label{sect:seshadri}
In this section, we extend the theory of moving Seshadri constants to varieties
over arbitrary
fields.
Seshadri constants measure the local positivity of nef Cartier divisors or
invertible sheaves.
Moving Seshadri constants are an extension of this notion to Cartier divisors
and invertible sheaves
that are not necessarily nef.\medskip
\par Seshadri constants were introduced by Demailly \cite[\S6]{Dem92} based on
Seshadri's ampleness criterion \citeleft\citen{Har70}\citemid Chapter I,
Theorem 7.1\citepunct \citen{Ses72}\citemid Remark 7.1\citeright, and moving
Seshadri constants were introduced by Nakamaye \cite{Nak03} and Ein, Lazarsfeld,
Musta\c{t}\u{a}, Nakamaye, and Popa \cite[\S6]{ELMNP09}, all for smooth complex
projective varieties.
Foundational results on Seshadri constants, especially those connecting Seshadri
constants with generation of jets, were proved over algebraically closed fields
of arbitrary characteristic in \cite[Chapter 5]{Laz04a}, \cite{MS14},
\cite{Mur18}, and \cite{FM21}.
Many of the results in \cite[Chapter 5]{Laz04a} and \cite{MS14} were extended to
arbitrary fields in \cite[Appendix]{DW22}.\medskip
\par The new material in this section extend many of the results for moving
Seshadri constants in \cite[\S6]{ELMNP09} in two directions:
\begin{enumerate*}
  \item to possibly singular varieties, which is new even over the complex
    numbers; and
  \item to varieties over arbitrary fields.
\end{enumerate*}
We also extend to arbitrary fields some results on Seshadri constants from
\cite[Chapter 5]{Laz04a} and \cite{Mur18} that were not proved in
\cite[Appendix]{DW22}.\medskip
\par Preliminary versions of some of this material for \(k\)-rational points on
normal projective varieties appears in the author's thesis
\cite[Chapter 7]{Mur19}.
This section is a corrected and generalized version of \cite[Chapter 7]{Mur19}.
\subsection{Seshadri constants}
Demailly introduced Seshadri constants in \cite[\S6]{Dem92} on smooth complex
projective varieties to measure the local positivity of nef divisors
with the hope that they could be used to prove cases of Fujita's
conjecture.
\begin{citeddef}[{\cite[Definition 5.1.1]{Laz04a}}]\label{def:seshadriconstant}
  Let \(X\) be a complete scheme and let
  \(D\) be a nef \(\RR\)-invertible sheaf on \(X\).
  Let \(x \in X\) be a closed point and let \(\mu\colon \tilde{X} \to X\) be the
  blowup of \(X\) at \(x\) with exceptional divisor \(E\).
  The \textsl{Seshadri constant of \(D\) at \(x\)} is
  \[
    \varepsilon(D;x) \coloneqq \sup\Set[\big]{t \in \RR_{\ge 0} \given \mu^*D -
    tE\ \text{is nef}}.
  \]
\end{citeddef}
\begin{remark}
  For any subring \(\bk \subseteq \RR\),
  Lemma \ref{lem:sbwelldef}\((\ref{lem:baseidealinvsheafiso})\) and the
  linearity of intersection numbers \cite[Theorem B.9(2)]{Kle05} imply that
  the Seshadri constant depends only on the image of a \(\bk\)-invertible sheaf
  \(D\) in \(\Pic_\RR(X)\).
  Thus, we can define the Seshadri constant of a \(\bk\)-invertible sheaf \(D\)
  as the Seshadri constant of the image of \(D\) in \(\Pic_\RR(X)\).
  When proving statements about Seshadri constants, it therefore suffices to
  consider the case when \(\bk = \RR\).
\end{remark}
We prove that Seshadri constants can be characterized in terms of intersection
numbers with subvarieties and multiplicities of those subvarieties.
The description in terms of intersecting with curves is Demailly's original
definition in \cite{Dem92} and was extended to arbitrary fields in \cite{DW22}.
The description in terms of intersecting with all positive-dimensional
subvarieties is due to Lazarsfeld \cite{Laz04a} for projective varieties over
algebraically closed fields.
The extensions of this description in terms of all positive-dimensional
subvarieties to complete schemes and to arbitrary fields are new to this
paper.
\begin{proposition}[cf.\ {\citeleft\citen{Laz04a}\citemid Propositions 5.1.5 and
  5.1.9\citepunct \citen{DW22}\citemid Lemma A.8\citeright}]\label{prop:laz519}
  Let \(X\) be a complete scheme over a field \(k\) and let
  \(D\) be a nef \(\RR\)-invertible sheaf on \(X\).
  For every closed point \(x \in X\), we have
  \begin{align*}
    \varepsilon(D;x) &= \inf_{C \ni x} \biggl\{ \frac{(D \cdot
    C)}{[k(x):k] \cdot e(\cO_{C,x})} \biggr\}\\
    &= \min_{V \ni x} \Biggl\{ \biggl( \frac{(D^{\dim(V)} \cdot
    V)}{[k(x):k] \cdot e(\cO_{V,x})} \biggr)^{1/{\dim(V)}} \Biggr\}
  \end{align*}
  where the infimum (resp.\ minimum) is taken over all integral closed curves \(C
  \subseteq X\) (resp.\ positive-dimensional
  closed subvarieties \(V \subseteq X\))
  containing \(x\), intersection numbers are computed over \(k\), and
  \(e(\,\cdot\,)\) denotes the Hilbert--Samuel multiplicity of a Noetherian
  local ring.
\end{proposition}
\begin{proof}
  The first equality is shown in \cite[Lemma A.8]{DW22} (which does not use the
  assumption that \(X\) is projective or that \(X\) is a variety).
  It therefore suffices to show that
  \begin{equation}\label{eq:seshwrtsubvararbdim}
    \varepsilon(D;x) \le \biggl( \frac{(D^{\dim(V)} \cdot
    V)}{[k(x):k] \cdot e(\cO_{V,x})} \biggr)^{1/{\dim(V)}}
  \end{equation}
  and that equality holds for some \(V \subseteq X\).
  Set \(\varepsilon \coloneqq \varepsilon(D;x)\).
  \par We adapt the proof of \cite[Proposition 5.1.9]{Laz04a}.
  Let \(\mu\colon \tilde{X} \to X\) be the blowup of \(X\) at \(x\) with
  exceptional divisor \(E\).
  Let \(V' \subseteq \tilde{X}\) be the strict transform of \(V\), which is the
  blowup of \(V\) at \(x\) with exceptional divisor \(E\rvert_{V'}\).
  Since \(\mu^*D - \varepsilon E\) is nef, we have
  \[
    \bigl((\mu^*D - \varepsilon E)^{\dim(V')} \cdot V'\bigr) \ge 0
  \]
  by \cite[Lemma 2.12]{Kee03}.
  We know that for every \(0 < s < \dim(V)\), we have
  \begin{equation}\label{eq:crossterms}
    \bigl((\mu^*D)^{\dim(V')-s} \cdot E^{s} \cdot V'\bigr)
    = \Bigl(\bigl(\mu^*D\rvert_{V'}\bigr)^{\dim(V')-s} \cdot
    \bigl(E\rvert_{V'}\bigr)^{s}\Bigr) = 0
  \end{equation}
  by the fact that \((E\rvert_{V'})^s\) is \(\mu\)-exceptional.
  Thus, we see that
  \begin{equation}\label{eq:multcomp}
    \begin{aligned}
      \MoveEqLeft[3]
      \bigl((\mu^*D - \varepsilon E)^{\dim(V')} \cdot V'\bigr)\\
      &= \bigl((\mu^*D)^{\dim(V')} \cdot V'\bigr) + \bigl((-\varepsilon
      E)^{\dim(V)} \cdot V'\bigr)\\
      &= \bigl(D^{\dim(V)} \cdot V\bigr) - \varepsilon^{\dim(V')} [k(x) : k] \cdot
      e(\cO_{V,x})\\
      &\ge 0
    \end{aligned}
  \end{equation}
  where the second equality holds by \cite[Example 4.3.1]{Ful98} (see also
  \cite[Definition A.1 and Remark A.2]{DW22}).
  Rearranging this equality, we obtain \eqref{eq:seshwrtsubvararbdim}.
  \par It remains to show that equality holds in \eqref{eq:seshwrtsubvararbdim}
  for some \(V \subseteq X\).
  By the Nakai--Moishezon theorem for \(\RR\)-invertible sheaves (Theorem
  \ref{thm:fujinomiyamoto}), the fact that
  \(\mu^*D - \varepsilon E\) is nef but not ample implies there exists a
  subvariety \(V' \subseteq \tilde{X}\) such that
  \[
    \bigl((\mu^*D - \varepsilon E)^{\dim(V')} \cdot V'\bigr) = 0.
  \]
  Note that \(V' \not\subseteq E\), for otherwise \eqref{eq:crossterms} would
  imply
  \[
    \bigl((\mu^*D - \varepsilon E)^{\dim(V')} \cdot V'\bigr)
    = \varepsilon^{\dim(V')} \bigl((-E)^{\dim(V')} \cdot V'\bigr)
    > 0
  \]
  since \(-E\) is \(\mu\)-ample.
  We can therefore set \(V = \mu(V')\).
\end{proof}
\begin{corollary}\label{cor:seshbirat}
  Let \(X\) be a complete scheme
  and let
  \(D\) be a nef \(\RR\)-invertible sheaf on \(X\).
  Let \(x \in X\) be a closed point and let \(\pi\colon X' \to X\) be a proper
  birational morphism that is an isomorphism at \(x\).
  Then, \(\varepsilon(D;x) = \varepsilon(\pi^*D;\pi^{-1}(x))\).
\end{corollary}
\begin{proof}
  This follows from Proposition \ref{prop:laz519} and the projection formula
  \cite[Proposition B.16]{Kle05}.
\end{proof}

\subsection{Moving Seshadri constants}
We define moving Seshadri constants following \cite{ELMNP09}.
We have adapted their definition to work for complete schemes over arbitrary
fields.
\begin{definition}[cf.\ {\cite[Definition 6.1]{ELMNP09}}]\label{def:movingsesh}
  Let \(X\) be a complete scheme and let
  \(D\) be an \(\RR\)-invertible sheaf on \(X\).
  Let \(x \in X\) be a closed point.
  The \textsl{moving Seshadri constant of \(D\) at \(x\)} is
  \[
    \varepsilon\bigl(\lVert D \rVert;x\bigr) \coloneqq
    \sup_{f^*D \equiv_\RR A+E} \varepsilon\bigl(A;f^{-1}(x)\bigr)
  \]
  where the supremum runs over all birational morphisms \(f\colon X' \to X\)
  from projective schemes \(X'\) that are isomorphisms at
  \(x\) and \(\RR\)-numerical equivalences \(f^*D \equiv_\RR
  A+E\) where \(A\) is an ample \(\QQ\)-Cartier divisor and \(E\) is an
  effective \(\RR\)-Cartier divisor such that \(x \notin f(\Supp(E))\).
  If no such birational morphisms or \(\RR\)-numerical equivalences \(f^*D
  \equiv_\RR A+E\) exist, we set \(\varepsilon(\lVert D \rVert;x) \coloneqq
  0\).
\end{definition}
\begin{remark}\label{rem:bplusandmovsesh}
  \par Since the Seshadri constants \(\varepsilon(A;f^{-1}(x))\) are positive by
  Seshadri's criterion for ampleness \citeleft\citen{Har70}\citemid Chapter I,
  Theorem 7.1\citepunct \citen{Ses72}\citemid Remark 7.1\citeright, we see that
  \(\varepsilon(\lVert D \rVert;x) > 0\) if and only if \(\RR\)-numerical
  equivalences of the form \(f^*D \equiv_\RR A+E\) above exist.
  By Proposition \ref{prop:bbp23}, we can therefore think of the condition
  ``\(\varepsilon(\lVert D \rVert;x) > 0\)'' as a version of the condition ``\(x
  \notin \Bplus(D)\)'' that makes sense on complete schemes that are not
  necessarily projective or normal.
\end{remark}
\begin{remark}
  For smooth complex projective varieties \(X\), this definition matches the
  definition in
  \cite{ELMNP09}.
  If \(x \in \Bplus(D)\), then both definitions have value \(0\) by
  Proposition \ref{prop:bbp23}.
  For \(x \notin \Bplus(D)\), we want to show that
  \begin{equation}\label{eq:ourdefiselmnpdef}
    \varepsilon\bigl(\lVert D \rVert;x\bigr) =
    \sup_{\substack{f^*D \equiv_\RR A+E\\X'\ \text{smooth}}}
    \varepsilon\bigl(A;f^{-1}(x)\bigr).
  \end{equation}
  The inequality \(\ge\) holds since the right-hand side ranges over fewer
  birational morphisms \(f\).
  \par To show the inequality \(\le\) holds in \eqref{eq:ourdefiselmnpdef},
  let \(f\colon X' \to X\) be a
  birational morphism and let \(f^*D
  \equiv_\RR A+E\) be an \(\RR\)-numerical equivalence of the form in Definition
  \ref{def:movingsesh}.
  Let \(g\colon X'' \to X'\) be a projective resolution with exceptional
  divisor \(F\) that is an isomorphism at \(x\).
  Note that \(-F\) is \(g\)-ample, and hence
  \(g^*A-\frac{1}{n} F\) is ample for all \(n \gg 0\).
  We then have the \(\RR\)-numerical equivalence
  \[
    g^*f^*D \equiv_\RR g^*A + g^*E = \biggl(g^*A-\frac{1}{n} F\biggr) 
    + \biggl(g^*E+\frac{1}{n} F\biggr).
  \]
  Moreover, we have
  \[
    \varepsilon\bigl(A;f^{-1}(x)\bigr) = \varepsilon\bigl(g^*A;(f \circ
    g)^{-1}(x)\bigr) \le \lim_{n \to \infty}
    \varepsilon\biggl(g^*A-\frac{1}{n}F;(f \circ g)^{-1}(x)\biggr).
  \]
  Here, the equality holds by Corollary \ref{cor:seshbirat}, and the
  inequality holds by the lower semi-continuity of Seshadri constants with respect to ample
  perturbations of a nef class \cite[Corollary 3.34]{FM21}.
  Taking suprema, we see that ``\(\le\)'' holds in \eqref{eq:ourdefiselmnpdef}.
\end{remark}
\par We extend properties of moving Seshadri constants to complete
varieties over arbitrary fields.
For smooth complex projective varieties, these results were proved
in \cite{ELMNP09}.
\begin{proposition}[cf.\ {\cite[Proposition 6.3 and Remark 6.5]{ELMNP09}}]
  \label{prop:elmnp0963}
  Let \(X\) be a complete scheme over a field \(k\) and
  let \(D\) be an \(\RR\)-invertible sheaf on \(X\).
  Consider a closed point \(x \in X\).
  \begin{enumerate}[label=\((\roman*)\),ref=\roman*]
    \item\label{prop:elmnp0963i} Suppose that \(X\) is a complete variety.
      We then have
      \[
        \varepsilon\bigl(\lVert D \rVert;x\bigr) \le
        \biggl(\frac{\vol_X(D)}{[k(x):k] \cdot e(\cO_{X,x})}\biggr)^{1/{\dim(X)}}
      \]
      where the volume \(\vol_X(D)\) for \(\RR\)-invertible sheaves
      is defined using \emph{\cite[Theorem
      2.5]{Cut15}}.
    \item\label{prop:elmnp0963ii}
      If \(E\) is an \(\RR\)-invertible sheaf such that \(D \equiv_\RR E\),
      then \(\varepsilon(\lVert D \rVert;x) = \varepsilon(\lVert E
      \rVert;x)\).
    \item\label{prop:elmnp0963iii}
      \(\varepsilon(\lVert \lambda D \rVert;x) = \lambda \cdot
      \varepsilon(\lVert D \rVert;x)\) for every positive real number
      \(\lambda\).
    \item\label{prop:elmnp0963iv}
      If \(D\) is a nef \(\RR\)-invertible sheaf, then \(\varepsilon(\lVert D
      \rVert;x) \le \varepsilon(D;x)\).
      Equality holds if
      \(\varepsilon(\lVert D \rVert;x) = 0\) or if \(D\) is ample.
    \item\label{prop:elmnp0963v}
      Let \(D_1,D_2\) be \(\RR\)-invertible sheaves such that
      \(\varepsilon(\lVert D_1 \rVert;x) > 0\) and
      \(\varepsilon(\lVert D_2 \rVert;x) > 0\).
      Then,
      \[
        \varepsilon\bigl(\lVert D_1+D_2 \rVert;x\bigr) \ge 
        \varepsilon\bigl(\lVert D_1 \rVert;x\bigr)
        + \varepsilon\bigl(\lVert D_2 \rVert;x\bigr).
      \]
  \end{enumerate}
\end{proposition}
\begin{proof}
  \((\ref{prop:elmnp0963i})\).
  If there are no \(\RR\)-numerical equivalences of the form \(f^*D \equiv_\RR
  A+E\) as in Definition \ref{def:movingsesh}, we have \(\varepsilon(\lVert D
  \rVert;x) = 0\), in which case there is nothing to show.
  Otherwise, consider a birational morphism \(f\colon X' \to X\) from a
  projective scheme \(X'\) and an
  \(\RR\)-numerical equivalence
  \(f^*D \equiv_\RR A+E\) as in Definition \ref{def:movingsesh}.
  After replacing \(X'\) by its reduction (which does not affect the
  birationality of \(f\) by \cite[Proposition 4.5.15\((vii)\)]{EGAInew}),
  we may assume that \(X'\) is a
  projective variety.
  We have
  \[
    \varepsilon\bigl(A;f^{-1}(x)\bigr) \le
    \biggl(\frac{\vol_{X'}(f^*D)}{[k(x):k] \cdot
    e(\cO_{X,x})}\biggr)^{1/{\dim(X)}}
  \]
  by Proposition \ref{prop:laz519}.
  Since \(E\) is effective and \(\vol_{X'}(\,\cdot\,)\) is numerically invariant
  \citeleft\citen{Laz04a}\citemid Proposition 2.2.41\citepunct
  \citen{Cut15}\citemid p.\ 9\citeright,
  we have
  \[
    \vol_{X'}(A) \le \vol_{X'}(A+E) = \vol_{X'}(f^*D) = \vol_X(D),
  \]
  where the last equality holds by the invariance of volume under birational
  morphisms \cite[Lemma 2.9]{Cut24}.
  Note that while \cite[Lemma 2.9]{Cut24} is stated for birational morphisms of
  \emph{projective} varieties, the reduction in \cite[Proof of Lemma 2.9, Last
  Paragraph]{Cut24} to the case of \(\ZZ\)-invertible sheaves is valid for
  complete varieties by \cite[Theorem 2.5]{Cut15}, and the case of
  \(\ZZ\)-invertible sheaves is proved in 
  \citeleft\citen{Hol}\citemid Lemma 4.3\citeright\ (see also
  \citeleft\citen{LM}\citemid Proposition 5.6\citeright).\medskip
  \par \((\ref{prop:elmnp0963ii})\) and \((\ref{prop:elmnp0963iii})\) hold by
  the fact that we use \(\RR\)-numerical equivalences \(f^*D \equiv_\RR A+E\)
  to define \(\varepsilon(\lVert D \rVert;x)\) in
  Definition \ref{def:movingsesh}.\medskip
  \par \((\ref{prop:elmnp0963iv})\).
  Consider a birational morphism \(f\colon X' \to X\) from a
  projective scheme \(X\) that is an isomorphism around \(x\)
  and an \(\RR\)-numerical equivalence \(f^*D \equiv_\RR A+E\) as in Definition \ref{def:movingsesh}.
  We have
  \[
    \varepsilon\bigl(A;f^{-1}(x)\bigr) \le \varepsilon\bigl(A+E;f^{-1}(x)\bigr)
    = \varepsilon\bigl(f^*D;f^{-1}(x)\bigr) = \varepsilon(D;x)
  \]
  where the first inequality holds by Proposition \ref{prop:laz519} and
  the fact that \(f^{-1}(x) \notin
  \Supp(E)\), and the third equality holds by Corollary \ref{cor:seshbirat}.
  Taking suprema, this shows that
  \( \varepsilon(\lVert D \rVert;x) \le \varepsilon(D;x) \).
  \par For the reverse inequality, first suppose that
  \(\varepsilon(\lVert D \rVert;x)
  = 0\).
  Let \(f\colon X' \to X\) be a birational morphism from a projective scheme
  that is an isomorphism at \(x\), which exists by
  \citeleft\citen{Con07}\citemid Corollary 2.6\citepunct \citen{Del10}\citemid
  Corollaire 1.4\citeright.
  Since no \(\RR\)-numerical equivalences of the form \(f^*D \equiv_\RR A+E\)
  exist, we see that \(f^{-1}(x) \in \Bplus(f^*D)\) by \eqref{eq:bpluselmnpdef}.
  We then have
  \[
    \varepsilon(D;x) = \varepsilon\bigl(f^*D;f^{-1}(x)\bigr) = 
    \min_{V \ni f^{-1}(x)} \Biggl\{ \biggl( \frac{(f^*D^{\dim(V)} \cdot
    V)}{[k(x):k] \cdot e(\cO_{V,x})} \biggr)^{1/{\dim(V)}} \Biggr\}
    = 0
  \]
  by Corollary \ref{cor:seshbirat}, Proposition \ref{prop:laz519}, and
  \cite[Theorem 1.4]{Bir17}.
  Thus, we have \(\varepsilon\bigl(\lVert D \rVert;x\bigr) = \varepsilon(D;x) =
  0\) in this case.
  \par Now suppose that \(D\) is ample.
  Let \(f\colon X' \to X\) be a birational morphism from a projective
  scheme that is an isomorphism at \(x\) as before, and consider an 
  \(\RR\)-numerical equivalence \(f^*D \equiv_\RR A+E\) as in Definition \ref{def:movingsesh}.
  For every integer \(n > 0\), we then have
  \[
    f^*D \equiv_\RR \frac{n-1}{n} f^*D + \frac{1}{n}(A+E)
    = \biggl(\frac{n-1}{n} f^*D + \frac{1}{n}A\biggr) + \frac{1}{n}E
  \]
  where \(A_n \coloneqq \frac{n-1}{n} f^*D + \frac{1}{n}A\) is ample.
  We then have
  \begin{align*}
    \varepsilon\bigl(\lVert D \rVert;x\bigr) &\ge
    \varepsilon\bigl(A_n;f^{-1}(x)\bigr)
    \intertext{by definition.
    Taking limits as \(n \to \infty\), we see that}
    \lim_{n \to \infty} \varepsilon\bigl(A_n;f^{-1}(x)\bigr) &\ge
    \varepsilon\bigl(f^*D;f^{-1}(x)\bigr) = \varepsilon(D;x)
  \end{align*}
  by the lower semi-continuity of Seshadri constants with respect to ample
  perturbations of a nef class
  \cite[Corollary 3.34]{FM21} and by Corollary
  \ref{cor:seshbirat}.
  Thus, we have
  \( \varepsilon(\lVert D \rVert;x) \ge \varepsilon(D;x) \).\medskip
  \par \((\ref{prop:elmnp0963v})\).
  For \(i \in \{1,2\}\), let \(f_i\colon X'_i \to X\) be a birational morphism
  from a projective scheme on which we can write \(f_i^*D_i \equiv_\RR A_i + E_i\)
  as in Definition \ref{def:movingsesh}.
  By \cite[Premi\`ere partie, Corollaire 5.7.12]{RG71} (see also
  \cite[Theorem 2.11]{Con07}),
  we can find a blowup \(X_i \to X'_i\) that is an
  isomorphism at \(x\) such that \(X_i \to X\) is also a blowup along a coherent
  ideal sheaf \(\mathcal{I}_i \subseteq \cO_X\).
  Then, the blowup of \(X\) along the product of ideals
  \(\mathcal{I}_1\mathcal{I}_2\)  yields a blowup \(h \colon X' \to X\) fitting
  into the commutative diagram
  \[
    \begin{tikzcd}[column sep=scriptsize]
      & X'\arrow{ddl}{g_1}\arrow{dl}\arrow{dr}
      \arrow{ddr}[swap]{g_2}\arrow{ddd}[pos=0.6]{h}\\
      X_1\arrow{d}
      & 
      & X_2 \arrow{d}\\
      X'_1\arrow{dr}[swap]{f_1}
      & 
      & X'_2 \arrow{dl}{f_2}\\
      & X
    \end{tikzcd}
  \]
  of birational morphisms.
  Note that by construction, the morphism \(h\) is
  an isomorphism at \(x\), and the morphisms 
  \(g_1\) and \(g_2\) are blowups.
  Denote by \(F_1\) and \(F_2\) the exceptional divisors for \(g_1\) and
  \(g_2\), where we note that \(-F_1\) and \(-F_2\) are \(g_1\)- and
  \(g_2\)-ample.
  We can then write
  \[
    h^*D_i \equiv_\RR g_i^*A_i + g_i^*E_i
    = \biggl(g_i^*A_i - \frac{1}{n} F_i\biggr) + \biggl(g_i^*E_i + \frac{1}{n}
    F_i\biggr)
  \]
  for \(i \in \{1,2\}\), where both \(g_i^*A_i -
  \frac{1}{n} F_i\) are ample for all \(n \gg 0\).
  We then have
  \begin{align*}
    \varepsilon\bigl(A_1;f_1^{-1}(x)\bigr)
    + \varepsilon\bigl(A_2;f_2^{-1}(x)\bigr)
    &= \varepsilon\bigl(g_1^*A_1;h^{-1}(x)\bigr)
    + \varepsilon\bigl(g_2^*A_2;h^{-1}(x)\bigr)\\
    &\le \lim_{n\to\infty}\Biggl(
      \varepsilon\biggl(g_1^*A_1-\frac{1}{n} F_1;h^{-1}(x)\biggr)
    + \varepsilon\biggl(g_2^*A_2-\frac{1}{n} F_2;h^{-1}(x)\biggr) \Biggr)\\
    &\le \lim_{n\to\infty}
    \varepsilon\biggl(g_1^*A_1-\frac{1}{n} F_1+g_2^*A_2-\frac{1}{n} F_2;h^{-1}(x)\biggr)\\
    &\le \varepsilon\bigl(\lVert D+D' \rVert;x)
  \end{align*}
  where the first inequality holds by the lower semi-continuity of Seshadri constants with respect to ample
  perturbations of a nef class \cite[Corollary 3.34]{FM21}, the second inequality holds by definition of Seshadri
  constants, and the last inequality holds since
  \[
    h^*D \equiv_\RR \biggl(g_1^*A_1-\frac{1}{n} F_1+g_2^*A_2-\frac{1}{n}
    F_2\biggr) + \biggl(g_1^*E_1 + \frac{1}{n}
    F_1 + g_2^*E_2 + \frac{1}{n} F_2\biggr)
  \]
  is an \(\RR\)-numerical equivalence of the form in Definition \ref{def:movingsesh}.
\end{proof}
We can prove that \(\varepsilon(\lVert \,\cdot\, \rVert;x)\)
is continuous on a certain open convex cone in \(N^1_\RR(X)\). 
\begin{definition}[cf.\ {\cite[Definition 5.1]{ELMNP09}}]
  Let \(X\) be a complete scheme and let \(x \in X\) be a closed point.
  For \(\bk \in \{\QQ,\RR\}\), we denote by
  \[
    \operatorname{Big}_\bk^{\{x\}}(X) \subseteq N^1_\RR(X)
  \]
  the open convex cone consisting of classes \(D \in N^1_\bk(X)\) for which
  \(\varepsilon(\lVert D \rVert;x) > 0\).
  Note that \(N^1_\bk(X)\) is finite-dimensional by
  \citeleft\citen{Kee03}\citemid Theorem 3.6\citepunct \citen{Kee18}\citemid
  Theorem E2.2\citeright.
\end{definition}
\begin{remark}
  If \(X\) is a normal projective variety, Proposition \ref{prop:bbp23}
  implies that this cone coincides with the cone
  \[
    \Set[\big]{D \in N^1_\RR(X) \given x \notin \Bplus(D)}
  \]
  denoted by \(\operatorname{Big}_\RR^{\{x\}}(X)\) in
  \cite[Definition 5.1]{ELMNP09} since given a birational morphism \(f\colon
  X' \to X\) and an \(\RR\)-numerical equivalence \(f^*D \equiv_\RR A+E\) as in
  Definition \ref{def:movingsesh}, pulling back to the normalization of \(X'\) yields
  an \(\RR\)-numerical equivalence of the same form on a normal projective
  variety birational to \(X\).
\end{remark}
\par The following result says that \(\varepsilon(\lVert \,\cdot\, \rVert;x)\)
is continuous on \(\operatorname{Big}_\RR^{\{x\}}(X)\) for complete varieties
over arbitrary fields.
For smooth complex projective varieties, the stronger statement that \(D \mapsto
\varepsilon(\lVert \,\cdot\, \rVert;x)\) is continuous on the entire space
\(N^1_\RR(X)\) holds
\cite[Theorem 6.2]{ELMNP09} (see Remark \ref{rem:elmnp62}).
\begin{corollary}\label{cor:movingseshcont}
  Let \(X\) be a complete variety and let \(x \in X\) be a closed point.
  Then, the function
  \[
    \begin{tikzcd}[row sep=0,column sep=1.475em]
      \operatorname{Big}_\RR^{\{x\}}(X) \rar & \RR_{>0}\\
      D \rar[mapsto] & \varepsilon\bigl(\lVert D \rVert;x\bigr)
    \end{tikzcd}
  \]
  is Lipschitz continuous on each compact subset of
  \(\operatorname{Big}_\RR^{\{x\}}(X)\).
  In particular, \(D \mapsto \varepsilon(\lVert D \rVert;x)\) is continuous on
  \(\operatorname{Big}_\RR^{\{x\}}(X)\).
\end{corollary}
\begin{proof}
  This follows from statements \((\ref{prop:elmnp0963i})\),
  \((\ref{prop:elmnp0963ii})\), and \((\ref{prop:elmnp0963v})\) in
  Proposition \ref{prop:elmnp0963} and continuity properties of convex functions
  \cite[Theorem 2.2]{Gru07} (see \cite[Proof of Theorem 5.2\((a)\) and Remark
  5.4]{ELMNP09}).
\end{proof}
\begin{remark}\label{rem:elmnp62}
  For smooth complex projective varieties, Ein, Lazarsfeld,
  Musta\c{t}\u{a},
  Nakamaye, and Popa prove the stronger statement that \(\varepsilon(\lVert
  \,\cdot\, \rVert;x)\) is continuous on \(N^1_\RR(X)\) \cite[Theorem
  6.2]{ELMNP09}.
  Proving this stronger statement would require extending the continuity results
  on restricted volume functions in \cite[\S5]{ELMNP09} to our setting.
  As far as we are aware, \cite[Theorem 6.2]{ELMNP09} is not known to hold
  even for smooth projective varieties
  over algebraically closed fields of positive characteristic.
\end{remark}
\subsection{Seshadri constants and generation of jets}
We can characterize Seshadri constants and moving Seshadri constants
in terms of generation of jets.
Generation of jets is defined as follows.
We follow the convention for \(s(L;x) = -\infty\) from \cite[p.\ 96]{Dem92},
which differs from the convention \(s(L;x) = -1\) from \cite[Definition
5.1.16]{Laz04a} and \cite[Definition 5.1]{FM21}.
Our convention is chosen so that the superadditivity property
\eqref{eq:superadditivity} below always holds.
\begin{citeddef}[{\cite[p.\ 96]{Dem92}}]\label{def:demaillys}
  Let \(X\) be a complete scheme and let \(L\) be a \(\ZZ\)-invertible sheaf on
  \(X\).
  Consider a closed point \(x \in X\) and denote by \(\fm_x \subseteq \cO_X\)
  the coherent ideal sheaf defining \(x\).
  For every integer \(\ell \ge 0\), we say that \textsl{\(L\) generates
  \(\ell\)-jets at \(x\)} if the restriction map
  \[
    H^0\bigl(X,\cO_X(L)\bigr) \longrightarrow H^0\bigl(X,\cO_X(L)
    \otimes_{\cO_X}
    \cO_X/\fm_x^{\ell+1}\bigr)
  \]
  is surjective.
  We denote by \(s(L;x)\) the largest integer \(\ell \ge 0\) such that
  \(L\) generates \(\ell\)-jets at \(x\).
  If no such \(\ell\) exists, then we set \(s(L;x) = -\infty\).
\end{citeddef}
In \cite[(6.3)]{Dem92}, Demailly introduced the following version of Seshadri
constants defined using generation of jets, which he denoted by \(\sigma(L,x)\).
\begin{citeddef}[{\citeleft\citen{Dem92}\citemid (6.3)\citeright}]
  \label{def:jetsepsesh}
  Let \(X\) be a complete scheme and let \(L\) be a Cartier divisor on \(X\).
  Consider a closed point \(x \in X\).
  We set
  \begin{equation}\label{eq:movingsesh}
    \varepsilon_{\mathrm{jet}}\bigl( \lVert L \rVert;x \bigr) \coloneqq
    \limsup_{m \to \infty} \frac{s(mL;x)}{m}
  \end{equation}
  if \(s(mL;x) \ge 0\) for some \(m > 0\), and set
  \(\varepsilon_{\mathrm{jet}}( \lVert L \rVert;x) = 0\) otherwise.
  For \(\ZZ\)-invertible sheaves \(L_1,L_2\) on \(X\), we have the superadditivity
  property \cite[p.\ 97]{Dem92} (see \cite[Lemma 5.4]{FM21} for a proof)
  \begin{equation}\label{eq:superadditivity}
    s(L_1;x) + s(L_2;x) \le s(L_1 + L_2;x).
  \end{equation}
  Thus, Fekete's lemma \cite[Part I, Chapter 3, No.\ 98]{PS98} implies that
  \begin{equation}\label{eq:feketeappz}
    \limsup_{m \to \infty} \frac{s(mL;x)}{m} =
    \lim_{m \to \infty} \frac{s(mL;x)}{m}
    = \sup_{m \ge1} \frac{s(mL;x)}{m}.
  \end{equation}
  We can therefore extend the definition of moving Seshadri constants to
  \(\QQ\)-invertible sheaves \(D\) on \(X\) by setting
  \[
    \varepsilon_{\jet}\bigl( \lVert D \rVert;x \bigr) \coloneqq
    \frac{1}{n} \cdot \varepsilon_{\jet}\bigl( \lVert L \rVert;x \bigr)
  \]
  where \(L\) is a \(\ZZ\)-invertible sheaf on \(X\)
  such that \(L\) maps to \(nD\) under the map
  \(\Pic(X) \to \Pic_\QQ(X)\).
  This definition does not depend on the choice of \(n\) or \(L\):
  By \eqref{eq:feketeappz},
  we can pass to a subsequence consisting of all sufficiently divisible \(m\) to
  compute 
  \eqref{eq:movingsesh}, and then apply Lemma
  \ref{lem:sbwelldef}\((\ref{lem:baseidealinvsheafiso})\).
\end{citeddef}
We prove some basic properties of this constant.
\begin{proposition}\label{prop:jetseshelem}
  Let \(X\) be a complete scheme over a field \(k\) and let \(D\) be an
  \(\QQ\)-invertible sheaf on \(X\).
  Consider a closed point \(x \in X\).
  \begin{enumerate}[label=\((\roman*)\),ref=\roman*]
    \item\label{prop:jetseshvol}
      Suppose that \(X\) is a complete variety.
      We then have
      \[
        \varepsilon_{\jet}\bigl(\lVert D \rVert;x\bigr) \le
        \biggl(\frac{\vol_{X \vert V}(D)}{[k(x) : k] \cdot
        e(\cO_{V,x})}\biggr)^{1/{\dim(V)}}
      \]
      for every
      positive-dimensional subvariety \(V \subseteq X\) containing \(x\), where
      \(\vol_{X \vert V}(D)\) is defined as in \emph{\cite[Definition
      2.1]{ELMNP09}}.
    \item\label{prop:jetseshnumequiv}
      Suppose that one of the following conditions holds.
      \begin{itemize}
        \item \(X\) is projective and \(x \notin \Bplus(D)\).
        \item \(X\) is normal and \(\varepsilon(\lVert D \rVert;x) > 0\).
      \end{itemize}
      If \(E\) is a \(\QQ\)-invertible sheaf such that \(D \equiv_\QQ E\), then
      \(\varepsilon_{\textup{jet}}(\lVert D \rVert;x) =
      \varepsilon_{\textup{jet}}(\lVert E \rVert;x)\).
    \item\label{prop:jetseshhomog}
      \(\varepsilon_{\textup{jet}}(\lVert \lambda D \rVert;x) = \lambda \cdot
      \varepsilon_{\textup{jet}}(\lVert D \rVert;x)\) for every positive
      rational number \(\lambda\).
    \item\label{prop:jetseshadd}
      Let \(D_1,D_2\) be \(\QQ\)-invertible sheaves such that
      \(s(m_1D_1;x) \ge 0\) and \(s(m_2D_2;x) \ge 0\) for some \(m_1,m_2 \in
      \ZZ_{>0}\).
      Then,
      \[
        \varepsilon_{\textup{jet}}\bigl(\lVert D_1 + D_2 \rVert;x) \ge
        \varepsilon_{\textup{jet}}\bigl(\lVert D_1 \rVert;x\bigr) +
        \varepsilon_{\textup{jet}}\bigl(\lVert D_2 \rVert;x\bigr).
      \]
  \end{enumerate}
\end{proposition}
\begin{proof}
  \par \((\ref{prop:jetseshhomog})\) holds by definition.
  It therefore suffices to show the remaining properties for \(\ZZ\)-invertible
  sheaves.\medskip
  \par We prove \((\ref{prop:jetseshvol})\) when \(D\) is a
  \(\ZZ\)-invertible sheaf.
  Since \((\ref{prop:jetseshvol})\) trivially holds when
  \(\varepsilon_{\textup{jet}}(\lVert D \rVert;x) =
  0\), we may assume that \(\varepsilon_{\textup{jet}}(\lVert D \rVert;x) > 0\).
  With notation as in \cite[Definition 2.1]{ELMNP09}, we have
  \begin{align*}
    \frac{\vol_{X \vert V}(D)}{[k(x):k] \cdot e(\cO_{V,x})} &= \lim_{m \to \infty}
    \frac{h^0\bigl(X \vert V,\cO_X(mD)\bigr)}{m^{\dim(V)}/(\dim(V))!}
    \cdot \lim_{\ell \to \infty} \frac{(\ell+1)^{\dim(V)}/(\dim
    (V))!}{h^0\bigl(V,\cO_V(mD) \otimes \cO_V/\fm_x^{\ell+1}\bigr)}\\
    &= \lim_{m \to \infty}
    \frac{h^0\bigl(X \vert V,\cO_X(mD)\bigr)}{h^0\bigl(V,\cO_V(mD) \otimes
    \cO_V/\fm_x^{s(mD;x)+1}\bigr)} \cdot
    \biggl(\frac{s(mD;x)+1}{m}\biggr)^{\dim(V)},
  \end{align*}
  where the second equality follows from setting \(\ell = s(mD;x)\) and
  the fact that \(s(mD;x) \to \infty\) as \(m \to \infty\).
  By definition of \(s(mD;x)\) and the commutativity of the diagram
  \[
    \begin{tikzcd}
      H^0\bigl(X,\cO_X(mD)\bigr) \rar\dar & H^0\bigl(X,\cO_X(mD)
        \otimes_{\cO_X}
      \cO_X/\fm_x^{\ell+1}\bigr)\dar[twoheadrightarrow]\\
      H^0\bigl(V,\cO_V(mD)\bigr) \rar & H^0\bigl(V,\cO_V(mD) \otimes_{\cO_X}
      \cO_V/\fm_x^{\ell+1}\bigr)
    \end{tikzcd}
  \]
  we have \(h^0(X \vert V,\cO_X(mD)) \ge h^0(V,\cO_V(mD) \otimes
  \cO_V/\fm_x^{s(mD;x)+1})\).
  Thus,
  \[
    \frac{\vol_{X \vert V}(D)}{[k(x):k] \cdot e(\cO_{V,x})}
    \ge \lim_{m \to \infty}
    \biggl(\frac{s(mD;x)+1}{m}\biggr)^{\dim V}
    = \varepsilon_{\textup{jet}}\bigl(\lVert D \rVert;x)^{\dim(V)}.\medskip
  \]
  \par We now prove \((\ref{prop:jetseshnumequiv})\) for \(\ZZ\)-invertible
  sheaves.
  We first reduce the statement for normal \(X\) to the statement for projective
  \(X\).
  Since \(\varepsilon(\lVert D \rVert;x) > 0\), there exists a birational
  morphism \(f\colon X' \to X\) from a projective scheme
  as in Definition \ref{def:movingsesh}, in which case \(x' \coloneqq f^{-1}(x)
  \notin \Bplus(f^*D)\).
  We then have the commutative diagram
  \begin{equation}\label{eq:smdwithnormal}
    \begin{tikzcd}
      H^0\bigl(X,\cO_X(mD)\bigr) \rar \dar[sloped]{\sim} & H^0\bigl(X,\cO_X(mD)
      \otimes_{\cO_X} \cO_X/\fm_x^{\ell+1}\bigr) \dar[sloped]{\sim}\\
      H^0\bigl(X',\cO_{X'}(m\,f^*D)\bigr) \rar & H^0\bigl(X',\cO_{X'}(m\,f^*D)
      \otimes_{\cO_{X'}} \cO_{X'}/\fm_{x'}^{\ell+1}\bigr)
    \end{tikzcd}
  \end{equation}
  where the left vertical map is an
  isomorphism by the fact that \(X\) is normal, and the right vertical map is an
  isomorphism by the fact that \(f\) is an isomorphism at \(x\).
  This shows that
  \[
    \varepsilon_\jet\bigl(\lVert D \rVert;x\bigr) =
    \varepsilon_\jet\bigl(\lVert f^*D \rVert;x'\bigr).
  \]
  Repeating the argument for \(E\), we see that to prove
  \((\ref{prop:jetseshnumequiv})\) for normal \(X\), it suffices to show that
  \[
    \varepsilon_\jet\bigl(\lVert f^*D \rVert;x'\bigr) =
    \varepsilon_\jet\bigl(\lVert f^*E \rVert;x'\bigr).
  \]
  \par We now show \((\ref{prop:jetseshnumequiv})\) for \(\ZZ\)-invertible
  sheaves when \(X\) is projective and \(x \notin \Bplus(D)\).
  First, recall that \(\Bplus(D)\) only depends on the numerical class of
  \(D\),
  and hence \(x \notin \Bplus(E)\) as well.
  By assumption, there exists a numerically trivial \(\ZZ\)-invertible sheaf
  \(P\) such that \(D \sim E + P\), and Theorem \ref{prop:kur13prop27}
  implies that there
  exists a positive integer \(j\) such that \(\cO_X(jD+iP)\) is globally generated
  away from \(\Bplus(D)\) for all integers \(i\).
  For every \(m\), we therefore see that
  \[
    s(mD;x) \le s\bigl( (m+j)D+(m+j)P;x\bigr) = s\bigl( (m+j) E;x\bigr)
  \]
  by setting \(i = m+j\), where the inequality follows from
  \eqref{eq:superadditivity}
  since \(\cO_X(jD+(m+j)P)\) separates \(0\)-jets at \(x\).
  Dividing by \(m\) and taking limits, we see that
  \begin{align*}
    \varepsilon_{\textup{jet}}\bigl(\lVert D \rVert;x\bigr)
    &\le \varepsilon_{\textup{jet}}\bigl(\lVert E \rVert;x\bigr).
    \intertext{Repeating the argument above after switching the roles of \(D\)
    and \(E\), we have}
    \varepsilon_{\textup{jet}}\bigl(\lVert D \rVert;x\bigr) &=
    \varepsilon_{\textup{jet}}\bigl(\lVert E \rVert;x\bigr).\medskip
  \end{align*}
  \par Finally, \((\ref{prop:jetseshadd})\) follows from
  \eqref{eq:superadditivity}.
\end{proof}
By repeating the proof of Corollary \ref{cor:movingseshcont} with Proposition
\ref{prop:jetseshelem} as input, we obtain the
following:
\begin{corollary}
  Let \(X\) be a complete variety and let \(x \in X\) be a closed point.
  Consider the function
  \[
    \begin{tikzcd}[row sep=0,column sep=1.475em]
      \Set[\big]{D \in \Pic_\QQ(X) \given \varepsilon_\jet\bigl(\lVert D
      \rVert;x\bigr) > 0} \rar & \RR_{>0}\\
      D \rar[mapsto] & \varepsilon_\jet\bigl(\lVert D \rVert;x\bigr)\mathrlap{.}
    \end{tikzcd}
  \]
  For every \(D \in \Pic_\QQ(X)\) such that \(\varepsilon_\jet(\lVert D
  \rVert;x) > 0\) and every finite-dimensional \(\QQ\)-subspace \(V \subseteq
  \Pic_\QQ(X)\) containing \(D\), there exists a compact subset of \(V\)
  containing \(D\) on which \(D \mapsto \varepsilon_\jet(\lVert D
  \rVert;x)\) is Lipschitz continuous.
  In particular, \(D \mapsto \varepsilon_\jet(\lVert D \rVert;x)\) is continuous
  on \(V\).
  \par If \(X\) is projective, then the same statements hold
  for the function
  \[
    \begin{tikzcd}[row sep=0,column sep=1.475em]
      \Set[\big]{D \in N^1_\QQ(X) \given x \notin \Bplus(D)} \rar & \RR_{>0}\\
      D \rar[mapsto] & \varepsilon_\jet\bigl(\lVert D \rVert;x\bigr)
    \end{tikzcd}
  \]
  where \(V = N^1_\QQ(X)\).
  If \(X\) is normal, then the same statements hold
  for the function
  \[
    \begin{tikzcd}[row sep=0,column sep=1.475em]
      \Set[\big]{D \in N^1_\QQ(X) \given \varepsilon\bigl(\lVert D
      \rVert;x\bigr) > 0} \rar & \RR_{>0}\\
      D \rar[mapsto] & \varepsilon_\jet\bigl(\lVert D \rVert;x\bigr)
    \end{tikzcd}
  \]
  where \(V = N^1_\QQ(X)\).
\end{corollary}
Seshadri constants can be characterized in terms of generation of jets.
The case of ample invertible sheaves on smooth complex projective
varieties is due to Demailly \cite[Theorem 6.4]{Dem92}.
The case of ample invertible sheaves on projective varieties over algebraically
closed fields is due to Fulger and the author of this paper \cite[Theorem
5.3]{FM21}.
The case of ample invertible sheaves on projective varieties over arbitrary
fields and regular points
\(x \in X\) is due to Das and Waldron \cite[Proposition A.11]{DW22}.
\begin{citedthm}[{\cite[Theorem 5.3 and Proposition 5.5]{FM21}}]
  \label{prop:seshjetample}
  Let \(X\) be a complete scheme and let
  \(D\) be a nef \(\QQ\)-invertible sheaf.
  Consider a closed point \(x \in X\).
  Then, we have
  \begin{equation}\label{eq:seshjetample}
    \varepsilon(D;x) \ge \varepsilon_{\jet}\bigl(\lVert D \rVert;x\bigr).
  \end{equation}
  Equality holds in \eqref{eq:seshjetample} if \(D\) is ample.
\end{citedthm}
\begin{proof}
  While \cite[Theorem 5.3 and Proposition 5.5]{FM21} are stated
  for projective schemes over algebraically closed fields, the proofs of these
  statements work for complete schemes over arbitrary fields.
  Note also that in \cite{FM21}, the definition of the Seshadri constant
  \cite[Definition 3.5]{FM21} is
  different than the one we use (Definition \ref{def:seshadriconstant}).
  The two definitions coincide when \(D\) is nef \cite[Proposition
  3.7]{FM21}.
\end{proof}

\subsection{Nakamaye's definition of moving Seshadri constants}
Next, we define moving Seshadri constants.
Nakamaye defined moving Seshadri constants for smooth complex projective
varieties in \cite{Nak03} as a generalization of the Seshadri
constant that makes sense for divisors that are not necessarily nef.
We have adapted his definition to work for complete schemes over
arbitrary fields.
This constant is denoted by \(\varepsilon_m(x,\,\cdot\,)\) in \cite{Nak03}.
The notation \(\varepsilon'(\lVert \,\cdot\, \rVert;x)\) is from \cite{ELMNP09}.
\par We start with the definition of \(\varepsilon'(\lVert \,\cdot\, \rVert;x)\)
for \(\ZZ\)-invertible sheaves.
\begin{definition}[{cf.\ \citeleft\citen{Nak03}\citemid Definition 0.4\citepunct
  \citen{ELMNP09}\citemid (41)\citeright}]\label{def:seshnakamaye}
  Let \(X\) be a complete scheme over a field \(k\)
  and let \(L\) be a \(\ZZ\)-invertible sheaf on \(X\).
  Consider a closed point \(x \in X\)
  such that \(x \notin \SB(L)\).
  For every \(m > 0\) such that \(x \notin \Bs(\lvert mL
  \rvert)\), let \(\pi_m\colon X_m \to X\) be the blowup of
  \(\fb(\lvert mL \rvert)\).
  Under \(\pi_m\), the surjection
  \[
    H^0\bigl(X,\cO_X(mL)\bigr) \otimes_k \cO_{X}(-mL) \longtwoheadrightarrow
    \fb\bigl(\lvert mL \rvert\bigr)
  \]
  pulls back to
  \begin{equation}\label{eq:defoffixedpart}
    H^0\bigl(X,\cO_{X}(mL)\bigr) \otimes_k
    \cO_{X_m}\bigl(-\pi^*_m(mL)\bigr)
    \longtwoheadrightarrow \pi_m^{-1}\bigl(\fb\bigl(\lvert mL \rvert\bigr)\bigr)
    \cdot \cO_{X_m} = \cO_{X_m}(-F_m).
  \end{equation}
  Twisting by \(\pi_m^*(mL)\) and setting \(M_m = \pi_m^*(mL) - F_m\),
  we see that \(\cO_{X_m}(M_m)\) is globally generated, and we have
  \[
    \pi_m^*(mL) = M_m + F_m.
  \]
  We then set
  \begin{equation}\label{eq:naklimsup}
    \varepsilon'\bigl(\lVert L \rVert;x\bigr) \coloneqq \limsup_{m \to \infty}
    \frac{\varepsilon\bigl(M_m;\pi_m^{-1}(x)\bigr)}{m}
  \end{equation}
  where the \(m\) range over all \(m \ge 1\) such that
  \(x \notin \Bs(\lvert mL \rvert)\).
  If \(x \in X\) is a closed point such that \(x \in \SB(L)\), we set
  \(\varepsilon'(\lVert L \rVert;x) = 0\).
\end{definition}
\begin{remark}
  For smooth complex projective varieties, this definition matches those in
  \cite{Nak03,ELMNP09}.
  In \cite{Nak03,ELMNP09}, one uses a log resolution \(\pi_m\colon X_m \to X\) of
  the base ideal that is an isomorphism at \(x\) instead of the blowup of the
  base ideal.
  This does not change the values of the Seshadri constants
  \(\varepsilon(M_m;\pi_m^{-1}(x))\) because of Corollary \ref{cor:seshbirat}.
\end{remark}
We show that the limit supremum in \eqref{eq:naklimsup} is in fact a limit.
The key input is the following superadditivity result.
\begin{lemma}\label{lem:naksuperadd}
  Let \(X\) be a complete scheme and let \(L,L'\) be \(\ZZ\)-invertible
  sheaves on \(X\).
  Let \(x \in X\) be a closed point such that
  \[
    x \notin \Bs\bigl(\lvert L \rvert\bigr)_\red \cup \Bs\bigl(\lvert L'
    \rvert\bigr)_\red.
  \]
  Let \(\pi\colon X \to X\), \(\pi'\colon X' \to X\), and
  \(\pi''\colon X'' \to X\) be the blowups of \(\fb(\lvert L \rvert)\),
  \(\fb(\lvert L' \rvert)\), and \(\fb(\lvert L+L' \rvert)\), respectively, and
  write
  \begin{align*}
    \pi^{*}L &= M + F\\
    \pi^{\prime*}L' &= M' + F'\\
    \pi^{\prime\prime*}(L+L') &= M'' + F''
  \end{align*}
  where \(F,F',F''\) are defined as in \eqref{eq:defoffixedpart}.
  Then, we have
  \begin{equation}\label{eq:seshmovingpartsuperadd}
    \varepsilon\bigl(M'';\pi^{\prime\prime-1}(x)\bigr) \ge 
    \varepsilon\bigl(M;\pi^{-1}(x)\bigr) +
    \varepsilon\bigl(M';\pi^{\prime-1}(x)\bigr).
  \end{equation}
\end{lemma}
\begin{proof}
  Note that \(x \notin \Bs(\lvert L+L' \rvert)_\red\).
  Let \(\tilde{\pi}\colon \tilde{X} \to X\) be the blowup of \(\fb(\lvert
  L\rvert) \cdot \fb(\lvert L'\rvert) \cdot \fb(\lvert L+L'\rvert)\).
  Then, \(\pi\) is a birational morphism from a complete scheme 
  that dominates \(\pi\), \(\pi'\), and \(\pi''\) and is an isomorphism
  at \(x\).
  We fix the following notation:
  \[
    \begin{tikzcd}[row sep=large]
      & \tilde{X}\arrow{dl}[swap]{\vphantom{\rho'}\rho}
      \arrow{d}[description]{\rho''}
      \arrow{dr}{\rho'}\\
      X\vphantom{X'}\arrow{dr}[swap]{\vphantom{\pi'}\pi}
      & X''\arrow{d}[description]{\pi''}
      & X' \arrow{dl}{\pi'}\\
      & X\mathrlap{.}
    \end{tikzcd}
  \]
  \par We claim that
  the Cartier divisor
  \[
    \tilde{F} \coloneqq \rho^*F + \rho^{\prime*}F' - \rho^{\prime\prime*}F''
  \]
  is effective.
  Consider the commutative diagram
  \[
    \begin{tikzcd}[column sep=1.475em]
      H^0\bigl(X,\cO_X(L)\bigr) \otimes_k H^0\bigl(X,\cO_X(L')\bigr) \otimes_k
      \cO_{\tilde{X}}\bigl(-\tilde{\pi}^*\bigl(L+L'\bigr)\bigr) \rar\dar
      & \cO_{\tilde{X}}\dar[equals]\\
      H^0\bigl(X,\cO_X\bigl(L+L'\bigr)\bigr) \otimes_k
      \cO_{\tilde{X}}\bigl(-\tilde{\pi}^*\bigl(L+L'\bigr)\bigr)
      \rar & \cO_{\tilde{X}}
    \end{tikzcd}
  \]
  where the left vertical map is obtained by pulling back the multiplication map
  on global sections and twisting by \(-\tilde{\pi}^*(L+L')\).
  Pulling back and tensoring the two surjections \eqref{eq:defoffixedpart} for \(L\)
  and \(L'\), we see that the top horizontal map has image
  \(\cO_{X'}(-(\rho^*F + \rho^{\prime*}F'))\).
  Pulling back the surjection \eqref{eq:defoffixedpart} for \(L+L'\), we see that the
  bottom horizontal map has image \(\cO_{X'}(-\rho^{\prime\prime*}F'')\).
  By the commutativity of this diagram, this shows that
  \[
    \cO_{X'}\bigl(-(\rho^*F + \rho^{\prime*}F')\bigr) \hooklongrightarrow
    \cO_{X'}(-\rho^{\prime\prime*}F'')
  \]
  as sub-\(\cO_{\tilde{X}}\)-modules of \(\cO_{\tilde{X}}\).
  Thus, we have 
  \[
    \rho^{\prime\prime*}F'' \le \rho^*F + \rho^{\prime*}F',
  \]
  and hence \(\tilde{F}\) is effective.
  \par We now prove \eqref{eq:seshmovingpartsuperadd}.
  First, we have
  \begin{align*}
    \varepsilon\bigl(M;\pi^{-1}(x)\bigr)
    &= \varepsilon\bigl(\rho^*M;\tilde{\pi}^{-1}(x)\bigr)\\
    \varepsilon\bigl(M';\pi^{\prime-1}(x)\bigr)
    &= \varepsilon\bigl(\rho^{\prime*}M';\tilde{\pi}^{-1}(x)\bigr)\\
    \varepsilon\bigl(M'';\pi^{\prime\prime-1}(x)\bigr)
    &= \varepsilon\bigl(\rho^{\prime\prime*}M'';\tilde{\pi}^{-1}(x)\bigr)
  \end{align*}
  by Corollary \ref{cor:seshbirat}.
  Since \(\tilde{\pi}^*(L+L') = \tilde{\pi}^*L+\tilde{\pi}^*L'\) implies
  \begin{align*}
    \rho^{\prime\prime*}M'' + \rho^{\prime\prime*}F'' &= \rho^{*}M
    + \rho^{*}F + \rho^{\prime*}M' + \rho^{\prime*}F',
    \intertext{solving for \(\rho^{\prime\prime*}M''\) yields}
    \rho^{\prime\prime*}M'' &= \rho^{*}M + \rho^{\prime*}M' + \tilde{F}
  \end{align*}
  where as shown above, \(\tilde{F}\) is an effective Cartier divisor.
  We therefore have
  \begin{align*}
    \varepsilon\bigl(M'';\pi^{\prime\prime-1}(x)\bigr)
    &= \varepsilon\bigl(\rho^{\prime\prime*}M'';\tilde{\pi}^{-1}(x)\bigr)\\
    &= \varepsilon\bigl(\rho^{*}M + \rho^{\prime*}M' +
    \tilde{F};\tilde{\pi}^{-1}(x)\bigr)\\
    &\ge \varepsilon\bigl(\rho^{*}M +
    \rho^{\prime*}M';\tilde{\pi}^{-1}(x)\bigr)\\
    &\ge \varepsilon\bigl(\rho^*M;\tilde{\pi}^{-1}(x)\bigr) + 
    \varepsilon\bigl(\rho^{\prime*}M';\tilde{\pi}^{-1}(x)\bigr)\\
    &= \varepsilon\bigl(M;\pi^{-1}(x)\bigr)
    + \varepsilon\bigl(M';\pi^{\prime-1}(x)\bigr)
  \end{align*}
  where the three equalities hold by the equalities shown above, the first
  inequality holds by Proposition \ref{prop:laz519} and the fact that
  \(\tilde{\pi}^{-1}(x) \notin \Supp(\tilde{F})\),
  and the second inequality holds by definition
  of Seshadri constants.
\end{proof}
We can now show that the limit supremum in \eqref{eq:naklimsup} is in fact a
limit.
The case when \(X\) is a smooth complex projective variety is proved
in \cite{Nak03,ELMNP09}.
\begin{proposition}[cf.\ {\citeleft\citen{Nak03}\citemid p.\ 552\citepunct
  \citen{ELMNP09}\citemid p.\ 645\citeright}]\label{prop:naklimit}
  Let \(X\) be a complete scheme and let \(L\) be a \(\ZZ\)-invertible
  sheaf on \(X\).
  Consider a closed point \(x \in X\) such that \(x \notin \SB(L)\).
  With notation as in Definition \ref{def:seshnakamaye}, we have
  \[
    \varepsilon'\bigl(\lVert L \rVert;x\bigr)
    = \lim_{m \to \infty} \frac{\varepsilon\bigl(M_m;\pi_m^{-1}(x)\bigr)}{m}
    = \sup_{m > 0} \frac{\varepsilon\bigl(M_m;\pi_m^{-1}(x)\bigr)}{m}
  \]
  where the \(m\) range over all \(m \ge 1\) such that
  \(x \notin \Bs(\lvert mL \rvert)\).
  As a consequence, for every integer \(n \ge 1\), we have
  \[
    \varepsilon'\bigl(\lVert nL \rVert;x\bigr) = n \cdot
    \varepsilon'\bigl(\lVert L \rVert;x\bigr).
  \]
\end{proposition}
\begin{proof}
  By Fekete's lemma \cite[Part I, Chapter 3, No.\ 98]{PS98}, it suffices to show
  that the sequence \(\varepsilon(M_m;\pi_m^{-1}(x))\) is superadditive, i.e.,
  that
  \[
    \varepsilon\bigl(M_{m+n};\pi_{m+n}^{-1}(x)\bigr) \ge
    \varepsilon\bigl(M_{m};\pi_{m}^{-1}(x)\bigr) +
    \varepsilon\bigl(M_{n};\pi_{n}^{-1}(x)\bigr)
  \]
  for all \(m,n \ge 1\) such that
  \(x \notin
  \Bs(\lvert mL \rvert) \cup \Bs(\lvert nL \rvert)\).
  This inequality is a special case of Lemma \ref{lem:naksuperadd}.
\end{proof}
By Proposition \ref{prop:naklimit}, we can extend Definition \ref{def:seshnakamaye} to
\(\QQ\)-invertible sheaves.
\begin{definition}[{cf.\ \citeleft\citen{Nak03}\citemid Definition 0.4\citepunct
  \citen{ELMNP09}\citemid (41)\citeright}]
  Let \(X\) be a complete scheme
  and let \(D\) be a \(\QQ\)-invertible sheaf on \(X\).
  Let \(n \ge 1\) be an integer such that \(nD\) is integral, and let \(L\) be a
  \(\ZZ\)-invertible sheaf on \(X\) such that \(L\) maps to \(nD\) under the map
  \(\Pic(X) \to \Pic_\QQ(X)\).
  We then set
  \[
    \varepsilon'\bigl(\lVert D \rVert;x\bigr) \coloneqq \frac{1}{n} \cdot
    \varepsilon'\bigl(\lVert L \rVert;x\bigr).
  \]
  We prove that
  \(\varepsilon'(\lVert D \rVert;x)\) does not depend on the choice of
  \(n\) or \(L\).
  If \(x \in \SB(D)\), then \(x \in \SB(L)\) by Lemma
  \ref{lem:sbwelldef}\((\ref{lem:sbwelldefitem})\) for any choice of \(n\) or
  \(L\), and hence \(\varepsilon'(\lVert D \rVert;x) = 0\).
  If \(x \notin \SB(D)\), then by Proposition \ref{prop:naklimit}, we can
  compute \eqref{eq:naklimsup} by passing to a subsequence consisting of all
  sufficiently divisible \(m\), after which we may apply 
  Lemma \ref{lem:sbwelldef}.
\end{definition}
We show that \(\varepsilon(\lVert D \rVert;x) = \varepsilon'(\lVert D
\rVert;x)\) for normal complete schemes over arbitrary fields.
The case when \(X\) is a smooth complex projective variety is proved in
\cite{ELMNP09}.
\begin{proposition}[{cf.\ \cite[Proposition 6.4]{ELMNP09}}]\label{prop:elmnp64}
  Let \(X\) be a complete scheme over a field \(k\)
  and let \(D\) be a \(\QQ\)-invertible sheaf on
  \(X\).
  For every closed point \(x \in X\), we have
  \begin{equation}\label{eq:elmnp64ineq}
    \varepsilon\bigl(\lVert D \rVert;x\bigr) \ge \varepsilon'\bigl(\lVert D
    \rVert;x\bigr).
  \end{equation}
  Equality holds in \eqref{eq:elmnp64ineq}
  if \(X\) is normal.
\end{proposition}
\begin{proof}
  Since both constants are homogeneous with respect to taking rational multiples
  by Proposition \ref{prop:elmnp0963}\((\ref{prop:elmnp0963iii})\) and
  Proposition \ref{prop:naklimit}, respectively, we reduce to the case when
  \(D\) is replaced by a \(\ZZ\)-invertible sheaf \(L\) that maps to \(mD\) for
  some \(m > 0\).
  Note that \(\SB(D) = \SB(L)\) by Lemma
  \ref{lem:sbwelldef}\((\ref{lem:sbwelldefitem})\).
  We fix notation as in Definition \ref{def:seshnakamaye}.\medskip
  \begin{step}
    The inequality \eqref{eq:elmnp64ineq} holds
    when \(\varepsilon(\lVert L \rVert;x) = 0\).
  \end{step}
  It suffices to show that \(\varepsilon(M_m;\pi_m^{-1}(x)) = 0\) for every \(m
  > 0\).
  Let \(f_m\colon X'_m \to X_m\) be a birational morphism from a projective
  scheme that is an isomorphism at \(x\), which exists by
  \citeleft\citen{Con07}\citemid Corollary 2.6\citepunct \citen{Del10}\citemid
  Corollaire 1.4\citeright.
  We then have
  \[
    (\pi_m \circ f_m)^{-1}(x) \in \Bplus\bigl((\pi_m \circ f_m)^*L\bigr)
    \subseteq \Bplus(f_m^*M_m) \cup
    \Supp(f_m^*F_m)
  \]
  because an \(\RR\)-numerical equivalence \(f_m^*M_m \equiv_\RR A+E\) yields
  an \(\RR\)-numerical equivalence
  \[
    (\pi_m \circ f_m)^*L \equiv_\RR A+E+f_m^*F_m.
  \]
  We therefore see that
  \[
    \varepsilon\bigl(M_m;\pi_m^{-1}(x)\bigr) = \varepsilon\bigl(f_m^*M_m;(\pi_m
    \circ f_m)^{-1}(x)\bigr) = 0
  \]
  by Corollary \ref{cor:seshbirat} and \cite[Theorem 1.4]{Bir17}.\medskip
  \begin{step}
    The inequality \eqref{eq:elmnp64ineq} holds when \(\varepsilon(\lVert L
    \rVert;x) > 0\).
  \end{step}
  It suffices to show that
  \[
    \varepsilon\bigl(\lVert L \rVert;x\bigr) \ge
    \frac{\varepsilon\bigl(M_m;\pi_m^{-1}(x)\bigr)}{m}
  \]
  for all \(m\).
  Let \(f_m\colon X'_m \to X_m\) be a birational morphism from a projective
  scheme that is an isomorphism at \(x\), which exists by
  \citeleft\citen{Con07}\citemid Corollary 2.6\citepunct \citen{Del10}\citemid
  Corollaire 1.4\citeright.
  \par If \((\pi_m \circ f_m)^{-1}(x) \in \Bplus(f_m^*M_m)\), then 
  \[
    \varepsilon\bigl(M_m;\pi_m^{-1}(x)\bigr) = \varepsilon\bigl(f_m^*M_m;(\pi_m
    \circ f_m)^{-1}(x)\bigr) = 0
  \]
  by Corollary \ref{cor:seshbirat} and \cite[Theorem 1.4]{Bir17}.
  \par It therefore suffices to consider the case when
  \((\pi_m \circ f_m)^{-1}(x) \notin
  \Bplus(f_m^*M_m)\).
  Write \(f_m^*M_m \equiv_\RR A+E\) where \(A\) is an ample \(\QQ\)-Cartier
  divisor and \(E\) is an effective \(\RR\)-Cartier divisor such that \((\pi_m
  \circ f_m)^{-1}(x) \notin \Supp(E)\).
  For every \(n \ge 1\), we have
  \[
    f_m^*M_m \equiv_\RR \frac{n-1}{n} f_m^*M_m + \frac{1}{n}(A+E)
    = \biggl(\frac{n-1}{n} f_m^*M_m + \frac{1}{n}A \biggr) + \frac{1}{n}E
  \]
  where \(A_n \coloneqq \frac{n-1}{n} f_m^*M_m + \frac{1}{n}A\) is ample.
  We then have
  \[
    \varepsilon\bigl(\lVert L \rVert;x\bigr) \ge \frac{\varepsilon\bigl(A_n;(\pi_m
    \circ f_m)^{-1}(x)\bigr)}{m}
  \]
  by definition and the homogeneity of \(\varepsilon(\lVert L
  \rVert;x)\) (Proposition \ref{prop:elmnp0963}\((\ref{prop:elmnp0963iii})\)).
  Taking limits as \(n \to \infty\), we see that
  \[
    \lim_{n \to \infty} \varepsilon\bigl(A_n;(\pi_m
    \circ f_m)^{-1}(x)\bigr) \ge \varepsilon\bigl(f_m^*M_m;(\pi_m
    \circ f_m)^{-1}(x)\bigr) = \varepsilon\bigl(M_m;\pi_m^{-1}(x)\bigr)
  \]
  by the lower semi-continuity of Seshadri constants with respect to ample
  perturbations of a nef class \cite[Corollary 3.34]{FM21} and
  Corollary \ref{cor:seshbirat}.\medskip
  \begin{step}
    Equality in \eqref{eq:elmnp64ineq} holds when \(X\) is
    normal and \(x \in \SB(L)\).
  \end{step}
  In this case, we have \(\varepsilon'(\lVert L
  \rVert;x) = 0\) by definition.
  We proceed by contradiction.
  Suppose that \(\varepsilon(\lVert L \rVert;x) > 0\), in which case there
  exists a birational morphism \(f\colon X' \to X\) and an
  \(\RR\)-numerical equivalence \(f^*L \equiv_\RR A+E\)
  as in Definition \ref{def:movingsesh}.
  We then have \(f^{-1}(x) \notin \Bplus(f^*L)\) by \eqref{eq:bpluselmnpdef}.
  Since \(\Bplus(f^*L) \supseteq \SB(f^*L)\), we see that
  \[
    f^{-1}(x) \notin \SB(f^*L) = f^{-1}\bigl(\SB(L)\bigr)
  \]
  by Lemma \ref{lem:basebirat}.
  Thus, we have \(x \notin \SB(L)\), a contradiction.\medskip
  \begin{step}\label{step:elmnp64four}
    Equality in \eqref{eq:elmnp64ineq} holds when \(X\)
    is normal and \(x \notin \SB(L)\).
  \end{step}
  If \(\varepsilon(\lVert L \rVert;x) = 0\), there is nothing to show.
  We may therefore assume that \(\varepsilon(\lVert L \rVert;x) > 0\).
  Consider a birational morphism \(f\colon X' \to X\) from a projective scheme
  and an \(\RR\)-numerical equivalence \(f^*L \equiv_\RR A+E\) as in Definition
  \ref{def:movingsesh}.
  \par We first perturb the coefficients in \(E\) and replace \(A\) to obtain a
  decomposition
  \[
    f^*L = A' + E'
  \]
  as \(\QQ\)-Cartier divisors where \(A'\) is ample and \(E'\) is effective such
  that \(f^{-1}(x) \notin \Supp(E')\).
  Write \(E = \sum e_iE_i\) for \(e_i \in \RR_{\ge0}\).
  By the openness of the ample cone \citeleft\citen{Kee03}\citemid Theorem
  3.9\citepunct \citen{Kee18}\citemid Theorem E2.2\citeright,
  for rational numbers \(e_i'\)
  such that \(\lvert e_i - e_i' \rvert \ll 1\),
  we know that
  \[
    f^*L \equiv_\RR \biggl(A + \sum (e_i-e_i')E_i \biggr)
    + \sum_i e_i'E_i
  \]
  and \(A + \sum (e_i-e_i')E_i\) is ample.
  Thus, the \(\QQ\)-Cartier divisor \(A' \coloneqq f^*L - \sum_i e_i'E_i\) is
  ample and satisfies
  \(f^*L = A' + E'\)
  for \(E' \coloneqq \sum_i e_i'E_i\).
  \par For each \(m\), we find a blowup \(h_m\colon X_m' \to X\) that dominates
  both \(X'\) and \(X_m\).
  By \cite[Premi\`ere partie, Corollaire 5.7.12]{RG71} (see also
  \cite[Theorem 2.11]{Con07}),
  we can find a blowup \(X'' \to X'\) that is an
  isomorphism at \(x\) such that \(X'' \to X\) is also a blowup along a coherent
  ideal sheaf \(\mathcal{I} \subseteq \cO_X\).
  The normalized blowup of \(X\) along the product of ideals
  \(\mathcal{I} \cdot \fb(\lvert mL \rvert)\) yields a blowup \(h_m\colon X'_m
  \to X\) from a normal projective scheme fitting into the commutative diagram
  \[
    \begin{tikzcd}[column sep=scriptsize]
      & X'_m \arrow{dl} \arrow{ddl}{g_m}
      \arrow{ddr}[swap]{\rho_m} \arrow{ddd}[pos=0.6]{h_m}\\
      X'' \dar[swap]{g}\\
      X' \arrow{dr}[swap]{f} & & X_m \arrow{dl}{\pi_m}\\
      & X
    \end{tikzcd}
  \]
  of birational morphisms.
  Note that by construction, the morphism \(h_m\) is an isomorphism at \(x\).
  \par For \(m\) sufficiently large and divisible, we now have the
  decomposition
  \[
    f^*(mL) = mA' + mE'
  \]
  where \(mA'\) and \(mE'\) are Cartier divisors.
  After possibly replacing \(m\) by a multiple, we may assume that \(mA'\) is
  free.
  Since in the decomposition
  \[
    h_m^*(mL) = g_m^*(mA') + g_m^*(mE')
  \]
  the Cartier divisor \(g_m^*(mA')\) is free, and since the fact that \(X\) is
  normal implies
  \begin{align*}
    \bigl\lvert h_m^*(mL) \bigr\rvert &= \bigl\lvert \rho_m^*M_m \bigr\rvert +
    \rho_m^*F_m
    \intertext{is the decomposition of \(\lvert h_m^*(mL) \rvert\) into its
    mobile and fixed parts, we see that \(g_m^*(mE') \ge \rho_m^*F_m\).
    Thus, we have}
    \rho_m^*M_m - g_m^*mA' &\sim g_m^*(mE') - \rho_m^*F_m \eqqcolon G_m
  \end{align*}
  where \(G_m\) is an effective Cartier divisor on \(X_m'\).
  Since \(\Supp(G_m) \subseteq \Supp(g_m^*(mE'))\),
  we know that \(h_m^{-1}(x) \notin \Supp(G_m)\), and hence
  \begin{align*}
    \frac{\varepsilon\bigl(M_m;\pi_m^{-1}(x)\bigr)}{m} 
    = \frac{\varepsilon\bigl(\rho_m^*M_m;h_m^{-1}(x)\bigr)}{m} &\ge
    \varepsilon\bigl(g_m^*A';h_m^{-1}(x)\bigr) =
    \varepsilon\bigl(A';f^{-1}(x)\bigr)
    \intertext{where the equalities hold by Corollary \ref{cor:seshbirat} and the inequality
    holds by Proposition \ref{prop:laz519} and the fact that \(h_m^{-1}(x) \notin
    \Supp(G_m)\).
    Taking the limit as \(E' \to E\), we see that \([A'] \to [A]\) as numerical
    classes in \(N^1_\RR(X')\).
    By the continuity of Seshadri constants on the ample cone \cite[Corollary
    3.34]{FM21}, we obtain}
    \frac{\varepsilon\bigl(M_m;\pi_m^{-1}(x)\bigr)}{m} &\ge
    \varepsilon\bigl(A;f^{-1}(x)\bigr).\qedhere
  \end{align*}
\end{proof}
We can now prove properties of \(\varepsilon'(\lVert D \rVert;x)\) analogous to
Proposition \ref{prop:elmnp0963}.
\begin{proposition}[cf.\ {\citeleft\citen{Nak03}\citemid p.\ 552\citeright}]
  Let \(X\) be a complete scheme over a field \(k\) and 
  let \(D\) be an \(\QQ\)-invertible sheaf on \(X\).
  Consider a closed point \(x \in X\).
  \begin{enumerate}[label=\((\roman*)\),ref=\roman*]
    \item\label{prop:nakelmnp0963i} Suppose that \(X\) is a complete variety.
      We then have
      \[
        \varepsilon'\bigl(\lVert D \rVert;x\bigr) \le
        \biggl(\frac{\vol_X(D)}{[k(x):k] \cdot
        e(\cO_{X,x})}\biggr)^{1/{\dim(X)}}.
      \]
    \item\label{prop:nakelmnp0963ii}
      Suppose that \(X\) is normal.
      If \(D\) and \(E\) are \(\QQ\)-numerically equivalent \(\QQ\)-invertible
      sheaves, then \(\varepsilon'(\lVert D \rVert;x) = \varepsilon'(\lVert E
      \rVert;x)\).
    \item\label{prop:nakelmnp0963iii}
      \(\varepsilon'(\lVert \lambda D \rVert;x) = \lambda \cdot
      \varepsilon'(\lVert D \rVert;x)\) for every positive rational number
      \(\lambda\).
    \item\label{prop:nakelmnp0963iv}
      If \(D\) is a nef \(\QQ\)-invertible sheaf, then \(\varepsilon'(\lVert D
      \rVert;x) \le \varepsilon(D;x)\).
      Equality holds if \(X\) is normal and either
      \(\varepsilon'(\lVert D \rVert;x) = 0\) or \(D\) is ample.
    \item\label{prop:nakelmnp0963v}
      Let \(D_1,D_2\) be \(\QQ\)-invertible sheaves such that \(x \notin
      \SB(D) \cup \SB(D')\).
      Then, we have
      \[
        \varepsilon'\bigl(\lVert D_1+D_2 \rVert;x\bigr) \ge 
        \varepsilon'\bigl(\lVert D_1 \rVert;x\bigr)
        + \varepsilon'\bigl(\lVert D_2 \rVert;x\bigr).
      \]
  \end{enumerate}
\end{proposition}
\begin{proof}
  \par \((\ref{prop:nakelmnp0963i})\), \((\ref{prop:nakelmnp0963ii})\), and
  \((\ref{prop:nakelmnp0963iv})\)
  follow from Proposition \ref{prop:elmnp64} and Proposition
  \ref{prop:elmnp0963}.
  \((\ref{prop:nakelmnp0963iii})\) holds by Proposition \ref{prop:naklimit}
  and the
  definition of \(\varepsilon'(\lVert D \rVert;x)\) for \(\QQ\)-invertible
  sheaves \(D\).
  \((\ref{prop:nakelmnp0963v})\) follows from Lemma \ref{lem:naksuperadd}.
\end{proof}
By Proposition \ref{prop:elmnp64} and Corollary \ref{cor:movingseshcont} (or by
repeating the proof of Corollary \ref{cor:movingseshcont}), we obtain the
following:
\begin{corollary}
  Let \(X\) be a normal complete variety and let \(x \in X\) be a closed point.
  Then, the function
  \[
    \begin{tikzcd}[row sep=0,column sep=1.475em]
      \operatorname{Big}_\QQ^{\{x\}}(X) \rar & \RR_{>0}\\
      D \rar[mapsto] & \varepsilon'\bigl(\lVert D \rVert;x\bigr)
    \end{tikzcd}
  \]
  is Lipschitz continuous on each compact subset of
  \(\operatorname{Big}_\QQ^{\{x\}}(X)\), and extends to a function on
  \(\operatorname{Big}_\RR^{\{x\}}(X)\) that is
  Lipschitz continuous on each compact subset of
  \(\operatorname{Big}_\RR^{\{x\}}(X)\).
  In particular, \(D \mapsto \varepsilon'(\lVert D \rVert;x)\) extends to a
  continuous function on \(\operatorname{Big}_\RR^{\{x\}}(X)\).
\end{corollary}
\subsection{Moving Seshadri constants and generation of jets}
We now prove the following characterization of moving Seshadri constants in
terms of generation of jets.
The case of smooth complex projective varieties is due to Ein, Lazarsfeld,
Musta\c{t}\u{a}, Nakamaye, and Popa \cite{ELMNP09}.
\begin{theorem}[cf.\ {\cite[Proposition 6.6]{ELMNP09}}]\label{thm:seshjet}
  Let \(X\) be a normal complete variety and let \(D\) be a \(\QQ\)-invertible
  sheaf.
  Consider a closed point \(x \in X\).
  Then, we have
  \begin{equation}\label{eq:movseshisjet}
    \varepsilon\bigl(\lVert D \rVert;x\bigr)
    = \varepsilon_{\jet}\bigl(\lVert D \rVert;x\bigr).
  \end{equation}
\end{theorem}
\begin{proof}
  Since both sides are homogeneous, it suffices to prove the case when \(D\) is
  a \(\ZZ\)-invertible sheaf.
  We proceed in steps.\medskip
  \begin{step}
    The inequality \(\le\) holds in \eqref{eq:movseshisjet}.
  \end{step}
  If \(\varepsilon(\lVert D \rVert;x) = 0\), there is nothing to show.
  Otherwise, let \(f\colon X' \to X\) be a birational morphism from a projective
  scheme and let \(f^*D \equiv_\RR A+E\) be an \(\RR\)-numerical equivalence as
  in Definition \ref{def:movingsesh}.
  As in Step \ref{step:elmnp64four} of the proof of Proposition
  \ref{prop:elmnp64}, we can perturb the coefficients \(E\) to obtain a
  decomposition
  \[
    f^*D = A' + E'
  \]
  as \(\QQ\)-Cartier divisors where \(A'\) is ample and \(E'\) is effective such
  that \(f^{-1}(x) \notin \Supp(E')\).
  We then have
  \begin{align*}
    \varepsilon\bigl(A';f^{-1}(x)\bigr) &= \varepsilon_\jet\bigl(\lVert A'
    \rVert;f^{-1}(x)\bigr)\\
    &\le \limsup_{m \to \infty}
    \frac{s\bigl(m\,f^*D;f^{-1}(x)\bigr)}{m}\\
    &= \varepsilon_\jet\bigl(\lVert D
    \rVert;x\bigr)
  \end{align*}
  where the first equality holds by Theorem \ref{prop:seshjetample}, the
  inequality holds by the fact that \(E'\) is effective in the decomposition
  \(f^*D = A'+E'\), and the last equality holds by the normality of \(X\) using
  the commutative diagram \eqref{eq:smdwithnormal}.
  Taking the limit as \(E' \to E\), we see that \([A'] \to [A]\) as numerical
  classes in \(N^1_\RR(X')\).
  By the continuity of Seshadri constants on the ample cone \cite[Corollary
  3.34]{FM21}, we obtain
  \[
    \varepsilon\bigl(A;f^{-1}(x)\bigr) \le \varepsilon_\jet\bigl(\lVert D
    \rVert;x\bigr).\medskip
  \]
  \begin{step}
    The inequality \(\ge\) holds in \eqref{eq:movseshisjet}.
  \end{step}
  If \(x \in \SB(D)\), then \(\varepsilon_\jet(\lVert D \rVert;x) = 0\), in
  which case there is nothing to show.
  We will show that \(\varepsilon(\lVert D \rVert;x) \ge s(mD;x)/m\) for all
  \(m\) such that \(x \notin \Bs(\lvert mD \rvert)\).
  This suffices by \eqref{eq:kurrem24} and \eqref{eq:feketeappz}.
  \par Let \(\pi_m\colon X_m \to X\) be the blowup of \(X\) at \(\fb(\lvert mD
  \rvert)\).
  Let \(f_m\colon X'_m \to X_m\) be a birational morphism from a projective
  scheme that is an isomorphism at \(x\), which exists by
  \citeleft\citen{Con07}\citemid Corollary 2.6\citepunct \citen{Del10}\citemid
  Corollaire 1.4\citeright.
  Set \(x' \coloneqq (\pi_m \circ f_m)^{-1}(x)\).
  Since \(X\) is normal,
  \[
    \bigl\lvert (\pi_m \circ f_m)^*D \bigr\rvert = \lvert f_m^*M_m \rvert +
    f_m^*F_m
  \]
  is the decomposition of \(\lvert (\pi_m \circ f_m)^*D \rvert\) into its mobile
  and fixed parts.
  Moreover, \(x' \notin \Supp(f_m^*F_m)\).
  Using the notation \(1_{f_m^*F_m}\) from
  \cite[\href{https://stacks.math.columbia.edu/tag/01WX}{Tag
  01WX}]{stacks-project}, we see that
  multiplication by \(1_{f_m^*F_m}\) yields the vertical isomorphisms in the top
  half of the commutative diagram
  \[
    \begin{tikzcd}
      H^0\bigl(X'_m,\cO_{X'_m}(f_m^*M_m)\bigr)
      \rar\arrow[hook]{d}[sloped]{\sim} & 
      H^0\Bigl(X'_m,\cO_{X'_m}(f_m^*M_m) \otimes_{\cO_{X'_m}}
      \cO_{X'_m}/\fm_{x'}^{\ell+1}\Bigr) \dar[sloped]{\sim}\\
      H^0\Bigl(X'_m,\cO_{X'_m}\bigl((\pi_m \circ f_m)^*D\bigr)\Bigr) \rar & 
      H^0\Bigl(X'_m,\cO_{X'_m}\bigl((\pi_m \circ f_m)^*D \otimes_{\cO_{X'_m}}
      \cO_{X'_m}/\fm_{x'}^{\ell+1}\Bigr)\\
      H^0\bigl(X,\cO_{X}(mD)\bigr) \rar\arrow[hook']{u}[sloped,swap]{\sim} & 
      H^0\bigl(X,\cO_{X}(mD) \otimes_{\cO_X}
      \cO_X/\fm_x^{\ell+1}\bigr)\mathrlap{.}
      \uar[sloped,swap]{\sim}
    \end{tikzcd}
  \]
  The vertical isomorphisms in the bottom half of the commutative diagram are
  isomorphisms because \(X\) is normal
  and because \(\pi_m \circ f_m\) is an isomorphism at \(x\).
  We therefore see that the bottom horizontal map in the commutative diagram is
  surjective if and only if
  the top horizontal map is surjective.
  Thus, we have the equality
  \begin{align*}
    \frac{s(mD;x)}{m} &= \frac{s(f_m^*M_m;x')}{m}.
    \intertext{Finally, we have}
    \frac{s\bigl(M_m;\pi_m^{-1}(x)\bigr)}{m} &\le \varepsilon_{\jet}\bigl( \lVert
    f_m^*M_m \rVert;x'\bigr)\\
    &\le \varepsilon(f_m^*M_m;x')\\
    &= \varepsilon\bigl(M_m;\pi_m^{-1}(x)\bigr)\\
    &\le \varepsilon'\bigl(\lVert D \rVert;x\bigr)\\
    &= \varepsilon\bigl(\lVert D \rVert;x\bigr)
  \end{align*}
  where the first and third inequalities hold by definition,
  the second inequality holds by Theorem \ref{prop:seshjetample}, the middle
  equality holds by Corollary \eqref{cor:seshbirat}, and the last equality holds
  by Proposition \ref{prop:elmnp64}.
\end{proof}

\section{Frobenius--Seshadri constants}\label{sect:frobsesh}
In this section, we continue developing the theory of Frobenius--Seshadri
constants to prove Theorem \ref{thm:introseshjetample}.
Frobenius--Seshadri constants are a Frobenius variant of Seshadri constants that
were introduced by Musta\c{t}\u{a} and Schwede \cite{MS14} to prove that lower
bounds on Seshadri constants yield adjoint-type divisors
that are free or very ample.
The author of this paper generalized their definition in
\cite{Mur18} to address generation of higher order jets.
See also \cite[Appendix]{DW22}, which develops material from \cite{MS14} over
arbitrary fields.
\subsection{The trace of Frobenius and Schwede's canonical linear system}
Let \(X\) be a locally Noetherian \(F\)-finite scheme of prime characteristic
\(p > 0\) admitting a dualizing complex \(\omega_X^\bullet\).
By Grothendieck duality for finite morphisms, we can consider the Grothendieck
trace map \cite[Chapter III, Proposition 6.5]{Har66}
\begin{align}
  \RR F^e_* F^{e!}\omega_X^\bullet &\longrightarrow \omega_X^\bullet\nonumber
  \intertext{associated to the \(e\)-th iterate of the Frobenius morphism, which is
  compatible with iterations of the Frobenius morphism by \cite[Chapter III,
  Proposition 6.6(1)]{Har66}.
  If \(F^!\omega_X^\bullet \cong \omega_X^\bullet\), we can write the trace map as
  \(\RR F^e_*\omega_X^\bullet \to \omega_X^\bullet\).
  Taking the cohomohology sheaf in lowest degree on each connected component of
  \(X\), we obtain the map}
  \Phi^e\colon F^e_* \omega_X &\longrightarrow \omega_X\label{eq:defphie}
\end{align}
which again is compatible with iterations of the Frobenius morphism.
\begin{remark}
  The hypothesis that \(F^!\omega_X^\bullet \cong \omega_X^\bullet\)
  holds for
  example when \(X\) is separated and essentially of finite type over an
  \(F\)-finite Noetherian semi-local ring \(R\) and the dualizing complex
  \(\omega_X^\bullet\) is the exceptional pullback of a dualizing complex
  \(\omega_R^\bullet\) on \(R\).
  This follows from Grothendieck duality for separated morphisms essentially of
  finite type \cite[Theorem 5.3]{Nay09}, where \(\omega_R^\bullet\)
  exists by the \(F\)-finiteness of \(R\) \cite[Remark 13.6]{Gab04}.
\end{remark}
We now recall the definition of the following canonical linear system due to
Schwede \cite{Sch14}.
These sections are the sections that are stable under the action of the trace of
Frobenius and act as though the Kodaira vanishing theorem were true
\cite[Proposition 5.3]{Sch14}.
\begin{citeddef}[{\citeleft\citen{Sch14}\citemid p.\ 75\citeright}]
  Let \(X\) be a complete scheme over an \(F\)-finite field of characteristic
  \(p > 0\).
  Let \(\sL\) be an invertible sheaf on \(X\).
  We define
  \begin{align*}
    \MoveEqLeft[2]
    S^0(X,\omega_X \otimes_{\cO_X} \sL)\\
    &\coloneqq\bigcap_{e \ge 0}
    \im\Bigl( H^0\bigl(X,\omega_X \otimes_{\cO_X} \sL^{\otimes p^e}\bigr)
    \xrightarrow{H^0(X,\Phi^e \otimes \sL)} H^0\bigl(X,\omega_X \otimes_{\cO_X}
    \sL\bigr)\Bigr)
  \end{align*}
  where \(\Phi^e \otimes \sL\) is obtained by twisting \eqref{eq:defphie} by
  \(\sL\).
  Since \(H^0(X,\omega_X \otimes_{\cO_X} \sL)\)
  is a finite-dimensional \(k\)-vector space, the descending intersection in the
  definition above stabilizes.
\end{citeddef}
\subsection{Frobenius--Seshadri constants}
We define Frobenius--Seshadri constants.
The case when \(\ell = 0\) below is the original notion due to Musta\c{t}\u{a}
and Schwede in \cite{MS14}.
The case when \(\ell > 0\) below is due to the author of this paper
\cite{Mur18}.
Below, we have changed the terminology ``separates \(p^e\)-Frobenius \(\ell\)-jets
at \(x\)'' from \cite{Mur18} to ``generates \(p^e\)-Frobenius \(\ell\)-jets at
\(x\)'' to match the terminology for \(\ell\)-jets in the classical setting.
\begin{citeddef}[\citeleft\citen{MS14}\citemid Definition 2.4\citepunct
  \citen{Mur18}\citemid Definition 2.2\citeright]
  Let \(X\) be a complete scheme over a field \(k\) of characteristic \(p > 0\).
  Let \(L\) be a \(\ZZ\)-invertible sheaf on \(X\).
  Consider a closed point \(x \in X\) with defining coherent ideal sheaf \(\fm_x
  \subseteq \cO_X\).
  For all integers \(\ell \ge -1\) and \(e \ge 0\), we say that
  \(L\) \textsl{generates \(p^e\)-Frobenius \(\ell\)-jets at \(x\)} if the
  restriction map
  \[
    H^0\bigl(X,\cO_X(L)\bigr) \longrightarrow
    H^0\bigl(X,\cO_X(L) \otimes_{\cO_X} \cO_X\big/(\fm_x^{\ell+1})^{[p^e]}\bigr)
  \]
  is surjective.
  \par Now suppose \(\ell \ge 0\) and for each integer \(m \ge 1\), let
  \(s^\ell_F(mL;x)\) be the largest
  integer \(e \ge 0\) such that \(mL\) generates \(p^e\)-Frobenius \(\ell\)-jets
  at \(x\).
  If no such \(e\) exists, set \(s^\ell_F(mL;x) = -\infty\).
  The \textsl{\(\ell\)-th Frobenius--Seshadri constant of \(L\) at \(x\)} is
  \[
    \varepsilon_F^\ell(L;x) \coloneqq \sup_{m \ge 1}
    \frac{p^{s^\ell_F(mL;x)}-1}{m/(\ell+1)}.
  \]
  For all closed points \(x \in X\) such that
  \(\varepsilon_F^\ell(L;x) > 0\), the supremum is in fact a limit supremum
  by \citeleft\citen{MS14}\citemid Proposition
  2.6\((iii)\)\citepunct \citen{Mur18}\citemid Proposition
  2.5\((iv)\)\citeright\ (the assumption that \(L\) is ample in
  \cite{MS14,Mur18} is only used to ensure that \(\varepsilon_F^\ell(L;x) > 0\)).
  If \(\ell = -1\), we set \(\varepsilon_F^{-1}(L;x) \coloneqq +\infty\).
  \par Finally, since Frobenius--Seshadri constants are homogeneous with
  respect to \(L\) as long as \(\varepsilon_F^\ell(L;x) > 0\)
  by \citeleft\citen{MS14}\citemid Proposition 2.8\citepunct
  \citen{Mur18}\citemid Proposition 2.8\citeright\ (the assumption that \(L\)
  is ample in
  \cite{MS14,Mur18} is only used to ensure that \(\varepsilon_F^\ell(L;x) > 0\)),
  we can extend the definition of \(\varepsilon_F^\ell(L;x)\) to
  \(\QQ\)-invertible sheaves.
\end{citeddef}
\subsection{Comparison with Seshadri constants}
We have the following comparison between the
version of Seshadri constants defined using generation of jets (Definition
\ref{def:jetsepsesh}) and Frobenius--Seshadri constants.
\begin{citedprop}[{\citeleft\citen{MS14}\citemid Proposition 2.12\citepunct
  \citen{Mur18}\citemid Proposition 2.9\citepunct
  \citen{DW22}\citemid Proposition A.14\citeright}]
  \label{prop:frobseshcomp}
  Let \(X\) be a complete variety and let \(L\) be a \(\QQ\)-invertible sheaf
  on \(X\).
  Consider a closed point \(x \in X\).
  For every integer \(\ell \ge 0\), we have
  \begin{equation}\label{eq:frobseshcomp}
    \frac{\ell+1}{\ell + n} \cdot \varepsilon_{\jet}\bigl(\lVert L \rVert;x\bigr)
    \le \varepsilon_F^\ell(L;x)
    \le \varepsilon_{\jet}\bigl(\lVert L \rVert;x\bigr).
  \end{equation}
\end{citedprop}
\par We start with the following commutative algebraic result.
In the statement below, the \textsl{analytic spread} of an ideal \(I\) of a
Noetherian local ring \((R,\fm)\) is defined to be the Krull dimension of the
fiber cone \(R[It]/\fm R[It]\) \cite[Definition 5.1.5]{SH06}.
\begin{lemma}\label{lem:pigeonhole}
  Let \((R,\fm,k)\) be a Noetherian local ring of prime characteristic \(p >
  0\).
  Then, for every ideal \(\fa \subseteq R\) of analytic spread \(h\), there
  exists an integer \(t \ge 0\) such that for all non-negative integers \(e\)
  and \(\ell\), we have
  \begin{equation}\label{eq:pigeonhole}
    \fa^{\ell p^e + h(p^e-1) + 1 + t} \subseteq (\fa^{\ell+1})^{[p^e]} \subseteq
    \fa^{(\ell+1)p^e}.
  \end{equation}
  In particular, if \(\fa = \fm\), then \eqref{eq:pigeonhole} holds for \(h =
  \dim(R)\).
\end{lemma}
\begin{proof}
  The second inclusion in \eqref{eq:pigeonhole} holds by definition of Frobenius
  powers of ideals.
  It therefore suffices to show the first inclusion in \eqref{eq:pigeonhole}.
  \par We first reduce to the case when the residue field \(k\) is infinite.
  Consider the ring extension
  \[
    R \hooklongrightarrow R[x]_{\fm R[x]} \eqqcolon S
  \]
  as in \cite[\S8.4]{SH06}.
  Then, \(S\) is a Noetherian local ring of characteristic \(p > 0\), \(R
  \hookrightarrow S\) is faithfully flat, and \(S/\fm S \cong k(x)\) is infinite.
  Since we can check the inclusions in \eqref{eq:pigeonhole} after a faithfully
  flat extension and since analytic spread does not change when passing to \(S\)
  \cite[Lemma 8.4.2(4)]{SH06}, we can replace \(R\) with \(S\) to assume that
  \(k\) is infinite.
  \par It remains to prove the first inclusion in \eqref{eq:pigeonhole} when \(k\) is
  infinite.
  Since \(k\) is infinite, there exists a minimal reduction \(\fq \subseteq
  \fa\) generated by \(h\) elements \cite[Proposition 8.3.7]{SH06}.
  Thus, by definition \cite[Definition 1.2.1]{SH06}, there exists an integer \(t
  > 0\) such that \(\fa^{s+t} = \fq^s \fa^t\) for every integer \(s \ge 0\).
  Setting \(s = \ell p^e+h(p^e-1)+1\), we have
  \begin{align*}
    \fa^{\ell p^e + h(p^e-1) + 1 + t} &=
    \fq^{\ell p^e + h(p^e-1) + 1} \cdot \fa^t\\
    &\subseteq (\fq^{\ell+1})^{[p^e]} \cdot \fa^t\\
    &\subseteq (\fa^{\ell+1})^{[p^e]}
  \end{align*}
  for all non-negative integers \(e\) and \(\ell\), where the inclusion in the
  second line holds by the pigeonhole principle \cite[Lemma 2.4\((a)\)]{HH02}.
  The special case for \(\fa = \fm\) follows from \cite[Corollary 8.3.9]{SH06}.
\end{proof}
We can now prove Proposition \ref{prop:frobseshcomp}.
\begin{proof}[Proof of Proposition \ref{prop:frobseshcomp}]
  Since both \(\varepsilon_{\jet}(\lVert \,\cdot\, \rVert;x)\) and
  \(\varepsilon_F^\ell(\,\cdot\,;x)\) are homogeneous, it suffices to consider
  the case when \(L\) is a \(\ZZ\)-invertible sheaf.
  By Lemma \ref{lem:pigeonhole}, there exists an integer \(t \ge 0\) such that
  \begin{equation}\label{eq:frobseshcompideals}
    \fm_x^{\ell p^e + n(p^e-1) + 1 + t} \subseteq (\fm_x^{\ell+1})^{[p^e]} \subseteq
    \fm_x^{(\ell+1)p^e}
  \end{equation}
  for all non-negative integers \(e\) and \(\ell\).
  The right inclusion in \eqref{eq:frobseshcompideals} implies
  \begin{align*}
    (\ell+1)\bigl(p^{s^\ell_F(mL;x)} - 1\bigr) &\le
    (\ell+1)\,p^{s^\ell_F(mL;x)} - 1 \le s(mL;x)
  \intertext{and hence the right inequality in \eqref{eq:frobseshcomp}
  follows after dividing by \(m\) throughout and taking limit suprema as \(m
  \to \infty\).
  See Definition \ref{def:demaillys} for the definition of \(s(mL;x)\).
  \endgraf For the left inequality in \eqref{eq:frobseshcomp}, let \(\delta > 0\) be
  given, and let \(m_0\) be a positive integer such that}
    \frac{s(m_0L;x)}{m_0} &> \varepsilon_{\jet}\bigl(\lVert L \rVert;x\bigr) - \frac{\ell + n}{\ell + 1} \cdot
    \delta.
  \end{align*}
  For every non-negative integer \(e\), set
  \begin{align*}
    d_e \coloneqq \biggl\lceil \frac{\ell p^e+n(p^e-1)+t}{s(m_0L;x)}
    \biggr\rceil
    &= \biggl\lceil \frac{(\ell+n)p^e-n+t}{s(m_0L;x)} \biggr\rceil.
    \intertext{By the superadditivity property of \(s(\,\cdot\,;x)\)
    \eqref{eq:superadditivity},
    we have}
    s(m_0d_e L;x) \ge d_e \cdot s(m_0L;x) &\ge \ell p^e + n(p^e-1)+t.
    \intertext{By the left inclusion in \eqref{eq:frobseshcompideals}, this
    inequality implies}
    s_F^\ell(m_0d_eL;x) &\ge e,
  \end{align*}
  and hence
  \begin{align*}
    \varepsilon_F^\ell(L;x)
    &\ge \frac{p^{s^\ell_F(m_0d_eL;x)}-1}{m_0d_e/(\ell+1)}\\
    &\ge \frac{(\ell+1)(p^e-1)}{m_0\bigl\lceil\bigl(
    (\ell+n)p^e-n+t\bigr)/s(m_0L;x)\bigr\rceil}.
  \end{align*}
  As \(e \to \infty\), the right-hand side converges to
  \[
    \frac{\ell+1}{\ell+n} \cdot \frac{s(m_0L;x)}{m_0} > \frac{\ell+1}{\ell+n}
    \cdot \varepsilon_{\jet}\bigl(\lVert L \rVert;x\bigr) - \delta,
  \]
  and hence we have the inequality
  \[
    \varepsilon^\ell_F(L;x) > \frac{\ell+1}{\ell+n} \cdot \varepsilon_{\jet}\bigl(\lVert L \rVert;x\bigr) -
    \delta
  \]
  for all \(\delta > 0\).
  Since \(\delta > 0\) was arbitrary, we obtain the left inequality in
  \eqref{eq:frobseshcomp}.
\end{proof}
\subsection{Generating jets using Frobenius--Seshadri constants}
We can now state the following theorem, which says that lower bounds on
Frobenius--Seshadri constants imply adjoint-type bundles separate points and/or
generate jets.
When \(X\) is smooth, the case \(\ell = 0\) over algebraically closed fields
is due to Musta\c{t}\u{a} and Schwede \cite{MS14}, which Das and Waldron
\cite{DW22} noted holds over all \(F\)-finite fields.
Again when \(X\) is smooth over an algebraically closed field,
the case \(\ell > 0\) and \(r = 1\) is due to the author of this paper \cite{Mur18}.
The \(F\)-injective case over algebraically closed fields  when \(\ell > 0\) and
\(r = 1\) was remarked in \cite{Mur18}.
\par The result below generalizes all of these results in \cite{MS14,Mur18,DW22}
in three directions:
\begin{enumerate*}[label=\((\arabic*)\)]
  \item the points \(x_i\) are only assumed to be \(F\)-injective,
  \item we can control generation of jets at multiple points simultaneously, and
  \item the jets we control can be of higher order.
\end{enumerate*}
\begin{theorem}[cf.\ {\citeleft\citen{MS14}\citemid Theorem 3.1\((i)\)\citepunct
  \citen{Mur18}\citemid Theorem C and Remark 3.3\citepunct \citen{DW22}\citemid
  Theorem A.13\citeright}]\label{thm:seshgenjets}
  Let \(X\) be a projective variety of dimension \(n\) over a
  field \(k\) of characteristic \(p > 0\).
  Let \(Z = \{x_1,x_2,\ldots,x_r\}\) be closed points in \(X\) such that \(X\)
  is \(F\)-injective at each \(x_i\) and
  let \(L_1,L_2,\ldots,L_r\) be \(\QQ\)-Cartier divisors on \(X\) such that
  \[
    L \coloneqq L_1 + L_2 + \cdots + L_r
  \]
  is \(\ZZ\)-Cartier and such that
  \(\Bplus(L_i) \cap Z = \emptyset\) for all \(i\).
  Let \(\ell_1,\ell_2,\ldots,\ell_r\) be non-negative integers.
  If \(\varepsilon^{\ell_i-1}_F(L_i;x_i) > \ell_i\) for every \(i\),
  then the restriction map
  \begin{alignat*}{3}
    H^0\bigl(X,\omega_X \otimes_{\cO_X} \cO_X(L)\bigr) &&{}\longrightarrow{}&
    H^0\biggl(X,\omega_X \otimes_{\cO_X} \cO_X(L) \otimes_{\cO_X}
    \cO_X\bigg/\prod_{i=1}^r\fm_{x_i}^{\ell_i}\biggr)
  \intertext{is surjective.
  Moreover, if \(k\) is \(F\)-finite, then the composition}
    S^0\bigl(X,\omega_X \otimes_{\cO_X} \cO_X(L)\bigr) &&{}\hooklongrightarrow{}&
    H^0\bigl(X,\omega_X \otimes_{\cO_X} \cO_X(L)\bigr)\\
    &&{}\longrightarrow{}&
    H^0\biggl(X,\omega_X \otimes_{\cO_X} \cO_X(L) \otimes_{\cO_X}
    \cO_X\bigg/\prod_{i=1}^r\fm_{x_i}^{\ell_i}\biggr)
  \end{alignat*}
  is surjective.
\end{theorem}
We start with the following preliminary result, which is a version of
\cite[Proposition 2.16 and Lemma 3.3]{MS14} for generating higher order jets at multiple
points.
\begin{proposition}[cf.\ {\cite[Proposition 2.16 and Lemma 3.3]{MS14}}]\label{prop:ms216}
  Let \(X\) be a projective variety of dimension \(n\) over a
  field \(k\) of characteristic \(p > 0\).
  Let \(Z = \{x_1,x_2,\ldots,x_r\}\) be closed points in \(X\) and
  let \(L_1,L_2,\ldots,L_r\) be \(\QQ\)-Cartier divisors on \(X\) such that
  \[
    L \coloneqq L_1 + L_2 + \cdots + L_r
  \]
  is \(\ZZ\)-Cartier and such that
  \(\Bplus(L_i) \cap Z = \emptyset\) for all \(i\).
  Let \(\ell_1,\ell_2,\ldots,\ell_r\) be non-negative integers.
  Fix a real number \(\alpha > 0\).
  If
  \[
    \varepsilon^{\ell_i-1}_F(L_i;x_i) > \ell_i\alpha
  \]
  for every \(i\),
  then there exist positive integers \(e\) and \(m\) such that
  \[
    \frac{p^e-1}{m} > \alpha
  \]
  and the restriction map
  \[
    H^0\bigl(X,\cO_X(mL)\bigr)
    \longrightarrow
    H^0\biggl(X,\cO_X(mL) \otimes_{\cO_X} \cO_X\bigg/
    \prod_{i=1}^r(\fm_{x_i}^{\ell_i})^{[p^e]}\biggr)
  \]
  is surjective.
  Moreover, we may assume that \(p^e-1 - m\alpha\) is arbitrarily large.
\end{proposition}
\begin{proof}
  We claim we can base change to a purely transcendental extension
  \(k(t)\) to assume that \(k\) is infinite.
  Under such a base change, \(X \otimes_k k(t)\) is a localization
  of \(X \times_k \bA^1_k\), and is therefore a variety.
  The formation of augmented base loci is compatible with ground field
  extensions \cite[\S2.7]{Bir17}.
  It remains to show that generation of \(p^e\)-Frobenius \(\ell\)-jets, and
  therefore Frobenius--Seshadri constants, are compatible with the ground
  field extension \(k \subseteq k(t)\).
  Let \(x \in X\) be a closed point with residue field \(k'\).
  Then, \(\fm_x \cdot \cO_{X \otimes_k k(t)}\) is a maximal ideal because
  \[
    \frac{\cO_{X,x} \otimes_k k(t)}{\fm_x \cdot ( \cO_{X,x} \otimes_k k(t))}
    \cong k' \otimes_k k(t) \cong k'(t).
  \]
  By flat base change, we conclude that generation of \(p^e\)-Frobenius
  \(\ell\)-jets, and therefore Frobenius--Seshadri constants, are compatible
  with the ground field extension \(k \subseteq k(t)\).\medskip
  \par We now prove the proposition when \(k\) is infinite.
  Let \(m_0\) be such that for all \(m' \ge m_0\) such that \(m'L_i\) is
  \(\ZZ\)-Cartier, the sheaves
  \(\cO_X(m'L_i)\) are globally generated away from \(\Bplus(L_i)\) for all
  \(i\).
  Such an integer \(m_0\) exists by Theorem \ref{prop:kur13prop27}.
  For every \(i\) we can find \(m_i\) sufficiently large and divisible and \(e_i\)
  sufficiently large such that \(m_iL_i\) generates
  \(p^{e_i}\)-Frobenius \((\ell_i-1)\)-jets at \(x_i\) and
  \begin{align}
    \frac{p^{e_i}-1}{m_i/\ell_i} &> \ell_i\alpha.\nonumber
    \intertext{Dividing both sides by \(\ell_i\), this inequality is equivalent
    to}
    \frac{p^{e_i}-1}{m_i} &> \alpha.\label{eq:lialpha}
  \end{align}
  By \cite[Lemma 2.4]{Mur18}, we may make the replacements
  \begin{equation}\label{eq:mur24replacement}
    e_i \longmapsto te_i \qquad \text{and} \qquad m_i \longmapsto
    \frac{m_i(p^{te_i}-1)}{p^{e_i}-1}
  \end{equation}
  without changing the inequality \eqref{eq:lialpha} or
  the fact that \(m_iL_i\) generates
  \(p^{e_i}\)-Frobenius \((\ell_i-1)\)-jets at \(x_i\).
  We may therefore assume that the \(e_i\) are all equal to a single value \(e\)
  and that \(m_i \ge
  m_0\) for all \(i\).\medskip
  \par Set \(m \coloneqq \operatorname{lcm}\{m_i\}\).
  We proceed in a sequence of steps.
  \begin{step}\label{step:ms2161}
    After possibly applying the replacements \eqref{eq:mur24replacement}, 
    the sheaves
    \[
      (\fm_{x_i}^{\ell_i})^{[p^{e}]} \otimes \cO_X\bigl((m+2m_0)L_i\bigr)
    \]
    are globally generated at \(x_j\) for every \(j\).
  \end{step}
  For each \(i\), since \(x_j \notin \Bplus(L_i)\) for all \(j\), we can write
  \[
    L_i \sim_\QQ A_{i}+E_{i}
  \]
  for ample \(\QQ\)-Cartier divisors \(A_{i}\) and effective
  \(\QQ\)-Cartier divisors \(E_{i}\) such that \(x_j \notin
  \Supp(E_i)\) for all \(j\).
  After replacing \(m_0\) and then \(m_i\) by multiples using
  \eqref{eq:mur24replacement}, we may assume the following:
  \begin{itemize}
    \item \(m_0L_i\), \(m_iL_i\), \(m_0 A_{i}\), \(m_i A_{i}\), \(m_0E_{i}\),
      and \(m_iE_{i}\) are \(\ZZ\)-Cartier divisors.
    \item \(m_i \ge nm_0\) where \(n \coloneqq \dim(X)\).
    \item \(\cO_{X}(m_0A_{i})\) and \(\cO_{X}(m_iA_{i})\)
      are globally generated.
    \item \(H^q(X,\cO_{X}((m+(2-q)m_0)A_{i})) = 0\) for all \(q \ge 1\).
  \end{itemize}
  Since \(\cO_X((m+m_0-m_i)L_i)\) is globally generated at every \(x_j\),
  \cite[Lemma
  2.7]{Mur18} implies that \((m+m_0)L_i\) generates \(p^{e}\)-Frobenius
  \((\ell_i-1)\)-jets at \(x_i\).
  \par Now consider the commutative diagram
  \[
    \begin{tikzcd}
      0 \dar & 0 \dar\\
      H^0\Bigl(X,\cO_{X}\bigl((m+m_0) A_{i}\bigr)\Bigr) \rar \dar
      & H^0\Bigl(X,\cO_{X}\bigl((m+m_0) A_{i}\bigr) \otimes_{\cO_X}
      \cO_{X}/(\fm_{x_i}^{\ell_i})^{[p^e]}\Bigr) \dar{1_{(m+m_0)E_i}}\\
      H^0\Bigl(X,\cO_{X}\bigl((m+m_0)L_i\bigr)\Bigr)
      \rar[twoheadrightarrow] \dar
      & H^0\Bigl(X,\cO_{X}\bigl((m+m_0)L_i\bigr) \otimes_{\cO_X}
      \cO_{X}/(\fm_{x_i}^{\ell_i})^{[p^e]}\Bigr) \dar\\
      H^0\Bigl(m_iE_{i},\cO_{m_iE_{i}}\bigl((m+m_0)
      L_i\rvert_{m_iE_{i}}\bigr)\Bigr)
      \dar
      & 0\\
      0
    \end{tikzcd}
  \]
  with exact columns, where we use the notation \(1_{(m+m_0)E_i}\) from
  \cite[\href{https://stacks.math.columbia.edu/tag/01WX}{Tag
  01WX}]{stacks-project}.
  Since the middle horizontal map is surjective, the top horizontal map must
  also be surjective.
  \par We claim that the sheaves
  \[
    (\fm_{x_i}^{\ell_i})^{[p^{e}]} \otimes_{\cO_X} \cO_{X_i}\bigl((m+2m_0)A_i\bigr)
  \]
  are \(0\)-regular with respect to \(m_0A_i\).
  Consider the long exact sequence on cohomology for the short exact
  sequence
  \[
    \begin{tikzcd}
      0 \dar\\
      (\fm_{x_i}^{\ell_i})^{[p^e]} \otimes_{\cO_X}
      \cO_{X}\bigl(\bigl(m+(2-q)m_0\bigr) A_{i}\bigr) \dar\\
      \cO_{X}\bigl(\bigl(m+(2-q)m_0\bigr) A_{i}\bigr) \dar\\
      \cO_{X}\bigl(\bigl(m+(2-q)m_0\bigr) A_{i}\bigr) \otimes_{\cO_X}
      \cO_X/\fm_{x_i}^{\ell_i})^{[p^e]} \dar\\
      0\mathrlap{.}
    \end{tikzcd}
  \]
  When \(q = 1\),
  the commutative diagram in the previous paragraph shows that this short exact
  sequence remains exact after applying \(H^0(X,\,\cdot\,)\).
  Thus, we have
  \[
    H^1\Bigl(X,(\fm_{x_i}^{\ell_i})^{[p^e]} \otimes_{\cO_X}
    \cO_{X}\bigl(\bigl(m+m_0\bigr) A_{i}\bigr)\Bigr) = 0.
  \]
  When \(q \ge 2\), we have
  \[
    H^q\bigl(X,\cO_{X}\bigl(\bigl(m+(2-q)m_0\bigr) A_{i}\bigr)\bigr) = 0
  \]
  by our choice of \(m\) and \(m_0\) and
  \[
    H^q\bigl(X,\cO_{X}\bigl(\bigl(m+(2-q)m_0\bigr) A_{i}\bigr) \otimes_{\cO_X}
    \cO_X/\fm_{x_i}^{\ell_i})^{[p^e]}\bigr) = 0
  \]
  since the sheaf has \(0\)-dimensional support.
  Thus, the sheaves
  \[
    (\fm_{x_i}^{\ell_i})^{[p^{e}]} \otimes_{\cO_X} \cO_{X_i}\bigl((m+2m_0)A_i\bigr)
  \]
  are \(0\)-regular with respect to \(m_0A_i\), and hence are
  globally generated
  by \cite[Chapter II, \S1, Proposition 1\((iii)\)]{Kle66}
  applied to \(X_{\bar{k}}\).
  Twisting by \((m+2m_0)E_i\), we see that
  \[
    (\fm_{x_i}^{\ell_i})^{[p^{e}]} \otimes_{\cO_X} \cO_X\bigl((m+2m_0)L_i\bigr)
  \]
  is globally generated at \(x_j\) for every \(j\).\medskip
  \begin{step}
    Conclusion of proof.
  \end{step}
  By Step \ref{step:ms2161} and the assumption that \(k\) is infinite,
  we can find sections
  \[
    s_i \in H^0\bigl(X,(\fm_{x_i}^{\ell_i})^{[p^{e}]} \otimes
    \cO_X\bigl((m+2m_0)L_i\bigr)\bigr)
  \]
  that do not vanish at any points in \(Z - \{x_i\}\).
  Here, we use the assumption that \(k\) is infinite to ensure that we can take
  general linear combinatons of sections that vanish at some of the points in
  \(Z - \{x_i\}\).
  Applying \cite[Lemma 2.7]{Mur18} again, we know that \((m+2m_0)L_i\) separates
  \(p^e\)-Frobenius \((\ell_i-1)\)-jets at \(x_i\), and hence the restriction
  maps
  \begin{equation}
    H^0\bigl(X,\cO_X\bigl((m+2m_0)L_i\bigr)\bigr) \longrightarrow
    H^0\bigl(X,\cO_X\bigl((m+2m_0)L_i\bigr) \otimes_{\cO_X}
    \cO_X\big/(\fm_{x_i}^{\ell_i})^{[p^e]}\bigr)\label{eq:eachxilijets}
  \end{equation}
  are surjective.
  The composition
  \begin{align}
    \bigotimes_{i=1}^r
    H^0\bigl(X,{}&\cO_X\bigl((m+2m_0)L_i\bigr)\bigr)\nonumber\\
    &\longrightarrow
    H^0\bigl(X,\cO_X\bigl((m+2m_0)L\bigr)\bigr)\label{eq:eachxilijetscomp}\\
    &\longrightarrow
    H^0\biggl(X,\cO_X\bigl((m+2m_0)L\bigr) \otimes_{\cO_X} \cO_X\bigg/
    \prod_{i=1}^r(\fm_{x_i}^{\ell_i})^{[p^e]}\biggr)\nonumber
  \end{align}
  is surjective: we have the natural isomorphism
  \begin{align*}
    H^0\biggl(X,\cO_X\bigl((m&+2m_0)L\bigr) \otimes_{\cO_X} \cO_X\bigg/
    \prod_i(\fm_{x_i}^{\ell_i})^{[p^e]}\biggr)\\
    &\cong
    \bigoplus_{i=1}^r H^0\bigl(X,\cO_X\bigl((m+2m_0)L\bigr) \otimes_{\cO_X} \cO_X\big/
    (\fm_{x_i}^{\ell_i})^{[p^e]}\bigr)
  \end{align*}
  by the Chinese remainder theorem,
  and the surjectivity of \eqref{eq:eachxilijets} implies that
  the \(i\)-th direct summand of this direct sum is the image of
  \[
    s_1 \otimes \cdots \otimes H^0\bigl(X,\cO_X\bigl((m+2m_0)L_i\bigr)\bigr)
    \otimes \cdots \otimes s_r.
  \]
  Because of the factorization \eqref{eq:eachxilijetscomp}, we have the
  surjection
  \[
    H^0\bigl(X,\cO_X\bigl((m+2m_0)L\bigr)\bigr)
    \longtwoheadrightarrow
    H^0\biggl(X,\cO_X\bigl((m+2m_0)L\bigr) \otimes_{\cO_X} \cO_X\bigg/
    \prod_{i=1}^r(\fm_{x_i}^{\ell_i})^{[p^e]}\biggr).
  \]
  By the proof of \cite[Lemma 2.4]{Mur18}, we see that this surjectivity is
  preserved under the replacements
  \[
    e \longmapsto te \qquad \text{and} \qquad m \longmapsto
    \frac{m(p^{te}-1)}{p^{e}-1}.
  \]
  Under these replacements, the ratio \((p^e-1)/(m+2m_0)\) transforms as follows:
  \[
    \frac{p^e-1}{m+2m_0} \longmapsto
    \frac{p^{te}-1}{\dfrac{m(p^{te}-1)}{p^{e}-1}+2m_0}.
  \]
  As \(t \to \infty\), this ratio approaches \((p^e-1)/m > \alpha\) as required.
  For the ``moreover'' statement, it suffices to note that
  \[
    p^{te}-1 - \biggl(\dfrac{m(p^{te}-1)}{p^{e}-1}+2m_0\biggr)\alpha
    \longrightarrow
    \infty
  \]
  as \(t \to \infty\).
\end{proof}
We can now prove Theorem \ref{thm:seshgenjets}.
\begin{proof}[Proof of Theorem \ref{thm:seshgenjets}]
  We first reduce to the case when \(k\) is \(F\)-finite.
  By the gamma construction of Hochster and Huneke \cite[\S6]{HH94}, our
  results on the gamma construction \cite[Theorem A]{Mur21} (see also
  \cite[Remark 5.8]{DM24}), and \cite[Lemma
  6.13\((b)\)]{HH94}, there exists a
  purely inseparable field extension \(k \subseteq k^\Gamma\) such that
  \(k^\Gamma\) is \(F\)-finite,
  the projection morphism \(X \otimes_k
  k^\Gamma \to X\) preserves reduced and \(F\)-injective loci, and the
  ideals \(\fm_{x_i} \cdot \cO_{X \otimes_k k^\Gamma}\) are maximal.
  By flat base change, generation of \(p^e\)-Frobenius \(\ell\)-jets and
  the Frobenius--Seshadri constants at the \(x_i\) do not
  change after extending the ground field to \(k^\Gamma\).
  The compatibility of \(\omega_X\) with field
  extensions \cite[Chapter V, Corollary 3.4\((a)\)]{Har66} therefore implies that
  it suffices to prove Theorem \ref{thm:seshgenjets} in the special case
  when \(k\) is \(F\)-finite.\medskip
  \par We now assume that \(k\) is \(F\)-finite.
  By Proposition \ref{prop:ms216},
  we can find \(m\) and \(e\) such that \(p^e-1 > m\) and the restriction map
  \begin{align}
    H^0\bigl(X,\cO_X(mL)\bigr)
    &\longrightarrow
    H^0\biggl(X,\cO_X(mL) \otimes_{\cO_X} \cO_X\bigg/
    \prod_{i=1}^r(\fm_{x_i}^{\ell_i})^{[p^e]}\biggr)\nonumber
    \intertext{is surjective.
    Moreover, we may assume that \(p^e-1-m\) is arbitrarily
    large, in which case \(\omega_X \otimes_{\cO_X} \cO_X((p^e-m)L)\) is globally
    generated at the \(x_i\) by Theorem \ref{prop:kur13prop27}, and}
      S^0\bigl(X,\omega \otimes_{\cO_X} \cO_X(L)\bigr) = \im\Bigl(&H^0\bigl(X,\omega_X
      \otimes_{\cO_X} \cO_X(p^eL)\bigr) \longrightarrow H^0\bigl(X,\omega_X \otimes_{\cO_X}
      \cO_X(L)\bigr)\Bigr)
      \label{eq:s0isimage}
    \intertext{by the fact that \(H^0(X,\omega_X \otimes \cO_X(L))\) is a
    finite-dimensional \(k\)-vector space \cite[p.\ 75]{Sch14}.
    By \cite[Lemma 2.7]{Mur18}, we then see that
    the restriction map}
    H^0\bigl(X,\omega_X \otimes_{\cO_X} \cO_X(p^eL)\bigr)
    &\longrightarrow
    H^0\biggl(X,\omega_X \otimes_{\cO_X} \cO_X(p^eL) \otimes_{\cO_X} \cO_X\bigg/
    \prod_{i=1}^r(\fm_{x_i}^{\ell_i})^{[p^e]}\biggr)\label{eq:h0pesurj}
  \end{align}
  is surjective.
  \par Next, we consider the commutative diagram
  \[
    \begin{tikzcd}
      0 \dar & 0 \dar\\
      F^e_*\biggl(\displaystyle
      \prod_{i=1}^r(\fm_{x_i}^{\ell_i})^{[p^e]} \otimes_{\cO_X} \omega_X \otimes_{\cO_X}
      \cO_X(p^eL) \biggr) \rar \dar
      & \displaystyle \prod_{i=1}^r \fm_{x_i}^{\ell_i} \otimes_{\cO_X} \omega_X \otimes_{\cO_X}
      \cO_X(L) \dar\\
      F^e_*\bigl(\omega_X \otimes_{\cO_X}
      \cO_X(p^eL) \bigr) \rar \dar
      & \omega_X \otimes_{\cO_X} \cO_X(L) \dar\\
      F^e_*\biggl(\omega_X \otimes_{\cO_X} \cO_X(p^eL) \otimes_{\cO_X}
      \cO_X\bigg/\displaystyle
      \prod_{i=1}^r(\fm_{x_i}^{\ell_i})^{[p^e]}\biggr) \rar[twoheadrightarrow] \dar
      & \omega_X \otimes_{\cO_X} \cO_X(L) \otimes_{\cO_X}
      \cO_X\bigg/\displaystyle\prod_{i=1}^r\fm_{x_i}^{\ell_i} \dar\\
      0 & 0
    \end{tikzcd}
  \]
  where the middle horizontal map is surjective at the \(x_i\) by the
  \(F\)-injectivity of \(X\) at the \(x_i\) (see \cite[Lemma A.1]{Mur21}), which
  implies the bottom horizontal map is surjective.
  Taking global sections in the bottom square, we obtain the commutative square
  \[
    \setbox0=\hbox\bgroup\ignorespaces
    \begin{tikzcd}
      H^0\bigl(X,\omega_X \otimes_{\cO_X} \cO_X(p^eL)\bigr) \rar \dar[twoheadrightarrow]
      & H^0\bigl(X,\omega_X \otimes_{\cO_X} \cO_X(L)\bigr) \dar\\
      H^0\biggl(X,\omega_X \otimes_{\cO_X} \cO_X(p^eL) \otimes_{\cO_X}
      \cO_X\bigg/\displaystyle
      \prod_{i=1}^r(\fm_{x_i}^{\ell_i})^{[p^e]}\biggr) \rar[twoheadrightarrow]
      & H^0\biggl(X,\omega_X \otimes_{\cO_X} \cO_X(L) \otimes_{\cO_X}
      \cO_X\bigg/\displaystyle\prod_{i=1}^r\fm_{x_i}^{\ell_i}\biggr)
    \end{tikzcd}
    \unskip\egroup\noindent\makebox[\textwidth]{\box0}
  \]
  where the left vertical map is surjective by \eqref{eq:h0pesurj} and the
  bottom horizontal map is surjective by the fact that the sheaves involved have
  \(0\)-dimensional support.
  Since the image of the top horizontal map is \(S^0(X,\omega_X \otimes_{\cO_X}
  \cO_X(L))\) by \eqref{eq:s0isimage}, the composition
  \begin{align*}
    S^0\bigl(X,\omega_X \otimes_{\cO_X} \cO_X(L)\bigr) \hooklongrightarrow{}&
    H^0\bigl(X,\omega_X \otimes_{\cO_X} \cO_X(L)\bigr)\\
    \longrightarrow{}&
    H^0\biggl(X,\omega_X \otimes_{\cO_X} \cO_X(L) \otimes_{\cO_X}
    \cO_X\bigg/\prod_{i=1}^r\fm_{x_i}^{\ell_i}\biggr)
  \end{align*}
  is surjective.
\end{proof}
We obtain the following version of \cite[Theorem 3.1]{MS14} for simultaneous
generation of higher order jets at finitely many closed points as a consequence.
Musta\c{t}\u{a} and Schwede \cite{MS14} prove results for global generation
\((\ref{cor:ms31gg})\), very bigness \((\ref{cor:ms31vbig})\), and
very ampleness \((\ref{cor:ms31va})\), but not \(\ell\)-jet ampleness
\((\ref{cor:ms31jetample})\).
See \cite[p.\ 195]{EKL95} for the definition of ``very big,'' which is also
known as ``birationally very ample.''
\begin{corollary}[cf.\ {\cite[Theorem 3.1]{MS14}}]\label{cor:ms31}
  Let \(X\) be a projective variety of dimension \(n\) over a
  field \(k\) of characteristic \(p > 0\)
  and let \(L\) be a Cartier divisor on \(X\).
  \begin{enumerate}[label=\((\roman*)\),ref=\roman*]
    \item\label{cor:ms31gg}
      If \(X\) is \(F\)-injective, \(L\) is ample,
      and \(\varepsilon^0_F(L;x) > 1\) for every \(x \in
      X\), then \(\omega_X \otimes_{\cO_X} \cO_X(L)\) is globally generated.
      If \(k\) is \(F\)-finite, then \(\omega_X \otimes_{\cO_X} \cO_X(L)\) is
      globally generated by
      \(S^0(X,\omega_X \otimes_{\cO_X}
      \cO_X(L))\).
    \item\label{cor:ms31vbig}
      Suppose that \(k\) is algebraically closed.
      If \(\varepsilon^0_F(L;x) > 2\) for some \(F\)-injective \(x \in X\) such
      that \(x \notin \Bplus(L)\), then
      \[
        \Bigl\lvert S^0\bigl(X,\omega_X \otimes_{\cO_X}
        \cO_X(L)\bigr)\Bigr\rvert
      \]
      is a very big linear system.
    \item\label{cor:ms31va}
      If \(X\) is \(F\)-injective, \(L\) is ample, and
      \(\varepsilon^0_F(L;x) > 2\) for every \(x \in X\), then
      \(K_X+L\) is \(1\)-jet ample.
      If \(k\) is algebraically closed, then
      \[
        \Bigl\lvert S^0\bigl(X,\omega_X \otimes_{\cO_X}
        \cO_X(L)\bigr)\Bigr\rvert
      \]
      is a \(1\)-jet ample linear system.
    \item\label{cor:ms31jetample}
      Fix an integer \(\ell > 0\).
      Let \(Z = \{x_1,x_2,\ldots,x_r\}\) be closed points in \(X\) such that
      \(X\) is \(F\)-injective at each \(x_i\) and suppose that \(\Bplus(L) \cap
      Z = \emptyset\) for all \(i\).
      Suppose that for every \(r\)-tuple of non-negative integers
      \((\ell_1,\ell_2,\ldots,\ell_r)\) such that
      \[
        \ell+1 = \sum_{i=1}^r \ell_i,
      \]
      we have \(\varepsilon^{\ell_i-1}_F(L;x_i) > \ell+1\) for every
      \(i\).
      Then,
      \[
        \Bigl\lvert S^0\bigl(X,\omega_X \otimes_{\cO_X}
        \cO_X(L)\bigr)\Bigr\rvert
      \]
      is an \(\ell\)-jet ample linear system at
      \(Z\).
    \item\label{cor:ms31jetampleglobal}
      Fix an integer \(\ell > 0\).
      If the hypothesis in \((\ref{cor:ms31jetample})\) holds for every finite
      set of points \(Z\), then \(K_X+L\) is \(\ell\)-jet ample.
      If \(k\) is \(F\)-finite, then
      \[
        \Bigl\lvert S^0\bigl(X,\omega_X \otimes_{\cO_X}
        \cO_X(L)\bigr)\Bigr\rvert
      \]
      is an \(\ell\)-jet
      ample linear system.
  \end{enumerate}
\end{corollary}
\begin{proof}
  First, \((\ref{cor:ms31gg})\) and
  \((\ref{cor:ms31jetampleglobal})\) immediately follow
  from Theorem \ref{thm:seshgenjets}.
  Next, \((\ref{cor:ms31jetample})\) follows from Theorem \ref{thm:seshgenjets} by
  setting
  \[
    L_i \coloneqq \frac{\ell_i}{\ell+1}L
  \]
  for every \(i\) since
  Frobenius--Seshadri constants are homogeneous with respect to multiples
  \citeleft\citen{MS14}\citemid Proposition 2.8\citepunct
  \citen{Mur18}\citemid Proposition 2.8\citeright\ (the assumption that \(L\)
  is ample in
  \cite{MS14,Mur18} is only used to ensure that \(\varepsilon_F^\ell(L;x) > 0\)).
  Finally,
  \((\ref{cor:ms31vbig})\) and \((\ref{cor:ms31va})\) follow from combining
  \cite[Lemma 3.3]{MS14} and the proof of Theorem \ref{thm:seshgenjets}.
\end{proof}
\subsection{Generating jets using Seshadri constants}
Finally, we can now prove Theorem \ref{thm:introseshjetample}, which we recall
generalizes the results in \cite{Dem92,MS14,Mur18} which connect Seshadri
constants to the local positivity of adjoint-type divisors.
\begin{customthm}{A}\label{thm:introseshjetample}
  Let \(X\) be a normal projective variety of dimension \(n\) over an arbitrary
  field \(k\) and let \(L\) be a Cartier divisor on \(X\).
  Let \(\ell\) be a non-negative integer.
      Let \(Z = \{x_1,x_2,\ldots,x_r\}\) be closed points in \(X\) such that \(X\)
      is of dense \(F\)-injective type (if \(\Char(k) = 0\)) or \(F\)-injective (if
      \(\Char(k) = p > 0\)) at every \(x_i\).
      Suppose that
      \[
        \varepsilon\bigl(\lVert L \rVert;x_i\bigr) > n+\ell
      \]
      for every \(i\).
      Then, \(\omega_X \otimes_{\cO_X}
      \cO_X(L)\) is \(\ell\)-jet ample at \(Z\).
  Moreover, if \(k\) is an \(F\)-finite field of characteristic \(p > 0\),
  then the linear system
  \[
    \Bigl\lvert S^0\bigl(X,\omega_X \otimes_{\cO_X} \cO_X(L)\bigr) \Bigr\rvert
  \]
  is \(\ell\)-jet ample at \(Z\).
\end{customthm}
\begin{proof}
  By reduction modulo \(p\) (see the proof of \cite[Corollary 3.15]{SZ20}
  or \cite[\S7.3.2]{Mur19} for details),
  it suffices to show the case when \(\Char(k) = p > 0\).
  By Theorem \ref{thm:seshjet}, we hve
  \begin{align*}
    \varepsilon\bigl(\lVert L \rVert;x_i\bigr) =
    \varepsilon_{\jet}\bigl(\lVert L \rVert;x_i\bigr) &> n+\ell
    \intertext{for all \(i\), and hence \(\Bplus(L) \cap Z = \emptyset\) by Remark
    \ref{rem:bplusandmovsesh}.
    By Proposition \ref{prop:frobseshcomp}, we have}
    \varepsilon_F^{\ell_i-1}(L;x_i) \ge \frac{\ell_i}{\ell_i-1+n} \cdot
    \varepsilon_{\jet}\bigl(\lVert L \rVert;x_i\bigr) &> 
    \ell_i \cdot \frac{n+\ell}{\ell_i-1+n} \ge \ell_i
  \end{align*}
  for all \(i\).
  Thus, Theorem \ref{thm:introseshjetample} follows from
  Theorem \ref{thm:seshgenjets} and Corollary \ref{cor:ms31}.
\end{proof}

\section{Lower bounds for Seshadri constants
of adjoint-type divisors}\label{sect:lowerbounds}
In this section, we prove lower bounds for Seshadri constants of adjoint-type
divisors in two situations:
\begin{enumerate*}[label=\((\arabic*)\)]
  \item for surfaces in arbitrary characteristic, and
  \item in arbitrary dimension over the complex numbers.
\end{enumerate*}
\subsection{Lower bounds for surfaces in arbitrary characteristic}
To show lower bounds for surfaces in arbitrary characteristic, we prove
the following result on possible values for Seshadri constants of adjoint-type
divisors.
This result was stated and proved for smooth complex surfaces by
Bauer and Szemberg \cite{BS11}.
\begin{theorem}[cf.\ {\cite[Theorem 4.1]{BS11}}]\label{thm:bslowerbound}
  Let \(X\) be a Gorenstein projective geometrically irreducible surface
  over a field \(k\) and
  let \(L\) be a nef Cartier divisor such that \(K_X+L\) is nef and big.
  Consider a regular \(k\)-rational point \(x \in X\).
  If \(\varepsilon(K_X+L;x) \in (0,1)\), then
  \[
    \varepsilon(K_X+L;x) = \frac{m-1}{m}
  \]
  for some integer \(m \ge 2\).
\end{theorem}
\begin{proof}
  We verify that the proof in \cite{BS11} for smooth complex surfaces
  works at this level of generality with some modifications.\medskip
  \begin{step}\label{step:bslowerbound1}
    For every \(\delta > 0\), there exists an integer \(m \ge 2\) such that
    \[
      \frac{m-1}{m} - \delta \le \varepsilon(K_X+L;x) \le \frac{m-1}{m}.
    \]
  \end{step}
  By Proposition \ref{prop:laz519},
  there exists an integral closed curve \(C \subseteq X\)
  such that \(x \in C\) and
  \[
    \frac{d}{m} - \delta \le \varepsilon(K_X+L;x) \le \frac{(K_X+L)\cdot C}{\mult_x(C)} =
    \frac{d}{m} \in (0,1)
  \]
  for integers \(d \coloneqq (K_X+L)\cdot C\) and \(m \coloneqq \mult_x(C)\)
  where \(d \le m-1\) and \(m \ge 2\).
  By the Hodge index theorem \cite[Theorem B.27]{Kle05} and the assumption that
  \(K_X+L\) is nef and big, we have
  \[
    d^2 = \bigl( (K_X+L) \cdot C \bigr)^2 \ge C^2\cdot(K_X+L)^2 \ge C^2.
  \]
  The hypothesis that \(L\) is nef and the adjunction formula for Gorenstein
  projective surfaces \cite[Proposition B.26]{Kle05} then imply
  \begin{align*}
    p_a(C) ={}& 1 + \frac{1}{2}\,C^2 + \frac{1}{2}\, C \cdot K_X\\
    \le{}& 1 + \frac{1}{2}\,d^2 + \frac{1}{2}\,C \cdot (K_X+L)\\
    ={}& 1 + \frac{d(d+1)}{2}.
  \intertext{\endgraf Next, we bound \(p_a(C)\) from below.
  Let \(\nu\colon C' \to C\) be the normalization map.
  By \cite[Proposition 7.5.4]{Liu02}, we have}
  p_a(C) ={}& p_a(C') + \sum_{z \in C} [k(z) : k]\,\delta_z
  \intertext{where denoting by \(\cO_{C,z}'\) the integral closure of \(\cO_{C,z}\), we
  have}
    \delta_z \coloneqq{}& {\len_{\cO_{C,z}}\bigl(\cO_{C,z}'/\cO_{C,z}\bigr)}\\
    ={}& \frac{1}{[k(z) : k]}\,\dim_k\bigl(\cO_{C,z}'/\cO_{C,z}\bigr).
    \intertext{By \cite[Theorem 1]{Hir57} and its proof, since \(x\) is a
    regular \(k\)-rational point on
    \(X\) such that \(C\) is of multiplicity \(m \ge 2\) at \(x\), we have}
    \delta_x ={}& \binom{m}{2} = \frac{m(m-1)}{2} \ge 1.
  \end{align*}
  Thus, we have
  \begin{align}
    \frac{m(m-1)}{2} &\le p_a(C) \le 1 + \frac{d(d+1)}{2}.\nonumber
    \intertext{Combining the two bounds for \(p_a(C)\) and the inequality \(d \le
    m-1\) obtained from the assumption that \(\varepsilon(K_X+L;x) \le d/m \in
    (0,1)\), we see that}
    \frac{m(m-1)}{2} \le p_a(C) &\le 1 + \frac{d(d+1)}{2} \le 1 + \frac{m(m-1)}{2}.\nonumber
  \intertext{Since \(m \ge 2\), the only integer solution \(d\) for this chain of
  inequalities is \(d = m-1\).
  We therefore conclude that}
    \frac{m-1}{m} - \delta &\le \varepsilon(K_X+L;x) \le \frac{m-1}{m}.\medskip
  \end{align}
  \setcounter{step}{1}
  \begin{step}
    Conclusion of proof.
  \end{step}
  For every integer \(n \ge 2\), choose real numbers \(\delta_n\) such that
  \[
    \delta_n \in \biggl(0, \frac{1}{n(n-1)}\biggr).
  \]
  For all distinct integers \(m,m' \ge 2\) and integers \(n \ge m\) and \(n' \ge
  m'\), we then have
  \begin{equation}\label{eq:segnointersect}
    \biggl[\frac{m-1}{m} - \delta_n,\frac{m-1}{m}\biggr] \cap
    \biggl[\frac{m'-1}{m'} - \delta_{n'},\frac{m'-1}{m'}\biggr] = \emptyset.
  \end{equation}
  \par Now by Step \ref{step:bslowerbound1}, for each integer \(n \ge 2\), we
  can choose an integer \(m_n\) such that
  \begin{equation}
    \frac{m_n-1}{m_n} - \delta_n \le \varepsilon(K_X+L;x) \le
    \frac{m_n-1}{m_n}.\label{eq:miineq}
  \end{equation}
  By \eqref{eq:segnointersect}, since \(\varepsilon(K_X+L;x)\) is fixed,
  we must have \(m_n > n\) for all \(n\).
  Passing to a subsequence of the \(\delta_n\), we obtain a infinite strictly
  increasing sequence of integers \(m_n\) such that \eqref{eq:miineq} holds.
  Taking limits as \(n \to \infty\) in \eqref{eq:miineq}, we see that 
  we see that \(\varepsilon(K_X+L;x) = (m-1)/m\).
\end{proof}
As a consequence, we obtain the following lower bound for Seshadri constants of
adjoint-type divisors.
This solves Conjecture \ref{conj:seshadjlower} for surfaces.
As mentioned in \S\ref{subsect:ourapproach}, the complex case is due to Bauer and
Szemberg \cite[Corollary 4.2]{BS11}.
The lower bound below on \(\varepsilon(K_X+L;x)\)
is sharp by \cite[Examples 4.3 and 4.4]{BS11}.
\begin{customthm}{B}\label{cor:seshlowerboundadj}
  Let \(X\) be a smooth projective surface over an algebraically closed field and
  let \(L\) be a nef divisor on \(X\) such that \(K_X+L\) is ample.
  Then,
  \[
    \varepsilon(K_X+L;x) \ge \frac{1}{2}
  \]
  for every closed point \(x \in X\).
  In particular, for any ample divisor \(L\) on \(X\), we have
  \[
    \varepsilon(K_X+4L;x) \ge \frac{1}{2}.
  \]
\end{customthm}
\begin{proof}
  The first statement follows from Theorem \ref{thm:bslowerbound}.
  The second statement follows from
  Mori's cone theorem \cite[Theorem
  1.5]{Mor82} (see \cite[Theorem 3.1\((i)\)]{MT81}), which says that if \(L\) is
  an ample divisor on a smooth projective variety \(X\) of dimension \(n\)
  over an algebraically closed field, then \(K_X+(n+2)L\) is
  ample.\footnote{According to
  \cite[p.\ 231n]{KM83}, this consequence of
  Mori's cone theorem was pointed out to Matsusaka by Mabuchi
  and Tsunoda during a lecture series given by Matsusaka at Osaka University in
  1981.
  The reference \cite{MT81} is Mabuchi and Tsunoda's account of these
  lectures.}
\end{proof}

\subsection{Lower bounds in all dimensions over the complex numbers}
Let \(X\) be a smooth projective variety of dimension \(n\)
and let \(L\) be a nef divisor on \(X\)
such that \(K_X+L\) is ample.
In \cite[Theorem 3.2 and Remark 3.3]{BS11}, Bauer and Szemberg show that 
\[
  \varepsilon(K_X+L;x) \ge
  \max\Biggl\{\frac{2}{n^2+n+4},
  \frac{1}{\bigl(e+\frac{1}{2}\bigr) n^{\frac{4}{3}}+
  \frac{1}{2} n^{\frac{2}{3}}+2}\Biggr\}
\]
using the results toward Fujita's freeness conjecture due to Angehrn and Siu
\cite{AS95} and to Heier \cite{Hei02}.
We use the known cases of Fujita's conjecture
\cite{Rei88,EL93a,Fuj93,Kaw97,YZ20}
and recent progress due to Ghidelli and Lacini \cite{GL24}
to improve this bound.
\begin{customthm}{C}\label{thm:lowerboundsoverc}
  Let \(X\) be a smooth complex projective variety of dimension \(n\).
  Let \(L\) be a nef divisor on \(X\) such that \(K_X+L\) is ample.
  For every closed point \(x \in X\), we have
  \[
    \varepsilon(K_X+L;x) \ge \begin{cases}
      \hfil \dfrac{1}{n+2} & \text{when}\ n \le 5,\\[1em]
      \hfil \dfrac{1}{9} & \text{when}\ n = 6,\\[1em]
      \dfrac{1}{\Bigl\lfloor
      n\Bigl(\log\bigl(\log(n)\bigr)+2.34\Bigr)
      \Bigr\rfloor + 2}
      & \text{for arbitrary}\ n.
    \end{cases}
  \]
\end{customthm}
\begin{proof}
  We claim that
  \[
    m(K_X+L) = K_X + (m-1)(K_X+L) + L
  \]
  is free whenever
  \begin{equation}\label{eq:ineqform}
    m \ge \begin{cases}
      \hfil n+2 & \text{when}\ n \le 5,\\
      \hfil 9 & \text{when}\ n = 6,\\
      \Bigl\lfloor
      n\Bigl(\log\bigl(\log(n)\bigr)+2.34\Bigr)
      \Bigr\rfloor + 2
      & \text{for arbitrary}\ n.
    \end{cases}
  \end{equation}
  Set \(A_m \coloneqq (m-1)(K_X+L) + L\), which is ample for every \(m \ge 2\).
  We claim that \(K_X+A_m = m(K_X+L)\) is free for all \(m\) satisfying
  the inequality \eqref{eq:ineqform}.
  \par We first consider the case when \(n \le 5\).
  We have
  \[
    A_m^d \cdot Z
    \ge (m-1)^d \bigl((K_X+L)^d \cdot Z\bigr)
    \ge (n+1)^d
  \]
  for every subvariety \(Z \subseteq X\) of dimension \(d\).
  By \cite[Theorem 1\((i)\)]{Rei88} (for \(n = 2\)), \cite[Theorems 3.1 and
  4.1]{Kaw97} (for \(n \in \{3,4\}\)), and \cite[Main Theorem]{YZ20} (for \(n =
  5\)), we conclude that \(K_X+A_m = m(K_X+L)\) is free for every \(m \ge n+2\).
  \par Next, we consider the case when \(n \ge 6\).
  We have
  \[
    \biggl( \frac{1}{m-1} A_m \biggr)^d \cdot Z \ge (K_X+L)^d \cdot Z \ge 1
  \]
  for every subvariety \(Z \subseteq X\) of dimension \(d\).
  By \cite[Lemma 5.4, Theorem 5.5, and Proof of Theorem 5.5]{GL24},
  we see that \(K_X+A_m\) is free when \(m \ge 9\) and \(n = 6\).
  By \cite[Theorem 4.9, Theorem 5.1, and Proof of Theorem 5.2]{GL24}, we see
  that \(K_X+A_m\) is free when
  \[
    m > n\Bigl(\log\bigl(\log(n)\bigr)+2.34\Bigr) + 1
  \]
  and \(n\) is arbitrary, and in particular, when
  \[
    m \ge \Bigl\lfloor n\Bigl(\log\bigl(\log(n)\bigr)+2.34\Bigr) \Bigr\rfloor +
    2.
  \]
  \par Finally, 
  since Seshadri constants of ample free Cartier divisors are always at least
  \(1\) \cite[Example 5.1.18]{Laz04a}, we see that
  \[
    m \cdot \varepsilon(K_X+L;x) = \varepsilon\bigl(m(K_X+L);x\bigr) \ge 1
  \]
  for all \(m\) satisfying
  the inequality \eqref{eq:ineqform}.
  Dividing by \(m\), we are done.
\end{proof}

\section{Effective global generation, very
ampleness, and jet
ampleness}\label{sect:effective}
In this section, we prove Theorem \ref{thm:smooths0ljetample}.
We also prove prove Conjecture \ref{conj:fujitavariant} for smooth complex
projective varieties (Theorem \ref{thm:fujitavariantcomplex}).
\subsection{Effective bounds for smooth surfaces in arbitrary characteristic}
We are now ready to prove the following stronger version of Theorem
\ref{thm:smooths0ljetample}.
\begin{theorem}\label{thm:smooths0ljetamplewithan}
  Let \(X\) be a smooth projective surface over an algebraically closed field
  \(k\).
  Fix an integer \(\ell \ge 0\).
  For all ample divisors \(L\) and \(A\), the linear system
  \[
    \bigl\lvert K_X+2(2+\ell)(K_X+4L)+A\bigr\rvert
    = \bigl\lvert K_X+(4+2\ell)K_X+8(2+\ell)L+A\bigr\rvert
  \]
  is \(\ell\)-jet ample.
  If \(\Char(k) = p > 0\), then in fact, the linear system
  \[
    \biggl\lvert S^0\Bigl(X,\omega_X \otimes
    \cO_X\bigl(2(2+\ell)(K_X+4L)+A\bigr)\Bigr) \biggr\rvert
  \]
  is \(\ell\)-jet ample.
\end{theorem}
\begin{proof}
  Thus, Theorem \ref{cor:seshlowerboundadj} implies \(\varepsilon(K_X+4L;x)
  \ge 1/2\) for every \(x \in X\).
  We then see that
  \[
    \varepsilon\bigl( 2(2+\ell)(K_X+4L) + A;x\bigr)
    > 2(2+\ell)\,\varepsilon(K_X+4L;x) \ge 2+\ell.
  \]
  The conclusion now follows from Theorem
  \ref{thm:introseshjetample}.
\end{proof}
\begin{customthm}{D}\label{thm:smooths0ljetample}
  Let \(X\) be a smooth projective surface over an algebraically closed field
  \(k\) and let \(L\) be an ample divisor on \(X\).
  For every integer \(\ell \ge 0\), the linear system
  \[
    \bigl\lvert K_X+(4+2\ell)K_X+(17+8\ell)L \bigr\rvert
  \]
  is \(\ell\)-jet ample.
  If \(\Char(k) = p > 0\), then in fact, the linear system
  \[
    \biggl\lvert S^0\Bigl(X,\omega_X \otimes_{\cO_X}
    \cO_X\bigl((4+2\ell)K_X+(17+8\ell)L\bigr)\Bigr) \biggr\rvert
  \]
  is \(\ell\)-jet ample.
\end{customthm}
\begin{proof}
  Set \(A = L\) in Theorem \ref{thm:smooths0ljetamplewithan}.
\end{proof}
\subsection{Effective bounds for smooth complex projective varieties}
Finally, we prove Conjecture \ref{conj:fujitavariant} over the complex numbers.
\begin{customthm}{E}\label{thm:fujitavariantcomplex}
  Let \(X\) be a smooth complex projective variety of dimension \(n\).
  Fix an integer \(\ell \ge 0\).
  For all ample divisors \(L\) and \(A\) and every integer
  \[
    m \ge \begin{cases}
      \hfil n+2 & \text{when}\ n \le 5,\\
      \hfil 8 & \text{when}\ n = 6,\\
      n\Bigl(\log\bigl(\log(n)\bigr)+2.34\Bigr)
      & \text{for arbitrary}\ n,
    \end{cases}
  \]
  the linear system
  \[
    \bigl\lvert K_X+(n+\ell)(K_X+mL)+A\bigr\rvert
  \]
  is \(\ell\)-jet ample.
  In particular, the linear system
  \[
    \biggl\lvert
    K_X+(n+\ell)\Bigl(K_X+n\Bigl(\log\bigl(\log(n)\bigr)+2.34\Bigr)L\Bigr)+L
    \biggr\rvert
  \]
  is \(\ell\)-jet ample.
\end{customthm}
\begin{proof}
  As before, \(K_X+mL\) is ample by Mori's cone theorem \cite[Theorem
  1.5]{Mor82} (see \cite[Theorem 3.1\((i)\)]{MT81}).
  By the known cases of Fujita's freeness conjecture
  \citeleft\citen{Rei88}\citepunct \citen{EL93a}\citepunct \citen{Fuj93}\citepunct
  \citen{Kaw97}\citepunct \citen{YZ20}\citeright\ and by \cite[Theorems 1.1 and
  1.2]{GL24}, we know that \(K_X+mL\) is free.
  By \cite[Example 5.1.18]{Laz04a}, we know that \(\varepsilon(K_X+mL;x) \ge 1\)
  for all \(x \in X\).
  We then see that
  \[
    \varepsilon\bigl((n+\ell)(K_X+mL)+A;x\bigr) \ge
    (n+\ell)\cdot\varepsilon(K_X+mL;x) + \varepsilon(A;x) > n+\ell.
  \]
  The conclusion now follows from Theorem
  \ref{thm:introseshjetample}.
\end{proof}

\end{document}